\documentclass[10 pt]{article}        	
\usepackage{amsmath, amsthm, amssymb, amsfonts}	
\newtheorem{lemma}{Lemma}
\usepackage{thmtools}
\usepackage[linesnumbered,ruled,vlined]{algorithm2e}
\usepackage{graphicx}
\usepackage{setspace}
\usepackage{geometry}
\usepackage{float}
\usepackage{booktabs}
\usepackage{hyperref}
\usepackage{tikz}
\usetikzlibrary{decorations.markings, arrows.meta}
\usepackage{algpseudocode}
\usepackage[utf8]{inputenc}
\usepackage[english]{babel}
\usepackage{framed}
\usepackage[dvipsnames]{xcolor}
\usepackage{tcolorbox}
\usepackage{caption}
\usepackage{subfigure}
\newtheorem{theorem}{Theorem}

\usepackage{multirow}
\usepackage{booktabs}

\makeatother

\usepackage{babel}

\makeatletter
\newcommand{\eqn}[0]{\begin{array}{rcl}}
\newcommand{\eqnend}[0]{\end{array} }  	

\title{Direct Sampling Methods for Inverse Interface Problems}
\author{Kazufumi Ito\thanks{Department of Mathematics and Center for Research in Scientific Computation, North Carolina State University, Raleigh, North Carolina (kito@ncsu.edu).}
\and Tong Wu\thanks{Department of Mathematics, The Chinese University of Hong Kong, Shatin, N.T., Hong Kong (twu@math.cuhk.edu.hk).}
\and Jun Zou\thanks{Department of Mathematics, The Chinese University of Hong Kong, Shatin, N.T., Hong Kong. The work of J. Zou was substantially supported by the Hong Kong RGC General Research Fund (projects 14310324 and 14306623) and NSFC / Hong Kong RGC Joint Research Scheme 2022/23 (project N CUHK465/22). (zou@math.cuhk.edu.hk).}}
\date{}
\begin{document}
\maketitle
\begin{abstract}
This work investigates two types of inverse interface problems in scenarios where only very limited Cauchy data is available. These problems are associated with the Laplace equation featuring a Robin-type flux jump across an internal interface. The first problem focuses on reconstructing the location of cracks along a known interface using Cauchy data measured on the outer boundary. The second problem involves determining the location of an unknown interface based on Cauchy data from the outer boundary. To address these challenges, we adopt an efficient Direct Sampling Method (DSM) and introduce innovative enhancements to the boundary conditions in the reference system, thereby maximizing the utility of the available Cauchy data. Additionally, we propose a novel refinement to further improve the robustness of the DSM against noise. 

We provide a detailed exposition of the general principles underlying the DSM and systematically present its computational implementation steps. Through detailed Fourier analysis and computations, we illustrate the theoretical background of the DSM as well as the effectiveness of our refinement approach. A series of numerical experiments demonstrates that our method yields highly satisfactory results, even when processing incomplete and noisy Cauchy data on the outer boundary. We introduce quantitative metrics, such as Mean Localization Error (MLE) and Contrast-to-Noise Ratio (CNR), to rigorously evaluate the performance of our method. These findings underscore the exceptional effectiveness and broad applicability of the proposed approach.
\end{abstract}
\section{Introduction}
\subsection{Background of inverse interface problems and direct sampling methods}
Interface problems are ubiquitous in various scientific and engineering disciplines. For instance, when a physical domain is composed of two different materials or media, an interface naturally emerges between them, leading to an interface problem. In the context of Partial Differential Equations (PDEs), numerical solutions to interface problems often require partitioning the computational domain into distinct regions or subdomains. These regions may exhibit different material properties or physical conditions and are typically separated by interfaces or boundaries. For example, in electrical conduction problems, resistivity varies across different materials, resulting in irregularities such as discontinuous parameters, jump conditions, and singular source terms in the governing differential equations at the interfaces or boundaries. Such complexities pose significant challenges in solving interface problems.

A flux jump condition describes a discontinuous change in the flux (or flow) of a quantity across an interface. This phenomenon typically arises when there is a sharp transition or inhomogeneity in the physical properties at the interface, and is particularly relevant in problems involving composite materials, phase boundaries, or fractured media.

In this paper, we focus on inverse interface problems in elliptic equations. While addressing two physically distinct scenarios, we unify them mathematically as the problem of recovering the spatial support of a singular perturbation. The first part of our work addresses the identification of the flux jump location at a known interior interface with a flux boundary condition, where the interface represents a surface with cracks, defects, or damages. In the second part, we identify the location of an entirely unknown interface. Both problems are formulated and solved using limited Cauchy measurements on the outer boundary. These problems find applications in diverse fields, including electrical conductivity problems \cite{torquato1985effective}, multi-material compressible flow simulations \cite{kucharik2010comparative}, medical diffusion MRI
imaging \cite{o2004interface}, and geological surface reconstruction \cite{frank20073d}.

Interface problems have been extensively studied through both analytical and computational approaches. For forward problems involving discontinuous coefficients at fixed or moving interfaces, the immersed interface methods (IIM) \cite{lito2006immersed, li2003overview} and certain optimization techniques \cite{li1998fast} have emerged as  accurate and efficient approaches for problems defined on irregular domains. In inverse interface problems, classical approaches often rely on iterative optimization techniques such as the Newton-type methods \cite{zhang2013newton} and the level-set method \cite{kito2001level}. However, these methods face challenges in handling noisy data and require relatively accurate initial guesses, which are crucial for their convergence. Apart from these iterative approaches, the class of sampling methods is also powerful for various inverse problems, including linear sampling methods \cite{cakoni2003linear, arens2001linear} and orthogonality sampling methods \cite{potthast2010study, griesmaier2011multi}. These sampling methods define an index function on the domain of interest to indicate the inhomogeneous part, yielding satisfactory results. As direct methods, they offer significant computational efficiency and demonstrate promising potential for inverse problems.

We employ the Direct Sampling Method (DSM) to solve these inverse problems. Specifically, we compute the probing index $I(x)$ of the interface, defined through a weighted inner product between probing functions and boundary measurement data (given as Dirichlet-type measurements). The DSM is highly efficient and provides accurate solutions to our inverse problems. While the DSM has been applied to inverse medium problems \cite{chow2015direct, ito2012direct, ito2013direct, chow2014direct, chow2021direct, chow2021direct_simul, chow2018time, li2012direct, ito2025iterative}, our work presents a novel application for determining the interior interface condition. We also evaluate the index function $I(x), x\in\Gamma$ by the adjoint method, which results in a PDE-based, highly direct approach.

Another main contribution is the first quantitative proof, to the best of our knowledge, of the almost orthogonal property for a DSM kernel arising from an inverse interface problem. Earlier DSM papers use almost orthogonality to explain why a sampling index is large near the target and small away from it; related notions, such as mutually almost orthogonal probing functions, have also been developed \cite{chow2021direct_simul}. These arguments do not cover the kernel used here, because both Cauchy data are built into the reference problem through a Robin condition and the probing functions are rescaled in the Sobolev inner product. In the disk case, we derive the Fourier representation of the normalized kernel and prove explicit radial monotonicity estimates: the kernel is maximized at the probing point and decreases quantitatively as the radial variable moves away. This turns the almost orthogonal property from a numerical observation into a proved localization result for the proposed interface DSM.

In this work, we develop a direct sampling method to solve a two-dimensional inverse interface problem where the flux jump across the interface is proportional to the flow. Specifically, we consider the following interface Laplace problem in a domain $\Omega\subset\mathbb{R}^{2}$:
\begin{equation}
     \left\{
\begin{aligned}
-\Delta u=0 \quad \text{in} \quad &\Omega\setminus\Gamma, \\
\left[\frac{\partial u}{\partial \nu}\right]=\gamma u\quad \text{on} \quad &\Gamma,\\
[u] = 0 \quad \text{on} \quad &\Gamma,\\
\end{aligned}
\right. \label{con:model_problem}
 \end{equation}
 where $\Gamma\subset\Omega$ is a one-dimensional interface. The jump condition is defined as
\begin{equation}
\left[ \frac{\partial u(\hat{x})}{\partial \nu} \right] := \left( \lim_{x \to \hat{x}, x \in \Omega_1} \nabla u(x) - \lim_{x \to \hat{x}, x \in \Omega_2} \nabla u(x) \right) \cdot \nu \text{ for } \hat{x} \in \Gamma,
\end{equation}
where $\nu$ denotes the unit normal vector on $\Gamma$. A pair of Cauchy data $(f, g)$ is measured on the outer boundary $\partial\Omega$ as follows:
 \begin{equation}
     f = u\bigg|_{L}\quad \text{and}\quad g = \frac{\partial u}{\partial n}\bigg|_{\partial\Omega},\label{Cauchy_def}
 \end{equation}
where $n$ represents the outward unit normal vector on $\partial\Omega$, and $L\subset\partial\Omega$ is the portion of the outer boundary where the measurement data $f$ is available.
 
 Our work addresses two distinct inverse problems. In Problem I, we consider a known interface $\Gamma$ where the jump coefficient $\gamma(x)$ consists of a known background value $\gamma_{0}$ and a sparse perturbation. Mathematically, this sparsity is characterized by $D = \text{supp}(\gamma-\gamma_{0})$ being a small subset of $\Gamma$. 
 The objective is to construct
a probing index whose large values indicate the defect support
\(D=\operatorname{supp}(\gamma-\gamma_0)\subset\Gamma\) from a single pair of Cauchy data \((f,g)\).
 In Problem II, the interface $\Gamma$ and the jump coefficient $\gamma$ are both unknown, and we aim to reconstruct
a probing index that highlights the location of the unknown interface \(\Gamma\) from a single pair of Cauchy data \((f,g)\). Since the response of
a defect may depend on the chosen boundary excitation, a single pair of data
does not necessarily illuminate all connected components of \(D\). When
several Cauchy pairs are available, the corresponding single pair probing indices can
be combined to obtain a more complete reconstruction. The analysis and results for the general case that $[u]\neq 0$ will be addressed in future work.

The remainder of the paper is structured as follows. In Section 1.2, we state a generic well-posedness result for the forward interface problem, with the proof deferred to the Appendix. Section 2 introduces the general philosophy of the DSM and details the construction of probing functions for both problems. In this section, we propose an innovative approach by utilizing the two Cauchy data together, reformulated as a Robin boundary condition, to solve the forward reference problem $u_0$, while retaining the Dirichlet boundary condition for measurement data fitting. It significantly enhances the robustness of the results in the presence of noise. Section 3 derives the Fourier representation of the DSM kernel and states its quantitative almost orthogonal property, whose rigorous radial monotonicity proof is provided in the Appendix. Furthermore, it employs a frequency-domain analysis to establish a systematic criterion for selecting the Robin parameter $\alpha$. In Section 4, we present a novel refinement approach to sharpen the performance of the DSM. Section 5 showcases our numerical results. Finally, Section 6 summarizes the main findings and discusses potential future research directions.
 \subsection{Well-posedness of the interface problem}
 In this section, we state the well-posedness result for the forward problem that will be used throughout the paper. It is classical that the problem is uniquely solvable if $\gamma<0$ almost everywhere on $\Gamma$, due to the coercivity of the associated bilinear form. The following lemma records a generic solvability statement for the more delicate case in which $\gamma$ may take positive values.
\begin{figure}[htbp]
\centering
\begin{minipage}[t]{0.47\linewidth}
\centering
\begin{tikzpicture}[scale=0.6]
    \draw[thick] (0,0) ellipse [x radius=3cm, y radius=2cm];
    \node[font=\small, outer sep=2pt] at (80:3.2cm and 2.2cm) {$\partial\Omega$};
    
    \draw[thick] (0,0) ellipse [x radius=1.5cm, y radius=1cm];
    \node[font=\small, outer sep=2pt] at (100:1.7cm and 1.2cm) {$\Gamma$};
    
    \node at (0,0) {$\Omega_1$};
    \node at (2,0.5) {$\Omega_2$};
\end{tikzpicture}
\caption{Domain partition diagram}
\label{fig:domain}
\end{minipage}\hfill
\begin{minipage}[t]{0.47\linewidth}
\centering
\begin{tikzpicture}[scale=0.65,
    normal/.style={-{Stealth[length=3mm]}, thick, red}]
    \def\innerRx{1.5}
    \def\innerRy{1}
    \def\outerRx{3}
    \def\outerRy{2}

    \draw[thick] (180:\outerRx cm and \outerRy cm) arc (180:360:\outerRx cm and \outerRy cm)
        node[pos=0.15, left] {$\partial\Omega$};
    \draw[thick] (0:\outerRx cm and \outerRy cm) arc (0:180:\outerRx cm and \outerRy cm);

    \draw[thick] (180:\innerRx cm and \innerRy cm) arc (180:360:\innerRx cm and \innerRy cm)
        node[pos=0.5, below] {$\Gamma$};
    \draw[thick, dashed] (0:\innerRx cm and \innerRy cm) arc (0:180:\innerRx cm and \innerRy cm)
        node[pos=0.2, right] {$\widetilde{\Gamma}$};

    \draw[normal] (30:\outerRx cm and \outerRy cm) -- ++(30:0.8cm)
        node[above] {$\mathbf{n}$};
    \draw[normal] (210:\innerRx cm and \innerRy cm) -- ++(210:0.8cm)
        node[below] {$\boldsymbol{\nu}$};

    \node at (0,0) {$\Omega_1$};
    \node at (2,1) {$\Omega_2$};
\end{tikzpicture}
\caption{$\widetilde{\Gamma}$ is a closed surface extension of $\Gamma$.}
\label{fig:extend_surface}
\end{minipage}
\end{figure}
\begin{lemma} \label{lemma-wellpo}
    Let \(\Omega\subset\mathbb R^2\) be a bounded connected Lipschitz domain, and let \(\Gamma \Subset \Omega\) be a closed Lipschitz interface of positive
length partitioning \(\Omega\) into \(\Omega_{1} \sqcup \Omega_{2}\) (see Figure \ref{fig:domain}). Assume that the boundary data $g \in L^2(\partial\Omega)$ and the jump coefficient is parameterized as $\gamma(x) = \lambda \rho(x)$, where \(\rho\in L^\infty(\Gamma;\mathbb R)\) is a given normalized profile function with $\|\rho\|_{L^\infty(\Gamma)} = 1$ and $\lambda \in \mathbb{R}$ is a scaling parameter. Then, there exists a discrete set $\Sigma \subset \mathbb{R}$ with no finite accumulation points, such that for any $\lambda \in \mathbb{R} \setminus \Sigma$, the interface problem \eqref{con:model_problem}, supplemented with the Neumann boundary condition \(\partial u/\partial n=g\) on \(\partial\Omega\), admits a unique weak solution \(u\in H^1(\Omega)\). More precisely, the weak solution is understood as a function
\(u\in H^1(\Omega)\) satisfying
\begin{equation}
    \int_\Omega \nabla u\cdot\nabla v\,dx
    -
    \lambda\int_\Gamma \rho u v\,ds
    =
    \int_{\partial\Omega} g v\,ds,
    \qquad \forall v\in H^1(\Omega).
    \label{eq:weak-forward-interface}
\end{equation}
\end{lemma}
The proof, based on compactness and the analytic Fredholm theorem, is deferred to Appendix \ref{app:lemma-wellpo-proof}.
\section{Direct sampling methods for inverse interface problems}
This section presents the Direct Sampling Method (DSM) for solving the two inverse interface problems introduced in Section 1. Before introducing the DSM construction, we establish a unified mathematical framework that characterizes the essence of their inverse identification. The key observation is that both problems can be viewed as the detection of a singular perturbation relative to a properly chosen reference or background model. This viewpoint separates the known part of the physics, encoded in a background operator, from the unknown object to be recovered, encoded in the support of a perturbation term. More precisely, we denote the corresponding background operator by $\mathcal{L}_0$ and write, at a formal distributional level,
$$(\mathcal{L}_0 + \mathcal{S})u = 0 \quad \text{in } \mathcal{D}'(\Omega),
$$
where $\mathcal{S}$ represents the unknown singular perturbation relative to the chosen background operator and is supported on a subset $\Sigma \subset \Omega$. The primary objective is to determine the spatial support $\Sigma = \text{supp}(\mathcal{S})$ using a limited set of Cauchy data pairs $(f, g)$. In this way, the known-interface coefficient problem and the unknown-interface geometry problem are placed in the same form: the DSM only has to identify where the discrepancy between the true model and the reference model is located. Within this unified framework, we consider the following two cases.
\begin{enumerate}
\item \textbf{Problem I (Local Anomaly Detection):} In the interface model \eqref{con:model_problem}, the interface $\Gamma\subset\Omega$ is known a priori and may be non-closed, while the Robin jump coefficient $\gamma\in L^\infty(\Gamma)$ is unknown. We assume that $\gamma$ has a sparse deviation from a known background value $\gamma_0$. The background operator $\mathcal{L}_0$ includes the known background jump condition $[\partial_\nu u]=\gamma_0u$ on $\Gamma$, and the perturbation operator $\mathcal{S}$ represents only the deviation of $\gamma$ from $\gamma_0$, acting distributionally as $\mathcal{S}u = -(\gamma - \gamma_0)u\delta_{\Gamma}$, where $\delta_{\Gamma}$ is the Dirac measure on $\Gamma$. Thus the inverse task is to locate the defect set $\Sigma = D = \text{supp}(\gamma - \gamma_0) \subseteq \Gamma$, rather than to reconstruct the whole known carrier $\Gamma$.
\item \textbf{Problem II (Full Topology Reconstruction):} In the same interface model \eqref{con:model_problem}, the interface $\Gamma$ may be non-closed and is itself unknown, as is the coefficient $\gamma\in L^\infty(\Gamma)$. In this case the background operator is the interface-free Laplacian, $\mathcal{L}_0=-\Delta$, and $\mathcal{S}$ represents the entire interface jump condition, acting as $\mathcal{S}u = -\gamma u \delta_{\Gamma}$. The target support is therefore the unknown interface itself, $\Sigma = \Gamma \subset \Omega$, and the sampling points must be searched over the interior of the domain.
\end{enumerate}
When \(\Gamma\) is not closed, we use the following convention. Choose a closed Lipschitz extension \(\widetilde\Gamma\Subset\Omega\) containing \(\Gamma\) and partitioning \(\Omega\), and extend the jump coefficient by \(\widetilde\gamma=\gamma\) on \(\Gamma\) and \(\widetilde\gamma=0\) on \(\widetilde\Gamma\setminus\Gamma\) (and similarly for the background coefficient when used). This zero extension leaves the weak formulation unchanged, since the added part contributes no interface term; equivalently, the artificial part carries zero flux jump. Therefore, the closed-interface well-posedness result in Lemma \ref{lemma-wellpo} applies to this extended formulation.
Under this formulation, the DSM constructs an index function that is expected
to take relatively large values near \(\Sigma\) by exploiting the discrepancy between the observed data and a reference ``unperturbed'' state $u_0$. The specific implementations are detailed in the following subsections.
\subsection{The DSM Approach}
The Direct Sampling Method (DSM) is a mathematical approach designed to efficiently and directly recover the location of inhomogeneities in a given equation. These inhomogeneities may appear as discontinuities in a coefficient or source term, or as the presence of an interface within the domain. The method operates by considering a suitably chosen reference or background model of the original equation, in which the unknown singular
perturbation is removed. Because the problem is overdetermined, the reference equation cannot simultaneously satisfy both sets of Cauchy boundary data. Therefore, the DSM enforces the reference equation to satisfy one prescribed boundary condition, which allows us to capture the discrepancy between the second boundary measurement and the solution to the homogeneous equation.

To quantify this discrepancy, we employ Green's representation theorem to express the difference between the observed data and the homogeneous solution in terms of Green's functions. Subsequently, we construct a family of probing functions that is almost orthogonal to the family of Green's functions under a specific inner product with proper scaling. The index function is then defined as the scaled inner product of the probing functions with the boundary data discrepancy. This almost orthogonal property ensures that the
index function attains significantly larger values near the inhomogeneities, effectively identifying
their locations.

We now adapt this DSM framework to our specific interface problems. Recall that
the boundary flux \(g\) is prescribed on the whole outer boundary
\(\partial\Omega\), while the Dirichlet measurement \(f\) is available only on
a subset \(L\subset\partial\Omega\). We define a reference state \(u_0\) through a mixed
Neumann--Robin background system. On \(\partial\Omega\setminus L\), where there is no Dirichlet measurement, we
impose the measured flux \(g\). On the measurement part \(L\), we combine the
two Cauchy data through the Robin condition
\[
    \frac{\partial u_0}{\partial n}+\alpha u_0=g+\alpha f,
    \qquad \alpha>0.
\]
The parameter \(\alpha\) controls the relative weight assigned to the
Dirichlet measurement. In particular, larger values of \(\alpha\) enforce the
trace of \(u_0\) to be closer to the measured data \(f\), while smaller values
place more emphasis on the Neumann datum.

For Problem I, the known background interface condition is retained in the
reference system. Thus \(u_0\) is defined as the solution of \eqref{step1_ref_sys} below.
For Problem II, where the interface itself is unknown, the reference state is
taken to be the solution of the interface-free background problem \eqref{step2_ref_sys} below.

\begin{minipage}{0.48\textwidth}
\begin{equation}
\left\{
\begin{aligned}
-\Delta u_{0}=0 \quad &\text{in } \Omega\setminus\Gamma, \\
\left[\frac{\partial u_{0}}{\partial \nu}\right]=\gamma_{0} u_{0}
    \quad &\text{on } \Gamma,\\
[u_{0}] = 0 \quad &\text{on } \Gamma,\\
\frac{\partial u_{0}}{\partial n}=g
    \quad &\text{on } \partial\Omega\setminus L,\\
\frac{\partial u_{0}}{\partial n}+\alpha u_{0}=g+\alpha f
    \quad &\text{on } L.
\end{aligned}
\right.
\label{step1_ref_sys}
\end{equation}
\end{minipage}%
\hfill
\begin{minipage}{0.48\textwidth}
\begin{equation}
\left\{
\begin{aligned}
-\Delta u_{0}=0 \quad &\text{in } \Omega, \\
\frac{\partial u_{0}}{\partial n}=g
    \quad &\text{on } \partial\Omega\setminus L,\\
\frac{\partial u_{0}}{\partial n}+\alpha u_{0}=g+\alpha f
    \quad &\text{on } L.
\end{aligned}
\right.
\label{step2_ref_sys}
\end{equation}
\end{minipage}

Throughout the paper, we assume that the corresponding reference systems are
uniquely solvable for the chosen parameters. In the interface case, this
amounts to excluding the same type of exceptional resonant values as in
Lemma \ref{lemma-wellpo}.

After obtaining the reference solution, we extract its trace $f_0 = u_0|_L$ and exclusively use the Dirichlet discrepancy $f - f_0$ for data fitting. We construct a family of probing functions $\eta_{x}:L\rightarrow\mathbb{R}$ such that the inner product between $f-f_{0}$ and $\eta_{x}$ serves as an effective probing index. For Problem I, this index locates the defect $D = \text{supp}(\gamma - \gamma_0) \subset \Gamma$. For Problem II, it identifies the entirely unknown interface $\Gamma$. Specifically, we expect the index function
\begin{equation}
    I(x)=\frac{(f-f_{0},\eta_{x})}{S(x)}
\end{equation}
to correctly locate the inhomogeneities for both problems. Here, $(\cdot, \cdot)$ denotes a specific inner product and $S(x)$ is a scaling function, both of which will be appropriately chosen in the subsequent sections along with the probing functions $\eta_x$.

\subsection{Representations of boundary measurement}
In this subsection, we derive a representation formula for the boundary data discrepancy $f-f_0$ via the Green's function under noise-free assumptions. Let $G_{x}$ be the Green's function satisfying the following boundary value problem:
\begin{equation}
    \left\{
\begin{aligned}
-\Delta G_{x}=\delta_{x} \quad in \quad &\Omega,\\
\frac{\partial G_{x}}{\partial n} = 0  \quad on \quad & \partial\Omega\setminus L,\\
\frac{\partial G_{x}}{\partial n}+\alpha G_{x}=0 \quad on \quad & L.\\
\end{aligned}
\right. \label{Gx-def}
\end{equation}
It is well known that the Green's function $G_x$ possesses the symmetric property $G_{x}(y) = G_{y}(x)$ for all $x, y\in\Omega$.\\


To apply Green's second identity, we extend the interface $\Gamma$ to be a closed surface $\widetilde{\Gamma}\subset\Omega$ which partitions the domain such that $\Omega = \Omega_1 \cup \Omega_2$ and $\Gamma \subseteq \partial\Omega_1$ (as illustrated in Figure \ref{fig:extend_surface}). Define $\widetilde{\gamma}: \widetilde{\Gamma}\to\mathbb{R}$ by $\widetilde{\gamma}(x) = \gamma(x)$ on $\Gamma$, and $\widetilde{\gamma}(x) = 0$ otherwise. 
By applying Green's second identity on the partitioned subdomains $\Omega_1$ and $\Omega_2$, we have for $x\in\text{int}(\Omega_{1})\cup\text{int}(\Omega_{2})$,
$$
\begin{aligned}
(u-u_{0})(x) &= \int_{\Omega} (G_x \Delta(u-u_0) - (u-u_0)\Delta G_x) dy \\
&= \int_{\Omega_1}(G_x \Delta(u-u_0) - (u-u_0)\Delta G_x) dy + \int_{\Omega_2}(G_x \Delta(u-u_0) - (u-u_0)\Delta G_x) dy\\
&= \int_{\Gamma} \left( G_x \left[ \frac{\partial (u-u_0)}{\partial \nu} \right] - (u-u_0) \left[ \frac{\partial G_x}{\partial \nu} \right] \right) ds_y + \int_{\partial\Omega} \left( G_{x}\frac{\partial (u-u_0)}{\partial n}-(u-u_{0})\frac{\partial G_{x}}{\partial n} \right) ds_y.
\end{aligned}
$$
Since $u$, $u_0$, and the Green's function $G_x$ are continuous across the interface $\Gamma$, the jump term $\left[ \frac{\partial G_x}{\partial \nu} \right]$ vanishes. Utilizing the transmission conditions $\left[ \frac{\partial u}{\partial \nu} \right] = \gamma u$ and $\left[ \frac{\partial u_0}{\partial \nu} \right] = \gamma_0 u_0$, the integral over $\Gamma$ simplifies directly to the flux jump discrepancy, yielding:
$$(u-u_{0})(x) = \int_{\Gamma}(\gamma u-\gamma_{0}u_{0})G_{x} ds_y + \int_{\partial\Omega} \left( G_{x}\frac{\partial (u-u_0)}{\partial n}-(u-u_{0})\frac{\partial G_{x}}{\partial n} \right) ds_y.
$$
We evaluate the boundary integral over $\partial\Omega$ by splitting it into $\partial\Omega \setminus L$ and $L$. On $\partial\Omega \setminus L$, both $u$ and $u_0$ satisfy the Neumann condition $\frac{\partial u}{\partial n} = \frac{\partial u_0}{\partial n} = g$, which causes the integral on this portion to vanish. On the measurement boundary $L$, by substituting $\frac{\partial u}{\partial n} = g$, $u = f$, and $\frac{\partial G_x}{\partial n} = -\alpha G_x$, we obtain:
$$
\int_{L} \left( G_{x} \left( g - \frac{\partial u_0}{\partial n} \right) - (f-u_{0})(-\alpha G_{x}) \right) ds_y = \int_{L} G_{x} \left( g - \frac{\partial u_{0}}{\partial n} + \alpha(f - u_{0}) \right) ds_y.
$$
From the Robin boundary condition \eqref{step1_ref_sys} for the reference system on $L$, we have $\frac{\partial u_0}{\partial n} + \alpha u_0 = g + \alpha f$, which implies $g - \frac{\partial u_0}{\partial n} = -\alpha(f - u_0)$. Substituting this relationship into the equation above, the terms inside the parenthesis exactly cancel out. Thus, the entire boundary integral over $\partial\Omega$ is zero, and we arrive at:
$$
(u-u_{0})(x) = \int_{\Gamma}(\gamma(y) u(y)-\gamma_{0}u_{0}(y))G_{x}(y)ds_y.
$$
By the symmetry of Green's function $G_{x}$, we have for $x\in\text{int}(\Omega_{1})\cup\text{int}(\Omega_{2})$,
\begin{equation}
    (u-u_{0})(x) = \int_{\Gamma}(\gamma(y) u(y)-\gamma_{0}u_{0}(y))G_{y}(x)ds_y.
\end{equation}
Therefore, by the continuity of $u-u_{0}$, taking the limit as $x$ approaches the boundary $\partial\Omega$ yields the representation for the discrepancy data:
\begin{equation}
    (f-f_{0})(x) = \int_{\Gamma}(\gamma(y) u(y)-\gamma_{0}u_{0}(y))G_{y}(x)ds_y. \label{step1repre}
\end{equation}
For later use, we denote the effective density in the above representation by
\begin{equation}
    J_I(y):=\gamma(y)u(y)-\gamma_0(y)u_0(y).
    \label{eq:effective-density-I}
\end{equation}
Since \(\gamma=\gamma_0\) on \(\Gamma\setminus D\), this density can be
decomposed as
\begin{equation}
    J_I
    =
    (\gamma-\gamma_0)u+\gamma_0(u-u_0).
    \label{eq:effective-density-decomposition-I}
\end{equation}

Similarly, for Problem II, by applying the exact same procedure and noting that the reference system employs an identical boundary condition matching strategy, the boundary integrals vanish identically. For $x \in \text{int}(\Omega_1) \cup \text{int}(\Omega_2)$, we directly obtain:
\begin{equation}
    (u-u_{0})(x) = \int_{\Gamma}\gamma(y) u(y)G_{y}(x)ds_y.
\end{equation}
Again, by the continuity of $u - u_0$, we obtain the representation formula for $x \in \partial\Omega$:
\begin{equation}
    (f-f_{0})(x) = \int_{\Gamma}\gamma(y) u(y)G_{y}(x)ds_y. \label{step2repre}
\end{equation}
At this stage, we have derived the representation formulae for the boundary data discrepancy $f-f_0$ in both Problem I (\ref{step1repre}) and Problem II (\ref{step2repre}).

\subsection{Choices of inner products and probing functions}
This subsection introduces the specific inner products and probing functions, which are crucial for the successful implementation of the DSM. Following the notation in \cite{ito2025iterative}, let $\nabla_{\Gamma}$ denote the surface gradient operator, and let $-\Delta_{\Gamma}$ denote the Laplace-Beltrami operator (surface Laplacian) on a given curve. We define the inner product $(\cdot,\cdot)$ as the $H^{s}$ semi-inner product. For a smooth curve $L$ and $s\in\mathbb{N}$, this $H^{s}$ semi-inner product is expressed as:
\begin{equation}
(\phi, \psi)_s = \int_{L} (-\Delta_{\Gamma})^s \phi \cdot \psi d\ell, \label{semi-inner-def}
\end{equation}
where $d\ell$ denotes the arc length element. The subscript $s$ in (\ref{semi-inner-def}) denotes the regularity order of the Sobolev space $H^{s}(L)$. We modify the family of probing functions in \cite{chow2014direct} slightly to suit our boundary condition matching strategy. For a sampling point $x \in \text{int}(\Omega)$, let $\omega_x$ solve the following mixed boundary value problem:
\begin{equation}
    \left\{
\begin{aligned}
-\Delta \omega_{x}=\delta_{x} \quad in \quad &\Omega,\\
\frac{\partial \omega_{x}}{\partial n} = 0\quad on \quad & \partial\Omega\setminus L,\\
\omega_{x}=0 \quad on \quad & L,\\
\end{aligned}
\right. \label{probe-def}
\end{equation}
where $\delta_x$ is the Dirac delta distribution centered at $x$. Although Problem I and II represent distinct inverse problems, they share a fundamental mathematical structure as both involve recovering the spatial location of inhomogeneities within the Laplace equation framework. This unified characterization motivates our adoption of a common probing function approach. To maintain notational clarity, we employ identical symbols for the probing formulations in both problems. We define the probing function $\eta_x$ and the corresponding index function $I_s$ as:
\begin{equation}
\eta_x(y) = \frac{\partial \omega_x}{\partial n}(y) \quad \text{for } y \in L, \quad \text{and} \quad I_s(x) = \frac{(f - f_0, \eta_x)_s}{S(x)}.
\label{index-def}
\end{equation}
The spatial domain of interest differs between the two problems. For Problem I (defect detection), the probing is performed over the known interface, meaning $x \in \Gamma$. For Problem II (full interface reconstruction), the search domain extends to the entire interior region, meaning $x \in \text{int}(\Omega)$. The scaling function $S(x)$ and the measurement boundary $L \subset \partial\Omega$ remain problem-independent. This unified formulation effectively concentrates the sensitivity of the index function on subdomains containing the target inhomogeneities: Either the support $D = \text{supp}(\gamma - \gamma_0)$ in Problem I, or the unknown interface $\Gamma$ in Problem II.
\subsection{Kernels of the DSM}
This subsection introduces the kernel functions, which constitute the theoretical core of the DSM. Based on the representation formulae (\ref{step1repre}) and (\ref{step2repre}), the scattered potential $f - f_0$ can be approximated numerically via a quadrature rule as:
$$(f-f_{0})(x) \approx  \sum_{k} a_{k}G_{x_{k}}(x), \quad x \in \partial\Omega,
$$
where $x_k \in \Gamma$ are the discrete sampling points on the interface, and $a_k$ denote the corresponding quadrature weights that represent the localized intensity of the flux jump discrepancy. Consequently, the probing index $I_s(x)$ can be approximated as:
$$I_{s}(x) \approx \frac{\sum_{k} a_{k}(G_{x_{k}}, \eta_{x})_{s}}{S(x)}.
$$

The probing index $I_s(x)$ is expected to attain significantly larger values when the sampling point $x$ approaches the inhomogeneities and remains small when $x$ is distant from them. For Problem I, the inhomogeneities correspond to the defect support $D = \text{supp}(\gamma - \gamma_0) \subset \Gamma$, while for Problem II, they represent the unknown interface $\Gamma \subset \Omega$.


The interpretation of the quadrature weights depends on the problem under
consideration. For Problem I, the weights approximate the effective density
\(J_I\) in \eqref{eq:effective-density-I}; more precisely,
\[
    a_k \approx J_I(x_k)\Delta s_k,
    \qquad x_k\in\Gamma,
\]
where \(\Delta s_k\) denotes the local quadrature weight. By \eqref{eq:effective-density-decomposition-I}, this
density has two parts. The term \((\gamma-\gamma_0)u\) is the direct
contribution of the defect and is supported on
\(D=\operatorname{supp}(\gamma-\gamma_0)\). The second term
\(\gamma_0(u-u_0)\) is caused by the difference between the true solution and
the reference solution. Since a local defect may change the solution along the
whole interface, this term is not necessarily supported only on \(D\). Therefore, the DSM index for Problem I should be understood as detecting the
parts of \(D\) that produce a sufficiently strong signal in the effective
density \(J_I\). For a given boundary flux \(g\), some components of \(D\) may
generate strong peaks in the index, while others may be weak or partially
masked by the background term \(\gamma_0(u-u_0)\).

For Problem II, the corresponding effective density is
\[
    J_{II}(y):=\gamma(y)u(y),
\]
which is supported on the unknown interface \(\Gamma\). Thus, provided that
\(J_{II}\) does not vanish on the relevant interface component for the chosen
boundary excitation, the kernel localization property makes the probing index
large near that component.

This observation also explains the use of multiple boundary excitations in
the numerical experiments. A single Cauchy pair produces one legitimate
probing index, but different choices of \(g\) may illuminate different
components of \(D\) or \(\Gamma\). When several Cauchy pairs are available, we
therefore combine the independently computed single-pair indices, for example
through a composite indicator of the form
\[
    I_{\rm comp}(x)
    =
    \sum_{m=1}^{N_{\rm data}}
    \left|\widetilde I^{(m)}(x)\right|^2,
\]
where \(\widetilde I^{(m)}\) denotes the normalized probing index obtained
from the \(m\)-th Cauchy pair.

Therefore, the effectiveness of the DSM relies on the property that the probing functions $\{\eta_x\}_{x \in \Omega}$ are almost orthogonal to the family of Green's functions $\{G_y|_L\}_{y \in \Omega}$ in the sense of the $H^s$ semi-inner product. That is, the kernel function
\begin{equation}
    K(x,y) = \frac{(\eta_{x}, G_{y})_{s}}{S(x)} \label{kernel}
\end{equation}
should have large magnitude when $x \approx y$ and decay rapidly in magnitude as the distance between $x$ and $y$ increases. The analysis of this almost-orthogonal property provides the theoretical foundation for the DSM's performance, which we further verify through both rigorous proof for a circular domain and numerical visualization in Section 3.
\subsection{Alternative characterization of the index function}
In this subsection, we introduce the adjoint method, which is an efficient and accurate technique to make the calculation of index functions less computationally expensive. The implementation of the index function $I_s(x)$ defined in Section 2.4 requires solving the mixed boundary value problem \eqref{probe-def} for the probing function $\eta_x$ at every sampling point $x \in \Omega$. In practice, when the sampling grid is dense, this point-by-point approach becomes computationally expensive. To enhance the efficiency of the DSM, we introduce an alternative characterization of the index function using the adjoint method, which allows for the simultaneous evaluation of $I_s(x)$ across the entire domain.

For a fixed regularity order $s \ge 0$, we define the adjoint state $\lambda_s \in H^1(\Omega)$ as the solution to the following adjoint problem:
\begin{equation}
    \left\{
\begin{aligned}
-\Delta \lambda_{s}=0 \quad &\text{in} \quad \Omega,\\
\frac{\partial\lambda_{s}}{\partial n}=0 \quad &\text{on} \quad \partial\Omega\setminus L,\\
\lambda_{s}=\mathbb{P}(f-f_{0}) \quad &\text{on} \quad L.\\
\end{aligned}
\right. \label{adjoint-def}
\end{equation}
Here, $\mathbb{P}$ denotes a projection operator  with respect to the chosen inner product. In the case of the $H^s$ semi-inner product, it is given by
\begin{equation}
    \mathbb{P} = -\frac{(-\Delta_{\Gamma})^{s}}{\delta}, \label{proj-def}
\end{equation}
where $\delta$ is a constant parameter, and $-\Delta_{\Gamma}$ is the Laplace-Beltrami operator. Then, applying Green's second identity, we obtain:
$$\lambda_s(x) = \int_\Omega (\omega_x\Delta\lambda_s - \lambda_s\Delta\omega_x) dy = \int_{\partial\Omega} \left(\omega_x\frac{\partial\lambda_s}{\partial n} - \lambda_s\frac{\partial\omega_x}{\partial n}\right) ds_y = \frac{1}{\delta}\int_L (-\Delta_\Gamma)^s(f-f_0)\frac{\partial\omega_x}{\partial n} ds_y = \frac{(f-f_0,\eta_x)_s}{\delta}.$$
Therefore, by the definition of the index function in (\ref{index-def}), we have
\begin{equation}
    I_{s}(x)=\delta\frac{\lambda_{s}(x)}{S(x)}.
\end{equation}
The constant \(\delta>0\) is introduced only as a scaling parameter, so that
the adjoint formulation is consistent with the refinement scheme in Section 4.
For the unrefined DSM index, its value has no effect on the peak locations. In the following, for the $H^1$ semi-inner product, a detailed analysis and calculations regarding the kernel function (\ref{kernel}) with the specific scaling function $S(x)$ are presented in Section 3.

\section{Almost orthogonal property of kernel functions}

In this section, we provide a detailed theoretical and numerical analysis of the DSM kernel functions. We first verify the almost orthogonal property for the case where $L = \partial\Omega$, meaning that full boundary measurement data $f$ are available on the entire boundary. Since numerical experiments indicate that the Sobolev regularity order $s=1$ is the most suitable choice for the DSM to maintain both accuracy and robustness against noise, we present the corresponding analytical calculations and numerical verifications for $s=1$ in Section 3.1. Furthermore, in Section 3.2, we conduct a frequency-domain analysis to establish a systematic criterion for selecting the Robin parameter $\alpha$.
\subsection{Theoretical analysis on a circular domain}
We consider the unit disk $\Omega = B_1(0)$. Let $(r_x, \theta_x)$ and $(r_y, \theta_y)$ be the polar coordinates of $x$ and $y$, respectively. For a harmonic function $v$ in $B_1(0)$ satisfying:
\begin{equation}
    \left\{
    \begin{aligned}
    -\Delta v=0\quad\text{in}\quad B_{1}(0),\\
    \frac{\partial v}{\partial n} + \alpha v =h\quad\text{on}\quad\mathbb{S}^{1},
    \end{aligned}
    \right. \label{circle-harmonic}
\end{equation}
we have:
\begin{equation}
    v(x) = \int_{\Omega}G_{x}\Delta v - v\Delta G_{x} = \int_{\partial\Omega}G_{x}\frac{\partial v}{\partial n}-v\frac{\partial G_{x}}{\partial n} = \int_{\partial\Omega}G_{x}(h-\alpha v)-v(-\alpha G_{x}) = \int_{\partial\Omega}G_{x}h. \label{compare1}
\end{equation}
Since $\Delta v = 0$, by using the polar coordinates, we obtain:
\begin{equation}
    \frac{\partial^{2} v}{\partial r^{2}} + \frac{1}{r}\frac{\partial v}{\partial r}+\frac{1}{r^{2}}\frac{\partial^{2}v}{\partial\theta^{2}}=0.
\end{equation}
Suppose $v(r,\theta) = R(r)\Theta(\theta)$. By using the separation of variables, we obtain:
\begin{equation}
    r^{2}R^{''}(r) + rR^{'}(r)-cR(r) = 0,\qquad \Theta^{''}(\theta) + c\Theta(\theta) = 0.
\end{equation}
By the periodical property of $\Theta$, we find that
$\{e^{in\theta}\}_{n\in\mathbb{Z}}$ is the fundamental solution set of $\Theta(\theta)$, and $\{r^{n}\}_{n\in\mathbb{Z}}\cup\{\log r\}$ is the fundamental solution set of $R(r)$. Hence, the solution can be shown as the following series:
\begin{equation}
    v(r,\theta) = c+d\log(r) + \sum_{n=1}^{\infty}(d_{n}r^{n}+d_{-n}r^{-n})(c_{n}e^{in\theta}+c_{-n}e^{-in\theta}). \label{general-sol}
\end{equation}
Since $\{r=0\}$ (the origin) is contained in $B_1(0)$, singular terms at the origin must be omitted. Hence, $v$ can be represented by
\begin{equation}
    v(r,\theta) = \sum_{n\in\mathbb{Z}}a_{n}r^{|n|}e^{in\theta} \label{sol-0include}
\end{equation}
with coefficients $\{a_n\}_{n\in\mathbb{Z}}$. Using (\ref{circle-harmonic}) and (\ref{sol-0include}), we have
$$h(\theta) = \sum_{n\in\mathbb{Z}}(|n|+\alpha)a_{n}e^{in\theta}.
$$
Therefore, the Fourier coefficients of $h\in L^{2}(\mathbb{S}^{1})$ can be given as
$$
    (|n|+\alpha)a_{n} = \frac{1}{2\pi}\int_{0}^{2\pi}h(\theta)e^{-in\theta}d\theta,\quad\text{hence}\quad a_{n} = \frac{1}{2\pi(|n|+\alpha)}\int_{0}^{2\pi}h(\theta)e^{-in\theta}d\theta.
$$
Consequently, we have
$$
    v(r_{x}, \theta_{x}) = \frac{1}{2\pi}\sum_{n\in\mathbb{Z}}\frac{1}{|n|+\alpha}r_{x}^{|n|}e^{in\theta_{x}}\int_{0}^{2\pi}h(\theta)e^{-in\theta}d\theta,
$$
which is equivalent to
\begin{equation}
    v(x) = \int_{\partial\Omega}h(y)\sum_{n\in\mathbb{Z}}\frac{1}{2\pi(|n|+\alpha)}r_{x}^{|n|}e^{in(\theta_{x}-\theta_{y})}ds_{y}. \label{compare2}
\end{equation}
Comparing (\ref{compare1}) and (\ref{compare2}), the Green's function $G_{x}(y)$ for $y\in\mathbb{S}^{1}$ can be written as
\begin{equation}
    G_{x}(y) = \sum_{n\in\mathbb{Z}}\frac{1}{2\pi(|n|+\alpha)}r_{x}^{|n|}e^{-in\theta_{x}}e^{in\theta_{y}}. \label{Gx-Fourier}
\end{equation}
Using a similar technique, the probing function $\eta_{x}$ on $\mathbb{S}^{1}$ can be evaluated as
\begin{equation}
    \eta_{x}(y) = \frac{\partial w_{x}}{\partial n} = -\frac{1}{2\pi}\sum_{n\in\mathbb{Z}}r_{x}^{|n|}e^{-in\theta_{x}}e^{in\theta_{y}}. \label{etax-Fourier}
\end{equation}
Thus, the inner product of $\eta_{x}$ and $G_{z}$ is given by
\begin{equation}
    (\eta_{x}, G_{z})_{1} = -\frac{1}{\pi}\text{Re}\left(\sum_{n=1}^{\infty}\frac{n^{2}}{n+\alpha}r_{x}^{n}r_{z}^{n}e^{in(\theta_{x}-\theta_{z})}\right).
\end{equation}
Furthermore, the $H^{0}$ and $H^{1}$ norms of $\eta_{x}$ and $G_{x}$ are computed as
\begin{equation}
    |\eta_{x}|_{H^{0}}^{2} = \frac{1}{2\pi}\sum_{n\in\mathbb{Z}}r_{x}^{2|n|} =\frac{1}{2\pi}\frac{1+r_{x}^{2}}{1-r_{x}^{2}},\quad |\eta_{x}|_{H^{1}}^{2} = \frac{1}{2\pi}\sum_{n\in\mathbb{Z}}n^{2}r_{x}^{2|n|} =  \frac{1}{\pi}\frac{r_{x}^{2}(1+r_{x}^{2})}{(1-r_{x}^{2})^{3}}.
\end{equation}
\begin{equation}
    |G_{x}|_{H^{0}}^{2} = \frac{1}{2\pi}\sum_{n\in\mathbb{Z}}\frac{1}{(|n|+\alpha)^{2}}r_{x}^{2|n|}, \quad |G_{x}|_{H^{1}}^{2} = \frac{1}{2\pi}\sum_{n\in\mathbb{Z}}\frac{n^{2}}{(|n|+\alpha)^{2}}r_{x}^{2|n|}.
\end{equation}
Intuitively, we take the scaling function as
$$S(x) = |\eta_{x}|_{H^{1}}^{1/2}|G_{x}|_{H^{1}}^{1/2}.
$$
Then, for a fixed $z\in\Omega$, the kernel function is expressed as
\begin{equation}
    K(x,z) = \frac{(\eta_{x}, G_{z})_{1}}{|\eta_{x}|_{H^{1}}^{1/2}|G_{x}|_{H^{1}}^{1/2}} = -c\frac{\sum_{n=1}^{\infty}\frac{n^{2}}{n+\alpha}r_{x}^{n}r_{z}^{n}\cos(n(\theta_{x}-\theta_{z}))}{\left(\frac{r_{x}^{2}(1+r_{x}^{2})}{(1-r_{x}^{2})^{3}}\right)^{1/4}\left(\sum_{n\in\mathbb{Z}}\frac{n^{2}}{(|n|+\alpha)^{2}}r_{x}^{2|n|}\right)^{1/4}}. \label{kernel_Fourier}
\end{equation}
Here, $c>0$ is a constant independent of $r_x$, $r_z$ and $\theta_x$, $\theta_z$. The minus sign in (\ref{kernel_Fourier}) is a global sign and does not affect peak locations; in the following numerical plots and in Appendix \ref{app:almost-orthogonal-proof}, we use the sign-normalized kernel and still denote it by $K$ for simplicity. We can see from (\ref{kernel_Fourier}) that for fixed $r_x$ and $r_z$, this kernel $K$ attains its maximum when $\theta_x = \theta_z$. Therefore, for each fixed $r_{z}$, we only need to verify whether $K(x,z)$ reaches its maximum at $r_{x}\approx r_{z}$ under the assumption that $\theta_x = \theta_z$. A rigorous proof establishing this almost orthogonal property for the kernel (\ref{kernel_Fourier}) with $\alpha=1$ is provided in Appendix \ref{app:almost-orthogonal-proof}. In this section, we first verify it numerically.

We compute the partial sum of the first 100 terms of the Fourier series for $r_{z}\in(0,1)$and plot the results, where the y-axis represents $\text{argmax}_{r_{x}}K(x,z)$ and the x-axis represents $r_z$. Figure \ref{argmax} shows the argmax graphs of the kernel function $K$ for different $\alpha$, and Figure \ref{circle} displays the profiles of $K(x,z)$ for two different $z$ in $B_{1}(0)$.
\begin{figure}[h]
    \centering
    \includegraphics[scale=0.27]{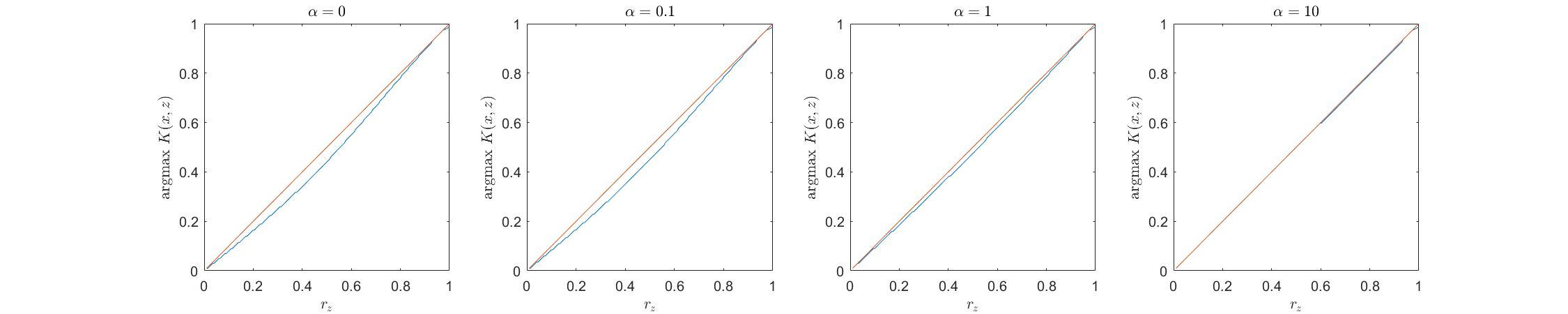}
    \caption{Location of the maximum value of $K(x,z)$ in (\ref{kernel_Fourier}) when $\theta_{x}=\theta_{z}$.}
    \label{argmax}
\end{figure}
\begin{figure}[h]
\centering
\subfigure
{
    \begin{minipage}[b]{.45\linewidth}
        \centering
        \includegraphics[scale=0.23]{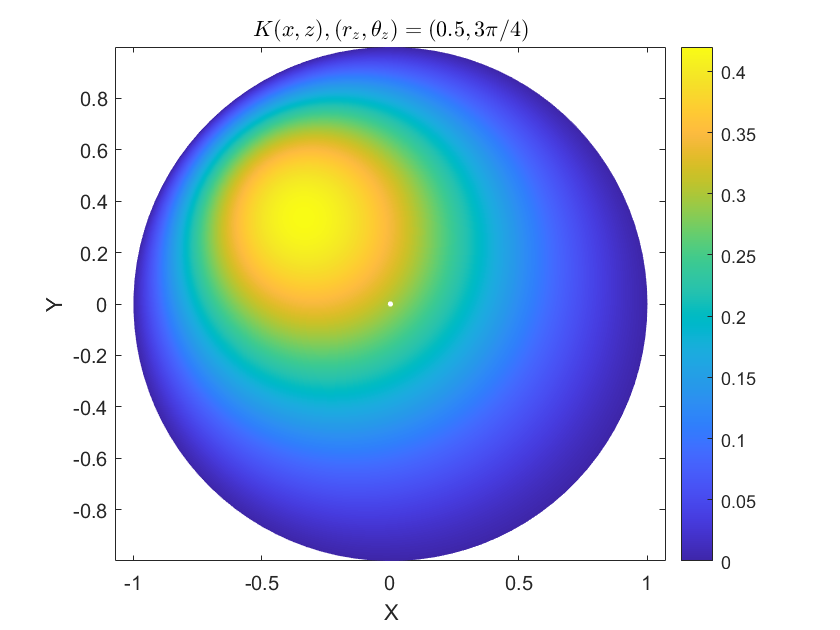}
    \end{minipage}
}
\subfigure
{
 	\begin{minipage}[b]{.45\linewidth}
        \centering
        \includegraphics[scale=0.23]{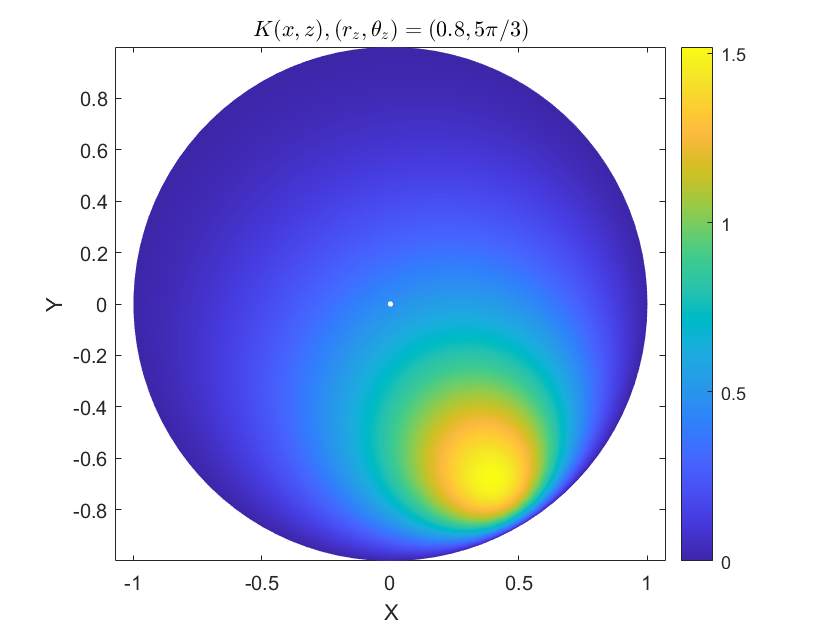}
    \end{minipage}
}
\caption{\centering{Profiles of the kernel $K(x,z)$ in $B_{1}(0)$ for $\alpha=1$.\\
Left: $(r_z,\theta_z)=(0.5, 3\pi/4)$. 
Right: $(r_z,\theta_z)=(0.8, 5\pi/3)$.}}\label{circle}
\end{figure}

For the square domain $\Omega = [-1,1]\times[-1,1]$, we adopt the same scaling function. To enhance computational efficiency, we avoid directly evaluating the Green's functions $w_{x}$ and $G_{x}$ at each point $x$. We follow \cite{chow2021direct_simul} by utilizing the fundamental solution $\Phi_{x}(y) = \frac{1}{2\pi}\log |x-y|$ to define the modified scaling function as
\begin{equation}
S(x) = \left|\frac{\partial\Phi_{x}}{\partial n}\right|_{H^{0}}^{1/2}\left|\frac{\partial\Phi_{x}}{\partial n}\right|_{H^{1}}^{1/2}, \label{scaling}
\end{equation}
which can be computed directly. Here, we employ the following approximations:
$$    |G_{x}|_{H^{1}}\approx\left|\frac{\partial\Phi_{x}}{\partial n}\right|_{H^{0}},\quad |\eta_{x}|_{H^{1}}\approx\left|\frac{\partial\Phi_{x}}{\partial n}\right|_{H^{1}}.
$$
To verify the almost orthogonal property of the modified kernel function in the square domain $\Omega = [-1,1]\times[-1,1]$, we plot a 3D surface $(z_0,D(z_0))$ for $z_0 \in \Omega$, where $D(z_0)$ is defined as
\begin{equation}
    D(z_0) = \left|\text{argmax}_{z} \frac{(\eta_{z}, G_{z_{0}})}{|\frac{\partial\Phi_{z}}{\partial n}|_{H^{0}}^{1/2}|\frac{\partial\Phi_{z}}{\partial n}|_{H^{1}}^{1/2}} - z_{0}\right|. \label{distance}
\end{equation}
The surface graph is shown in Figure \ref{fig:square-dist}.
We can see from Figure \ref{fig:square-dist} that when $K(z,z_{0})$ attains its maximum in $\Omega=[-1,1]\times[-1,1]$, the distance between $z$ and $z_{0}$ remains strictly less than 0.1. This consistently tight bound numerically verifies the almost orthogonal property of the kernel $K$ utilizing the modified scaling function (\ref{scaling}).

Figure \ref{square-kernel1} illustrates the kernel function $K(x,z)$ for two different source points $z$ (marked by dark blue dots). The heatmap demonstrates that the maximum values (represented by red regions) are sharply localized around each $z$ while exhibiting rapid decay away from them. This concentration behavior numerically confirms the almost orthogonal property of the kernel, with 
$K(x,z)$ effectively serving as a spatial indicator that peaks near $x=z$ and decreases with increasing distance from the source point. This concludes the verification of our chosen scaling function $S(x)$ for the Sobolev regularity order $s=1$.
\begin{figure}[H]
\centering
\begin{minipage}[t]{0.40\linewidth}
\centering
\vspace{0pt}
\includegraphics[width=\linewidth]{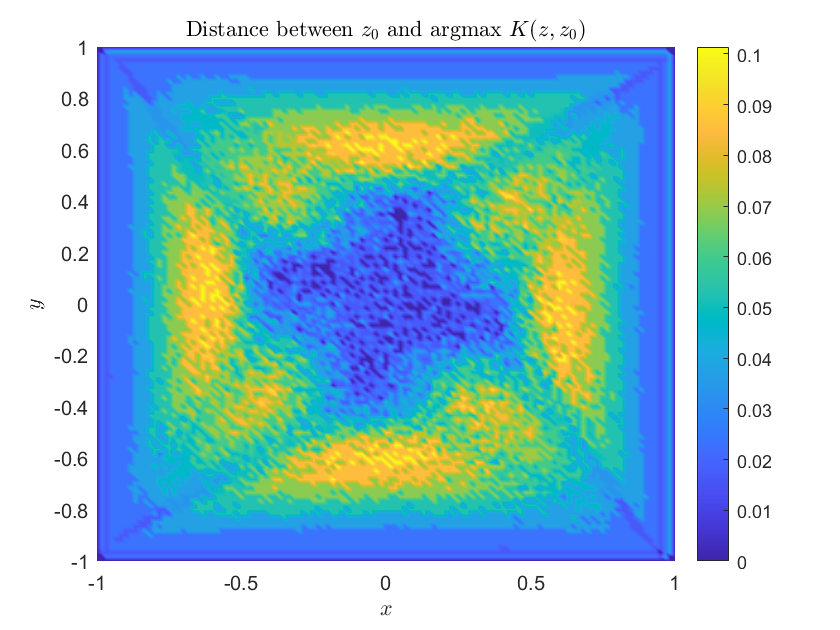}
\caption{Distance between $z_0$ and argmax $K(z, z_0)$ in (\ref{distance}) when $\alpha=1$.}
\label{fig:square-dist}
\end{minipage}\hfill
\begin{minipage}[t]{0.56\linewidth}
\centering
\vspace{0pt}
\includegraphics[width=0.48\linewidth]{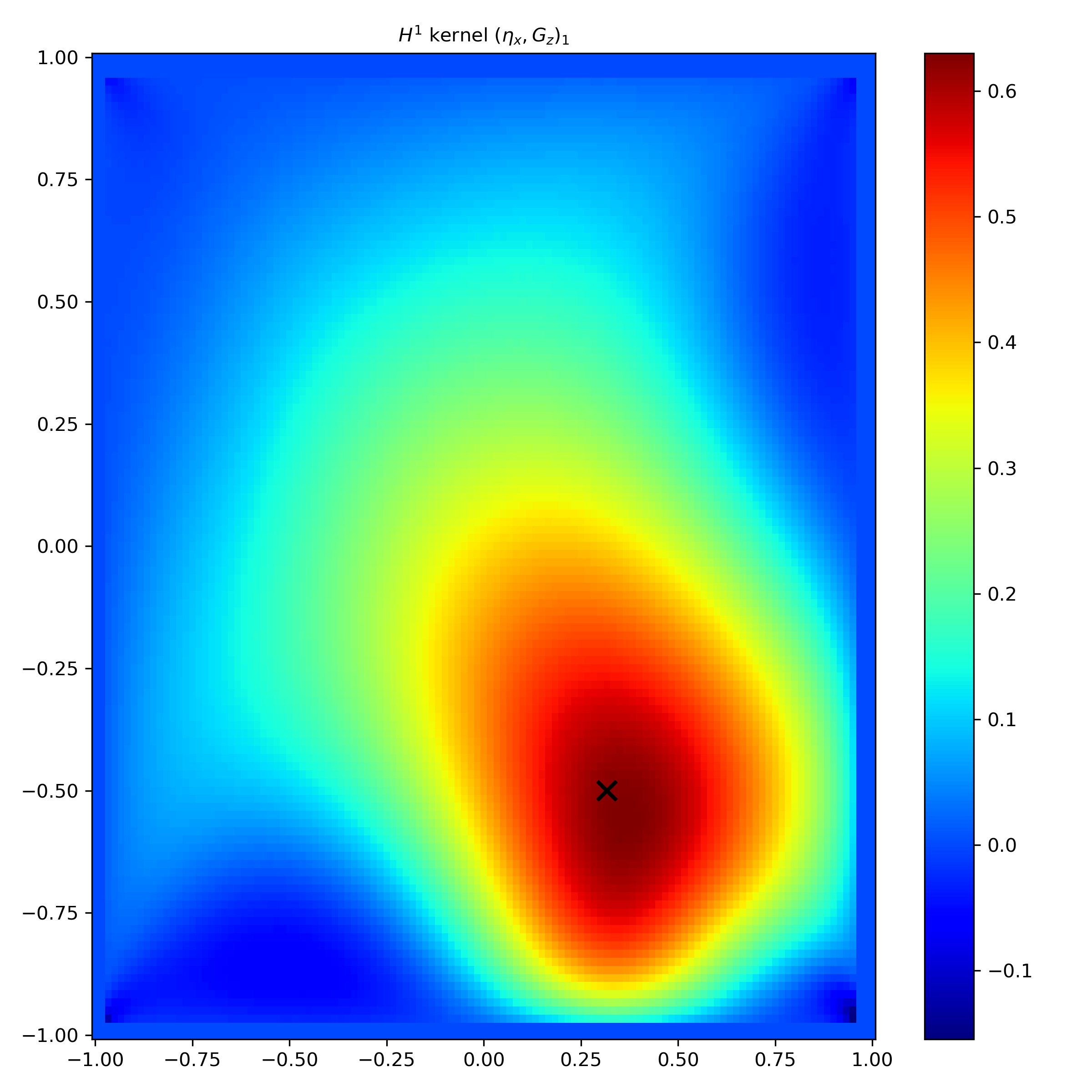}\hfill
\includegraphics[width=0.48\linewidth]{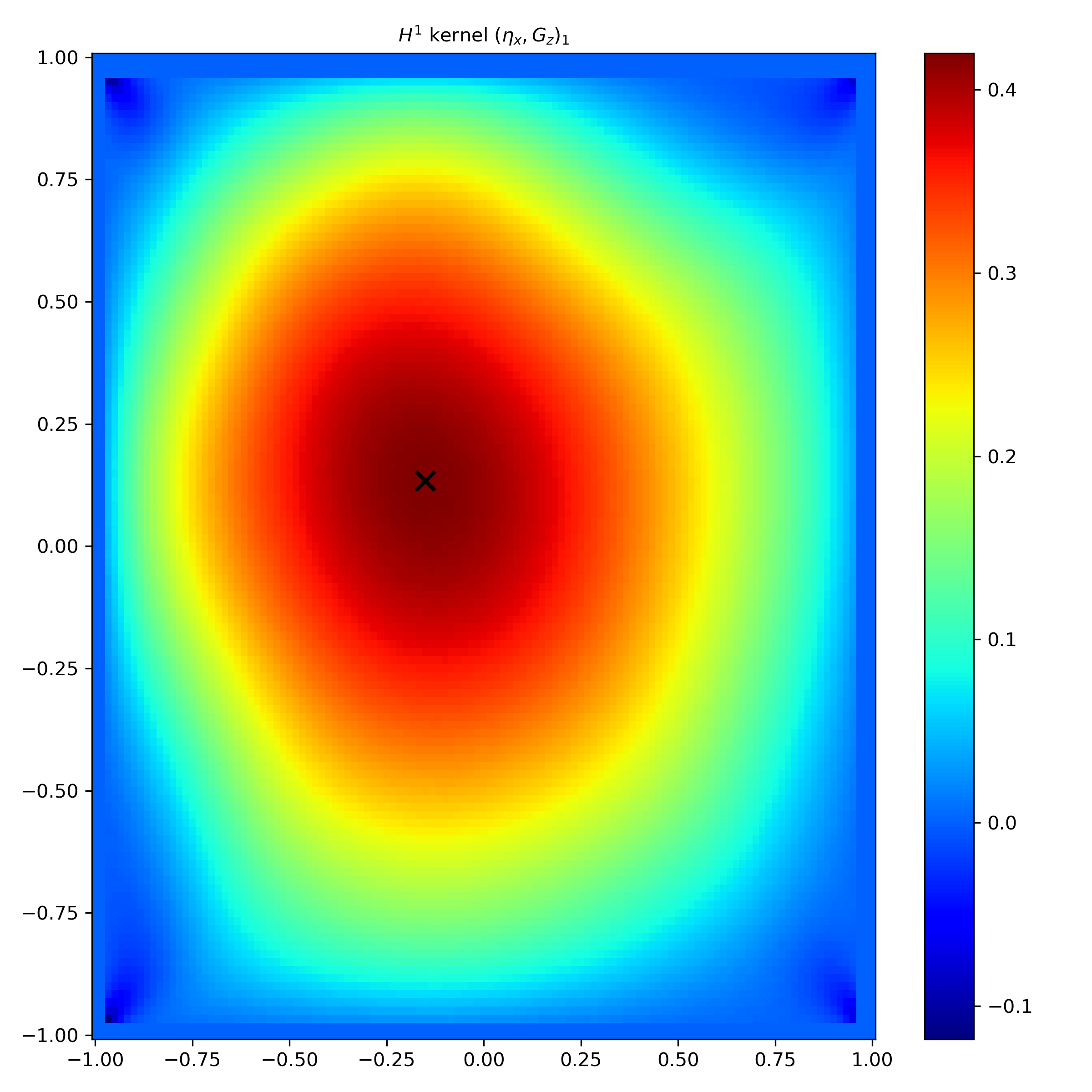}
\caption{$H^{1}$ kernel $K(x,z)$ for $z = (0.317,-0.5)$ and $z=(-0.15, 0.133)$ when $\alpha=1$.}
\label{square-kernel1}
\end{minipage}
\end{figure}

\subsection{Frequency-domain analysis and the selection of the Robin parameter $\alpha$}
\label{sec:alpha_analysis}

The Robin parameter $\alpha$ is pivotal in the reference system, as it governs the balance between high-frequency spatial resolution and the preservation of low-frequency macroscopic structures. Because the scaling function $S(x)$ in the DSM acts as a data-independent spatial normalizer, the spatial profile and the spectral energy distribution of the index function $I_{1}(x)$ are uniquely determined by its unscaled numerator $(f^{\epsilon}-f_{0}^{\epsilon},\eta_{x})_{1}$. To establish a systematic criterion for selecting $\alpha$, we analyze the spectral behavior of this numerator and how $\alpha$ weights the signal and noise across different frequencies using Fourier series.

In this subsection, we distinguish between the exact and noisy reference states. Let $u_0^{\rm ex}$ be the reference solution generated from the exact boundary value $f_{\rm exact}$, and let $u_0^{\epsilon}$ be the reference solution generated from the noisy boundary value $f^{\epsilon}$. We write $f_0^{\rm ex}=u_0^{\rm ex}|_L$ and $f_0^{\epsilon}=u_0^{\epsilon}|_L$.

\begin{lemma}\label{lem:fourier_decay}
Let $J(y)$ denote the physical source induced by the jump conditions on the interior interface $\Gamma$. The exact discrepancy $f_{\rm exact} - f_0^{\rm ex}$ on the outer measurement boundary $L$ (assumed to be the unit circle, $r_x = 1$) satisfies the following exponential spatial decay bound in the Fourier domain:
\begin{equation} \left|\widehat{(f_{\rm exact} - f_{0}^{\rm ex})}_{n}\right| \le C\frac{R_{\text{max}}^{|n|}}{|n|+\alpha}, \label{fourier_dacay}
\end{equation}
where $R_{max} = \max_{y \in \Gamma} r_y < 1$, and $C = \frac{1}{2\pi} \int_{\Gamma} |J(y)| ds_y$.
\end{lemma}

\begin{proof}
The exact discrepancy on the measurement boundary can be expressed via Green's formula:
\begin{equation}
    (f_{\rm exact} - f_0^{\rm ex})(x) = \int_{\Gamma} J(y) G_y(x) ds_y.
\end{equation}
Substituting (\ref{Gx-Fourier}) into the integral and exchanging the order of integration and summation, we obtain the $n$-th Fourier coefficient of the exact boundary signal:
\begin{equation}
    \widehat{(f_{\rm exact} - f_0^{\rm ex})}_n = \frac{1}{2\pi(|n| + \alpha)} \int_{\Gamma} J(y) r_y^{|n|} e^{-in\theta_y} ds_y.
\end{equation}
Given that the interface $\Gamma$ is strictly enclosed within the unit disk, $r_y \le R_{max} < 1$. Taking the absolute value yields the desired inequality (\ref{fourier_dacay}).
\end{proof}

This inequality highlights the ill-posedness of the inverse problem: high-frequency structural information of the source is subjected to exponential decay $R_{max}^{|n|}$ before reaching the boundary. In practice, the exact data is polluted by noise $\epsilon$, yielding the noisy Fourier coefficients $\hat f^\epsilon_n = \hat{f}_{{\rm exact},n} + \epsilon_n$. The exact and noisy reference solutions satisfy
\begin{equation}
    \widehat{u_0^{\rm ex}}_n = \frac{\hat g_n + \alpha \hat f_{{\rm exact},n}}{|n|+\alpha}, \qquad
    \widehat{u_0^{\epsilon}}_n = \frac{\hat g_n + \alpha \hat f^\epsilon_n}{|n|+\alpha}.
\end{equation}
Hence the Fourier coefficients of the noisy input discrepancy $d^{\epsilon}:=f^{\epsilon}-u_0^{\epsilon}$ satisfy
\begin{equation}
    \hat{d}^{\epsilon}_n = \widehat{(f_{\rm exact} - u_0^{\rm ex})}_n + \frac{|n|}{|n| + \alpha} \epsilon_n.
\end{equation}
The $H^1$ inner product in the index function introduces a multiplier $|n|^2$. Following (\ref{etax-Fourier}), the numerator of the probing index becomes, up to the same global sign normalization as in Section 3.1,
\begin{equation}
    \text{Numerator}^{\epsilon}(x) = \frac{1}{2\pi} \sum_{n \in \mathbb{Z}} |n|^2 \widehat{(f_{\rm exact} - u_0^{\rm ex})}_n r_x^{|n|} e^{in\theta_x} + \frac{1}{2\pi} \sum_{n \in \mathbb{Z}} \frac{|n|^3}{|n| + \alpha} \epsilon_n r_x^{|n|} e^{in\theta_x}.
\end{equation}
Based on the physical decay bound (\ref{fourier_dacay}), we define the effective signal weight $SW(n)$ and noise weight $NW(n)$ as:
\begin{equation}
    SW(n) = \frac{|n|^2}{|n| + \alpha} R_{max}^{|n|}, \quad NW(n) = \frac{|n|^3}{|n| + \alpha}.
    \label{eq:spectral_weights}
\end{equation}

Regarding the selection of $\alpha$, numerical experiments indicate that utilizing a small parameter such as $\alpha = 0.1$ yields an overly flat index profile. When $\alpha \to 0$, the signal weight degenerates to $SW(n) \approx |n| R_{max}^{|n|}$. Without sufficient penalization of fundamental modes, the signal energy is excessively concentrated at the lowest frequencies (e.g., $n=1,2$), which spatially manifests as broad, smooth waves. Moreover, a small $\alpha$ leads to an abnormally large scaling factor $S(x, \alpha)$, uniformly squashing the amplitude of the index function.

To systematically evaluate the optimal $\alpha$, Table \ref{tab:spectral_weights} compares these weights at low-frequency ($n=2$) and mid-frequency ($n=5$) modes for $R_{max} = 0.5$.

\begin{table}[htbp]
\centering
\caption{Comparison of spectral weights for different $\alpha$ ($R_{max} = 0.5$).}
\label{tab:spectral_weights}
\begin{tabular}{lcccc}
\toprule
Parameter & Mode ($n$) & $SW(n)$ & $NW(n)$ \\
\midrule
\multirow{2}{*}{$\alpha=1$} & Low ($n=2$) & $0.333$ & $2.667$\\
 & Mid ($n=5$) & $0.130$ & $20.833$ \\
\midrule
\multirow{2}{*}{$\alpha=10$} & Low ($n=2$) & $0.083$ & $0.667$ \\
 & Mid ($n=5$) & $0.052$ & $8.333$\\
\bottomrule
\end{tabular}
\end{table}

For Problem I, localizing sharp flux jumps necessitates isolating higher-frequency features. Table \ref{tab:spectral_weights} shows that $\alpha=10$ shifts the signal energy centroid toward higher frequencies. Specifically, the high-to-low ratio $SW(5)/SW(2)$ increases from $0.39$ (when $\alpha=1$) to $0.63$. By damping the absolute high-frequency noise $NW(5)$, this spectral tilt achieves high spatial resolution for local anomalies. Conversely, for Problem II, reconstructing an unknown macroscopic topology relies on the survival of low-frequency modes. In contrast to the high-pass effect of $\alpha=10$, a moderate $\alpha=1$ yields a smaller ratio $SW(5)/SW(2) = 0.39$. This ensures the spectral centroid remains anchored at lower frequencies, which is necessary to recover the macroscopic geometry against the $R_{max}^{|n|}$ decay.

\section{A refinement approach for DSM}
In this section, the refinement approach is presented and tested only for the
full Dirichlet measurement case \(L=\partial\Omega\). Experiments show that the numerical results of the DSM with the $H^{0}$ semi-inner product are more robust to noise but less precise than those with the $H^{1}$ semi-inner product. We then present a refinement approach to enhance the performance of the DSM with the $H^{0}$ semi-inner product. Using Green's second identity and equations (\ref{con:model_problem}) (\ref{step1_ref_sys})(\ref{step2_ref_sys})(\ref{probe-def})(\ref{adjoint-def}), we have for Problem I:
\begin{equation}
    \lambda_{0}(x)=(f-f_{0}, \eta_{x})_{0} = -(u-u_{0})(x)+\int_{\Gamma}(\gamma u-\gamma_{0}u_{0})\omega_{x},\quad x\in\text{int}(\Omega). \label{lam-step1-rep}
\end{equation}
Similarly, for Problem II, we obtain:
\begin{equation}
    \lambda_{0}(x)=(f-f_{0}, \eta_{x})_{0} = -(u-u_{0})(x)+\int_{\Gamma}\gamma u\omega_{x},\quad x\in\text{int}(\Omega). \label{lam-step2-rep}
\end{equation}
Note that in numerical tests, these two equations are valuable tools for verifying the correct solution of the forward model problem (\ref{con:model_problem}). If a better approximation $\hat{u}$ of the true solution $u$ is available compared to the reference solution $u_{0}$, we can sharpen the numerical results of the DSM using the $H^{0}$ semi-inner product. To find a better approximation $\hat{u}$ to replace $u_0$ in (\ref{lam-step1-rep}) and (\ref{lam-step2-rep}), we propose the following iterative algorithm.
\begin{enumerate}
    \item Initialize $\lambda^{0}(x)=0$, and set a constant $\delta > 0$.
    \item For $k = 0, 1$:\\
For Problem I, solve
\begin{equation}
    \left\{
\begin{aligned}
-\Delta u_{k}=-\lambda^{k} \quad \text{in} \quad &\Omega\setminus\Gamma,\\
\left[\frac{\partial u_{k}}{\partial \nu}\right] = \gamma_{0}u_{k} \quad \text{on}\quad &\Gamma, \\
\frac{\partial u_{k}}{\partial n}+\alpha u_{k}=g+\alpha f \quad \text{on} \quad & L,\\
\frac{\partial u_{k}}{\partial n} = g\quad \text{on} \quad &\partial\Omega\setminus L.
\end{aligned}
\right.
\end{equation}
for Problem II, solve
\begin{equation}
    \left\{
\begin{aligned}
-\Delta u_{k}=-\lambda^{k} \quad \text{in} \quad &\Omega,\\
\frac{\partial u_{k}}{\partial n}+\alpha u_{k}=g+\alpha f \quad \text{on} \quad & L,\\
\frac{\partial u_{k}}{\partial n} = g\quad \text{on} \quad &\partial\Omega\setminus L,
\end{aligned}
\right. \label{u_k_equation}
\end{equation}
to obtain $f_{k} = u_{k}|_{\partial\Omega}$. Then, solve
\begin{equation}
        \left\{
\begin{aligned}
-\Delta \lambda^{k+1}=0 \quad \text{in} \quad &\Omega,\\
\lambda^{k+1}=-\frac{f-f_{k}}{\delta}\quad \text{on} \quad &\partial\Omega,\\
\end{aligned}
\right.
\end{equation}
to obtain $\lambda^{k+1}$, and update $k \leftarrow k+1$.
\item Compute the modified probing index
\begin{equation}
    \hat{I}_{0}(x) = \frac{\lambda^{2}(x)}{S(x)} .\label{index-refine-def}
\end{equation}
To satisfy the almost orthogonal property, we define the scaling function for the $H^0$ semi-inner product as
$$S(x) = \left|\frac{\partial\Phi_{x}}{\partial n}\right|_{L^{2}}.
$$
\end{enumerate}
From this iterative scheme, we can utilize Green's second identity to show that for Problem I,
\begin{equation}
    \lambda^{2}(x) = \frac{1}{\delta}\left(-(u-u_{1})(x) + \int_{\Omega}\lambda^{1}\omega_{x} + \int_{\Gamma}(\gamma u-\gamma_{0}u_{1})\omega_{x}\right). \label{lamk-step1-rep}
\end{equation}
Similarly, for Problem II, we have:
\begin{equation}
    \lambda^{2}(x) = \frac{1}{\delta}\left(-(u-u_{1})(x) + \int_{\Omega}\lambda^{1}\omega_{x} + \int_{\Gamma}\gamma u\omega_{x}\right). \label{lamk-step2-rep}
\end{equation}
Comparing (\ref{lamk-step1-rep})(\ref{lamk-step2-rep}) with (\ref{lam-step1-rep})(\ref{lam-step2-rep}), the $u_{0}$ terms are replaced by $u_{1}$, and an additional term $\int_{\Omega}\lambda^{1}\omega_{x} dy$ is introduced. This new term is continuously differentiable due to the elliptic regularity of Poisson's equation. The numerical results presented in Section 5.3 demonstrate that this refinement enhances the performance of the DSM with the $H^0$ semi-inner product while maintaining its excellent robustness against noise.

To theoretically justify the proposed refinement, we perform the following calculations for the case where $L = \partial\Omega$, meaning that the full measurement data $f$ is available on the entire outer boundary. For Problem I, using the calculations in Section 2.6 and equation (\ref{step1repre}), we obtain:
$$\lambda^{1}(x) = \frac{1}{\delta}\int_{\partial\Omega}(f-f_0)(y)\eta_x(y)ds_y = \frac{1}{\delta}\int_{\partial\Omega}\left(\int_{\Gamma}(\gamma(z)u(z)-\gamma_0u_0(z))G_{z}(y)ds_z\right)\eta_{x}(y)ds_y.
$$
Then, by Fubini's theorem and the definition of the kernel function (\ref{kernel}), the original index function is given by:
\begin{equation}
  I_0(x) = \frac{\lambda^{1}(x)}{S(x)} = \frac{1}{\delta S(x)}\int_{\Gamma}(\gamma(z)u(z)-\gamma_0u_0(z))K(x,z)S(x)ds_z = \frac{1}{\delta}\int_{\Gamma}(\gamma(z)u(z)-\gamma_0u_0(z))K(x,z)ds_z  .  \label{I_step1_repre}
\end{equation}
Using a technique similar to that in Section 2.3 along with the refinement scheme, we expand the difference:
$$\begin{aligned}
    (u-u_{1})(x) &= \int_{\Omega}-(u-u_{1})\Delta G_{x} dy + \int_{\Omega_{1}}G_{x}\Delta(u-u_1) dy + \int_{\Omega_{2}}G_{x}\Delta(u-u_1) dy + \int_{\Omega}G_x\Delta u_1 dy \\
    &=\int_{\Gamma}(\gamma(y) u(y)-\gamma_0 u_1(y))G_x(y)ds_y + \int_{\Omega}G_x(y)\lambda^{1}(y)dy.
\end{aligned}$$
Subsequently, according to the refinement scheme, we have:
$$\begin{aligned}
    \lambda^{2}(x) &= \frac{1}{\delta}\int_{\partial\Omega}(f-f_1)(y)\eta_x(y)ds_y = \frac{1}{\delta}\int_{\partial\Omega}(u-u_1)(y)\eta_x(y)ds_y \\
    &= \frac{1}{\delta}\int_{\Gamma}(\gamma(y) u(y)-\gamma_0 u_1(y))S(x)K(x,y)ds_y + \frac{1}{\delta}\int_{\Omega}\lambda^{1}(y)S(x)K(x,y)dy.
\end{aligned}$$
Therefore, the refined index function can be expressed as:
\begin{equation}
    \begin{aligned}
    \hat{I}_0(x) &= \frac{\lambda^{2}(x)}{S(x)} = \frac{1}{\delta}\int_{\Gamma}(\gamma(y) u(y)-\gamma_0 u_1(y))K(x,y)ds_y + \frac{1}{\delta}\int_{\Omega}\lambda^{1}(y)K(x,y)dy\\
    &=\frac{1}{\delta}\int_{\Gamma}(\gamma(y) u(y)-\gamma_0 u_1(y))K(x,y)ds_y + \frac{1}{\delta^2}\int_{\Gamma}(\gamma(y) u(y)-\gamma_0 u_0(y))\left(\int_{\Omega}S(z)K(y,z)K(x,z)dz\right)ds_y. \label{hat_I_step1_repre}
\end{aligned}
\end{equation}
Similarly, for Problem II, the original index function is represented as:
\begin{equation}
  I_0(x) = \frac{\lambda^{1}(x)}{S(x)} = \frac{1}{\delta S(x)}\int_{\Gamma}\gamma(z)u(z)K(x,z)S(x)ds_z = \frac{1}{\delta}\int_{\Gamma}\gamma(z)u(z)K(x,z)ds_z  .  \label{I_step2_repre}
\end{equation}
Furthermore, the refined index function is expressed as:
\begin{equation}
    \begin{aligned}
    \hat{I}_0(x) &= \frac{\lambda^{2}(x)}{S(x)} = \frac{1}{\delta}\int_{\Gamma}\gamma(y)u(y)K(x,y)ds_y + \frac{1}{\delta}\int_{\Omega}\lambda^{1}(y)K(x,y)dy\\
    &=\frac{1}{\delta}\int_{\Gamma}\gamma(y)u(y)K(x,y)ds_y + \frac{1}{\delta^2}\int_{\Gamma}\gamma(y)u(y)\left(\int_{\Omega}S(z)K(y,z)K(x,z)dz\right)ds_y .\label{hat_I_step2_repre}
\end{aligned}
\end{equation}
Comparing (\ref{I_step1_repre})(\ref{hat_I_step1_repre}) and (\ref{I_step2_repre})(\ref{hat_I_step2_repre}), we observe that the refinement term $\int_{\Omega}S(z)K(y,z)K(x,z)dz$ enhances the original index function through a fundamental mechanism: The product $K(y,z)K(x,z)$ effectively reinforces the almost orthogonal property of the kernel function. This reinforcement occurs because the kernel product amplifies coherence when both $x$ and $y$ approach the true inhomogeneity location $z$, and accelerates decay when either $x$ or $y$ moves away from $z$. Consequently, the modified index function $\hat{I}_0(x)$ exhibits superior localization properties while maintaining the inherent robustness of the $H^0$ formulation.

\section{Numerical results}
In this section, we first present numerical results using the $H^{1}$ semi-inner product for Problem I and Problem II without the refinement approach. The corresponding figures are shown in Sections 5.1 and 5.2. Subsequently, the numerical results of the refinement approach utilizing the $H^{0}$ semi-inner product are presented in Section 5.3. We first outline the computation procedure without refinement.

First, we designed a specific interface $\Gamma\subset\Omega = [-1,1]\times[-1,1]$ and a function $g:\partial\Omega\rightarrow\mathbb{R}$ as the boundary flux. Second, we applied the finite difference method to solve \eqref{con:model_problem}, supplemented with the Neumann boundary condition \(\partial u/\partial n=g\) on \(\partial\Omega\), to obtain the exact boundary measurement $f=u|_L$. To test the robustness of the DSM, we added random noise to $f$ and obtained $f^{\epsilon}$, so that $f^{\epsilon} = (1+\epsilon\xi)f$.
Here, \(\xi\) is a uniformly distributed random variable in \([-1,1]\), and $\epsilon$ is the noise level, which is specified in each figure caption. In our numerical tests, we took $\gamma$ as a piecewise continuous function. Third, we solved the reference system (\ref{step1_ref_sys})(\ref{step2_ref_sys}) using the boundary flux $g$ and the noisy measurement data $f^{\epsilon}$ to obtain the reference solution $u_{0}$. Fourth, using the reference data $f_{0} = u_{0}|_{L}$, we solved the adjoint system (\ref{adjoint-def}) to obtain the adjoint function $\lambda_{s}$. Finally, we computed our desired probing index $I_{s}(x)$:
\begin{equation}
    I_{s}(x) = \frac{\lambda_{s}(x)}{S(x)}.
\end{equation}
Here, the scaling function $S$ was computed by (\ref{scaling}) for $s=1$. To normalize our probing index, we plotted the graph of $\widetilde{I}_{1}(x)$ for our numerical results, where
\[
\widetilde{I}_{1}(x)=
\left({I_{1}(x)}/{\max(\sup I_{1}(x),10^{-6})}\right)^2
\quad \text{if } I_{1}(x)\geq0,\qquad
\left({I_{1}(x)}/{\min(\inf I_{1}(x),-10^{-6})}\right)^2
\quad \text{if } I_{1}(x)<0.
\]

For Problem I, the probing index $I_1$ is defined on $\Gamma$ and is set to 0 otherwise. We set the background value $\gamma_{0}=1$, where $\gamma$ exhibited inhomogeneity on $\Gamma$. On the graph of the normalized probing index $\widetilde{I}_{1}$, we plot small red marks at the location of $D=\text{supp}(\gamma - \gamma_{0})\subset \Gamma$. Ideally, these marks should coincide precisely with the peaks of the probing index. For each interface $\Gamma$, 3 pairs of Cauchy data $(f,g)$ were provided, yielding three corresponding normalized probing indices $\widetilde{I}_{1}$. We emphasize that each probing index is generated from one single Cauchy pair.
The use of three boundary excitations in the numerical experiments is a
fusion strategy: each index provides partial
information, and the composite index aggregates the visible components from
different measurements. We conducted numerical tests using both full boundary data $f^{\epsilon}$ and $75\%$ partial data (utilizing measurements from only three sides of $\partial\Omega$). We added $4\%$ noise to the full data and $2\%$ noise to the partial data. We observe that the probing indices $\widetilde{I}_{1}$ can reliably capture at least one location of $D\subset\Gamma$ for each boundary flux $g$. By combining results from three distinct fluxes $g$, all inhomogeneous regions $D$ were successfully identified. The results for Problem I are presented in Section 5.1.

For Problem II, we assumed $\gamma(x)$ to be piecewise continuous on $\Gamma$. As in Problem I, three pairs of Cauchy data $(f,g)$ were provided for each interface, and we investigated both the full and $75\%$ partial data cases under identical noise levels ($4\%$ and $2\%$, respectively). Remarkably, despite the inherent ill-posedness of Problem II characterized by multiple unknown interfaces and limited noisy data, our DSM approach demonstrated robust performance. Each probing index $\widetilde{I}_{1}$ successfully captured at least one interface segment, indicated by larger index values near the true locations. By integrating the three individual results, a comprehensive and accurate reconstruction of $\Gamma$ was achieved. The results for Problem II are presented in Section 5.2.

For both Problem I and Problem II, we made use of a $181\times 181$ uniform mesh to perform our finite difference PDE solver to solve for the true solution $u$. To avoid inverse crime, we used a $121\times 121$ uniform mesh alternatively to solve $u_{0}$ and $\lambda$ using the interpolated value of $f^{\epsilon}$. The values of $\alpha$ were set be 10 and 1 for Problem I and Problem II, respectively. The settings for the examples are listed as follows.

In Section 5.1, we detail the specific numerical examples for Problem I. Throughout these examples, we set the Robin parameter to $\alpha=10$.\\
\textbf{Example 1:} Non-closed interface. $\Gamma  = [-0.766, 0.5]\times\{-0.5\}\cup \{-0.766\}\times[-0.5,0.233]$.
$$g_{1}(x,y) = |x+y+1|,\quad g_{2}(x,y) = 2,\quad g_{3}(x,y) = \cos(2\pi y)
$$
$$\gamma(x,y) = \left\{
\begin{aligned}
10+10x,& \quad y = -0.5, |x-0.2|\leq 0.05\\
1,& \quad\text{else}
\end{aligned}
\right.
$$
\textbf{Example 2:} Closed interface. $\Gamma = \partial\Omega_{1}$, where $\Omega_{1} = [-0.766, 0.5]\times[-0.766, 0.733]$.
$$g_{1}(x,y) = |x+y+1|,\quad g_{2}(x,y) = \sin(\pi x),\quad g_{3}(x,y) = \cos(2\pi y) - x
$$
$$\gamma(x,y) = \left\{
\begin{aligned}
5,& \quad y = -0.766, |x-0.2|\leq 0.05 \\
5,& \quad y = 0.733, |x-0.3|\leq 0.05\\
5,& \quad x = -0.766, |y-0.1|\leq 0.05\\
1,& \quad\text{else}
\end{aligned}
\right.
$$
\textbf{Example 3:} Non-closed interface. $\Gamma = [-0.733, 0.767]\times\{0.5\}$. $$g_{1}(x,y) = \sin(2\pi(x+y)),\quad g_{2}(x,y) = \sin(4\pi(x+y)),\quad g_{3}(x,y) = \sin(\pi(x+y))
$$
$$\gamma(x,y) = \left\{
\begin{aligned}
10,& \quad y = 0.5, |x-0.5|\leq 0.05 \\
10,& \quad y = 0.5, |x+0.5|\leq 0.05 \\
1,& \quad\text{else}
\end{aligned}
\right.
$$
\textbf{Example 4:} Non-closed interface. $\Gamma = [-0.266, 0.733]\times\{0.5\}\cup\{-0.5\}\times[-0.766, 0.233]$. $$g_{1}(x,y) = \sin(2\pi(x+y)),\quad g_{2}(x,y) = \sin(4\pi(x+y)),\quad g_{3}(x,y) = \sin(\pi(x+y))
$$
$$\gamma(x,y) = \left\{
\begin{aligned}
10,& \quad y = 0.5, |x-0.3|\leq 0.05 \\
10,& \quad x = -0.5, |y+0.2|\leq 0.05 \\
1,& \quad\text{else}
\end{aligned}
\right.
$$
\textbf{Example 5:} $75\%$ data available. The same setting as in Example 1. The unavailable data part $\partial\Omega\setminus L=\{1\}\times[-1,1]$.\\
\textbf{Example 6:} $75\%$ data available. The same setting as 
in Example 3. The unavailable data part $\partial\Omega\setminus L=\{1\}\times[-1,1]$.\\
\textbf{Example 7:} $75\%$ data available. The same setting as in Example 4. The unavailable data part $\partial\Omega\setminus L=\{1\}\times[-1,1]$.\\ 

In Section 5.2, we detail the specific numerical examples for Problem II. Throughout these examples, we set the Robin parameter to $\alpha=1$. For Examples 1 and 2, the interfaces consisted of the boundaries of small rectangles. In Example 1, $\gamma=2$ on all four edges of each rectangle. In Example 2, for each rectangle, $\gamma=2$ on the left and right edges, $\gamma=-1$ on the top edge, and $\gamma=3$ on the bottom edge; corner values are immaterial. For Examples 3 and 4, the interfaces were composed of several separate line segments.\\
\textbf{Example 1:} Closed interface. $\Gamma = \partial\Omega_{1}\cup  \partial\Omega_{2}\cup  \partial\Omega_{3}$, where $\Omega_{1} = [-0.083, 0.167]\times[-0.667, -0.417]$, $\Omega_{2} = [-0.667, -0.333]\times[0.333, 0.583]$, $\Omega_{3} = [0.333, 0.617]\times[0.333, 0.667]$; $\gamma=2$ on each $\partial\Omega_j$.
$$g_{1}(x,y) = 2,\quad g_{2}(x,y) = |x+y+1|,\quad g_{3}(x,y) = |\cos(x+y)|
$$
\textbf{Example 2:} Closed interface. $\Gamma = \partial\Omega_{1}\cup  \partial\Omega_{2}$, where $\Omega_{1} = [-0.667, -0.333]\times[-0.667, -0.333]$, $\Omega_{2} = [0.333, 0.667]\times[0, 0.333]$. For each $\Omega_j=[a_j,b_j]\times[c_j,d_j]$, $\gamma=2$ on $\{a_j,b_j\}\times[c_j,d_j]$, $\gamma=-1$ on $[a_j,b_j]\times\{d_j\}$, and $\gamma=3$ on $[a_j,b_j]\times\{c_j\}$.
$$g_{1}(x,y) = \sin(\pi(x+y)),\quad g_{2}(x,y) = |x+y+1|,\quad g_{3}(x,y) = \sin(2\pi(x+y))
$$
\textbf{Example 3:} Non-closed interface. $\Gamma = [-0.75, -0.167]\times\{-0.667\}\cup [0.2, 0.667]\times\{0.333\}$, $\gamma(x,y) = 5$ on $\Gamma$.
$$g_{1}(x,y) = x+y,\quad g_{2}(x,y) = \sin(\pi(x+y)),\quad g_{3}(x,y) = \sin(2\pi(x+y))
$$
\textbf{Example 4:} Non-closed interface. $\Gamma = \{-0.167\}\times[-0.667, 0.667]$.
$$\gamma(x,y) = \left\{
\begin{aligned}
-10,& \quad |y|\leq 0.4 \\
10,& \quad\text{else}
\end{aligned}
\right.$$
$$g_{1}(x,y) = 2,\quad g_{2}(x,y) = |\sin(\pi(x+y))|,\quad g_{3}(x,y) = |\cos(\pi(x+y))|
$$
\textbf{Example 5:} $75\%$ data available. The same setting as in Example 2. The unavailable data part $\partial\Omega\setminus L=[-1,1]\times\{1\}$.\\
\textbf{Example 6:} $75\%$ data available. The same setting as in Example 3. The unavailable data part $\partial\Omega\setminus L=[-1,1]\times\{1\}$.\\
\textbf{Example 7:} $75\%$ data available. The same setting as in Example 4. The unavailable data part $\partial\Omega\setminus L=\{1\}\times[-1,1]$.
\subsection{Problem I}
\begin{figure}[h]
\centering
\subfigure
{
    \begin{minipage}[b]{.11\linewidth}
        \centering
        \includegraphics[scale=0.27]{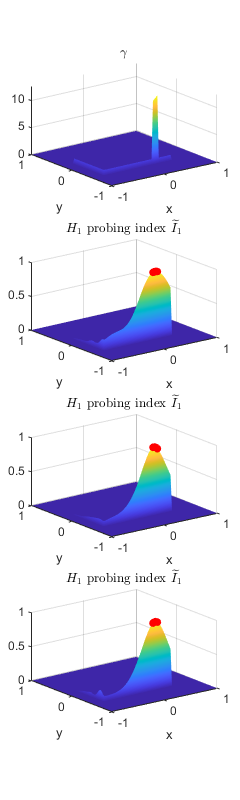}
    \end{minipage}
}
\subfigure
{
 	\begin{minipage}[b]{.11\linewidth}
        \centering
        \includegraphics[scale=0.27]{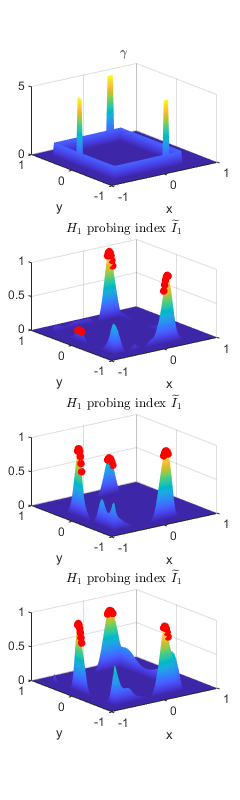}
    \end{minipage}
}
\subfigure
{
 	\begin{minipage}[b]{.11\linewidth}
        \centering
        \includegraphics[scale=0.27]{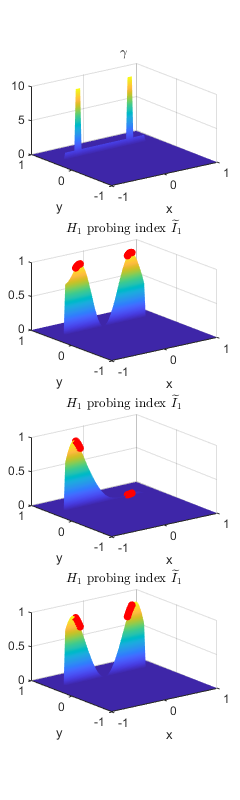}
    \end{minipage}
}
\subfigure
{
 	\begin{minipage}[b]{.11\linewidth}
        \centering
        \includegraphics[scale=0.27]{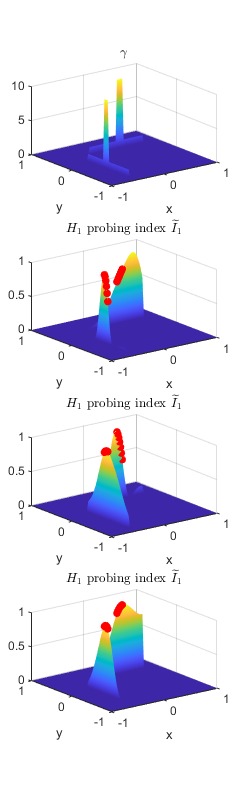}
    \end{minipage}
}
\subfigure
{
 	\begin{minipage}[b]{.11\linewidth}
        \centering
        \includegraphics[scale=0.27]{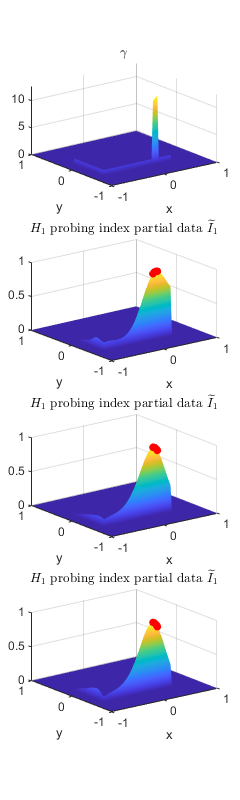}
    \end{minipage}
}
\subfigure
{
 	\begin{minipage}[b]{.11\linewidth}
        \centering
        \includegraphics[scale=0.27]{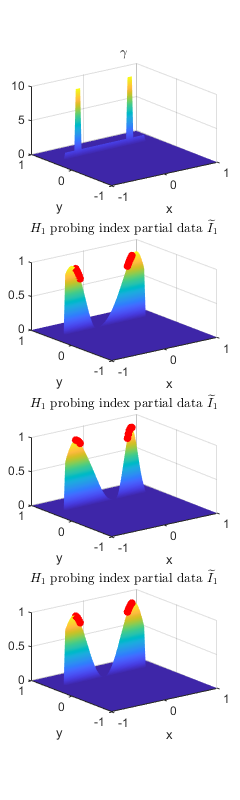}
    \end{minipage}
}
\subfigure
{
 	\begin{minipage}[b]{.11\linewidth}
        \centering
        \includegraphics[scale=0.27]{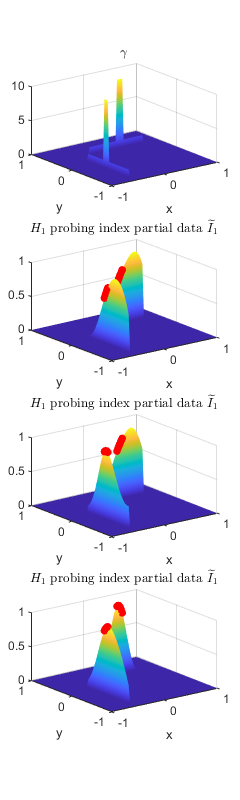}
    \end{minipage}
}
\caption{Problem I results: full data with $4\%$ noise (left 4 columns) and $75\%$ partial data with $2\%$ noise (right 3 columns). Rows correspond to different boundary fluxes $g$.} \label{problem1figure}
\end{figure}
For Problem I with full data available on the outer boundary, we present four numerical examples to demonstrate the capability of the proposed DSM in reconstructing flux jump locations on a known interface under four scenarios: (i) a single jump on a non-closed interface; (ii) three jumps on a closed interface; (iii) two jumps on a non-closed interface; and (iv) two jumps distributed across two non-closed interfaces. For the partial data case ($75\%$ available), we investigate scenarios (i), (iii), and (iv). As illustrated in Figure \ref{problem1figure}, a single pair of Cauchy data was sufficient for the $H^1$ probing index to accurately capture at least one jump location. More importantly, by combining the probing indices derived from multiple distinct pairs of Cauchy data, we successfully identified every jump location, thereby achieving a complete and highly accurate reconstruction of the defect distribution.

To quantitatively assess the performance of the proposed DSM and to facilitate an objective evaluation of the probing indices, we introduce two quantitative measures: Mean Localization Error (MLE) and Contrast-to-Noise Ratio (CNR).

To comprehensively evaluate the reconstruction performance across all available measurements, we first define the composite evaluation index $I_{\text{comp}}(x)$ as the sum of the squared normalized probing indices obtained from the three distinct pairs of Cauchy data. That is,
$$    I_{\text{comp}}(x) = \sum_{i=1}^3 |\widetilde{I}^{(i)}(x)|^2,$$
where $\widetilde{I}^{(i)}(x)$ represents either the refined or unrefined normalized index for the $i$-th measurement. Both the MLE and the CNR are evaluated based on this composite index.

The Mean Localization Error (MLE) evaluates the spatial accuracy of the overall reconstruction. Let $x_j^{\dagger}$ be the true geometric center of the $j$-th flux jump. The reconstructed peak $\hat{x}_j$ is identified by searching for the local maximum of the composite index $I_{\text{comp}}(x)$ within a neighborhood of radius $R_{\text{search}} = 0.2$:
$$    \hat{x}_j = \arg\max_{\substack{x \in \Gamma_h \\ \|x - x_j^{\dagger}\| \le 0.2}} I_{\text{comp}}(x).$$
With the defect radius fixed at $0.05$, the single-defect localization error and the mean localization error are defined by
\[
LE_j=\max(0,\|\hat{x}_j-x_j^{\dagger}\|-0.05),\qquad
MLE=\frac{1}{M}\sum_{j=1}^{M}LE_j .
\]

The Contrast-to-Noise Ratio (CNR) quantifies the sharpness of the composite discrete probing index against background artifacts. Let $\Gamma_h$ denote the set of discrete grid points on the known interface. We define $S_{\text{target}}$ as the subset of grid points on the interface located within a distance of $0.05$ from the true defect centers $x^{\dagger}$, i.e., $S_{\text{target}} = \{ x \in \Gamma_h \mid \|x - x^{\dagger}\| \le 0.05 \}$. The background set is defined as the remainder of the interface, $S_{\text{bg}} = \Gamma_h \setminus S_{\text{target}}$. We then define
\[
CNR=\frac{| \bar{I}_{\text{target}}-\bar{I}_{\text{bg}} |}{\sigma_{\text{bg}}},
\quad \text{where}\quad
\bar{I}_{\text{target}}=\frac{1}{|S_{\text{target}}|}
\sum_{x \in S_{\text{target}}}I_{\text{comp}}(x),\quad
\bar{I}_{\text{bg}}=\frac{1}{|S_{\text{bg}}|}
\sum_{x \in S_{\text{bg}}}I_{\text{comp}}(x).
\]
Here $\sigma_{\text{bg}}$ is the sample standard deviation of the discrete background index values:
\[
\sigma_{\text{bg}}
=\sqrt{\frac{1}{|S_{\text{bg}}|-1}
\sum_{x \in S_{\text{bg}}}(I_{\text{comp}}(x)-\bar{I}_{\text{bg}})^2}.
\]

Table \ref{tab:quant_metrics} summarized the quantitative metrics for Problem I under various configurations. The results confirmed that the $H^1$ probing index provides highly precise localizations, with MLE values achieving exactly zero in most full-data scenarios. Even under the severe condition of 25\% missing boundary data (Examples 5-7), the spatial deviation remained minimal (e.g., bounded by $0.0167$, which is the size of one mesh grid), while maintaining a sufficiently high CNR to effectively distinguish the defects.

\begin{table}[htbp]
\centering
\caption{Quantitative evaluation of $H^1$ probing index for Problem I}
\label{tab:quant_metrics}
\renewcommand{\arraystretch}{1.3} 
\resizebox{\textwidth}{!}{ 
\begin{tabular}{clccl}
\toprule
\textbf{Example} & \textbf{Data Strategy / Noise} & \textbf{MLE} & \textbf{CNR} & \textbf{Peak Deviation (True Center $\rightarrow$ Reconstructed Peak)} \\
\midrule
\multirow{1}{*}{\textbf{Example 1}} & Full Data / 4\% & 0.0000 & 2.0889 & $(0.2, -0.5) \rightarrow (0.2, -0.5)$ \\
\midrule
\multirow{3}{*}{\textbf{Example 2}} & \multirow{3}{*}{Full Data / 4\%} & \multirow{3}{*}{0.0000} & \multirow{3}{*}{3.8074} 
& $(0.2, -0.767) \rightarrow (0.217, -0.767)$ \\
& & & & $(0.3, 0.733) \rightarrow (0.283, 0.733)$ \\
& & & & $(-0.767, 0.1) \rightarrow (-0.767, 0.133)$ \\
\midrule
\multirow{2}{*}{\textbf{Example 3}} & \multirow{2}{*}{Full Data / 4\%} & \multirow{2}{*}{0.0000} & \multirow{2}{*}{1.1189} 
& $(0.5, 0.5) \rightarrow (0.55, 0.5)$ \\
& & & & $(-0.5, 0.5) \rightarrow (-0.55, 0.5)$ \\
\midrule
\multirow{2}{*}{\textbf{Example 4}} & \multirow{2}{*}{Full Data / 4\%} & \multirow{2}{*}{0.0000} & \multirow{2}{*}{1.7862} 
& $(0.3, 0.5) \rightarrow (0.283, 0.5)$ \\
& & & & $(-0.5, -0.2) \rightarrow (-0.5, -0.183)$ \\
\midrule
\multirow{1}{*}{\textbf{Example 5}} & 75\% Data / 2\% & 0.0000 & 2.1926 & $(0.2, -0.5) \rightarrow (0.183, -0.5)$ \\
\midrule
\multirow{2}{*}{\textbf{Example 6}} & \multirow{2}{*}{75\% Data / 2\%} & \multirow{2}{*}{0.0167} & \multirow{2}{*}{1.2849} 
& $(0.5, 0.5) \rightarrow (0.55, 0.5)$ \\
& & & & $(-0.5, 0.5) \rightarrow (-0.583, 0.5)$ \\
\midrule
\multirow{2}{*}{\textbf{Example 7}} & \multirow{2}{*}{75\% Data / 2\%} & \multirow{2}{*}{0.0167} & \multirow{2}{*}{1.1657} 
& $(0.3, 0.5) \rightarrow (0.35, 0.5)$ \\
& & & & $(-0.5, -0.2) \rightarrow (-0.5, -0.283)$ \\
\bottomrule
\end{tabular}
}
\end{table}

\subsection{Problem II}
\begin{figure}[h]
\centering
\subfigure
{
    \begin{minipage}[b]{.11\linewidth}
        \centering
        \includegraphics[scale=0.251]{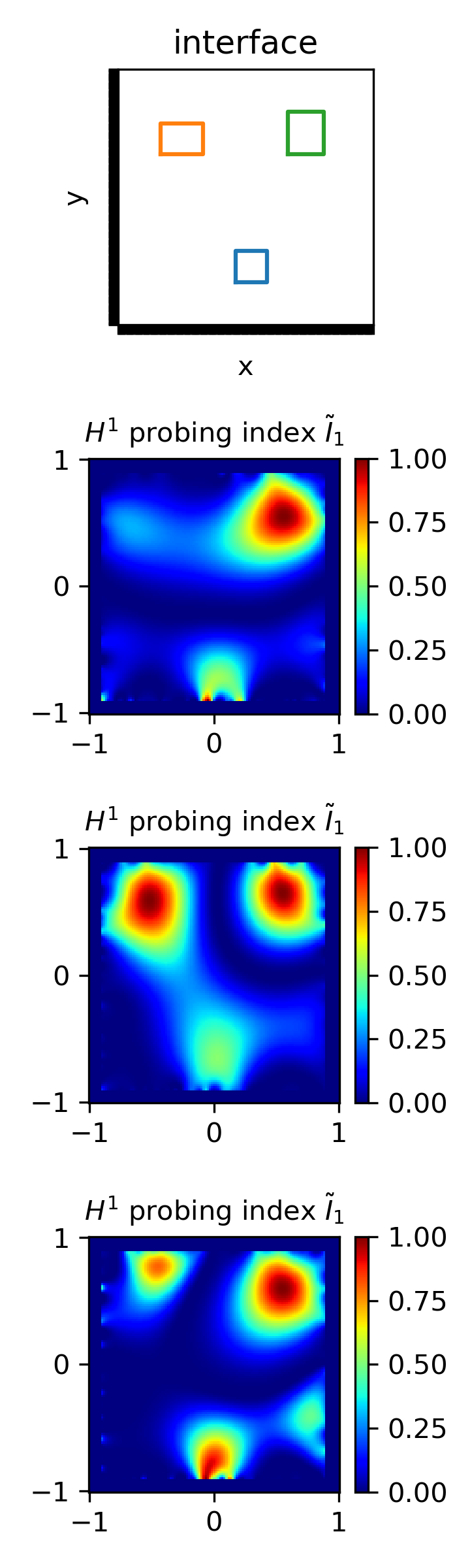}
    \end{minipage}
}
\subfigure
{
 	\begin{minipage}[b]{.11\linewidth}
        \centering
        \includegraphics[scale=0.251]{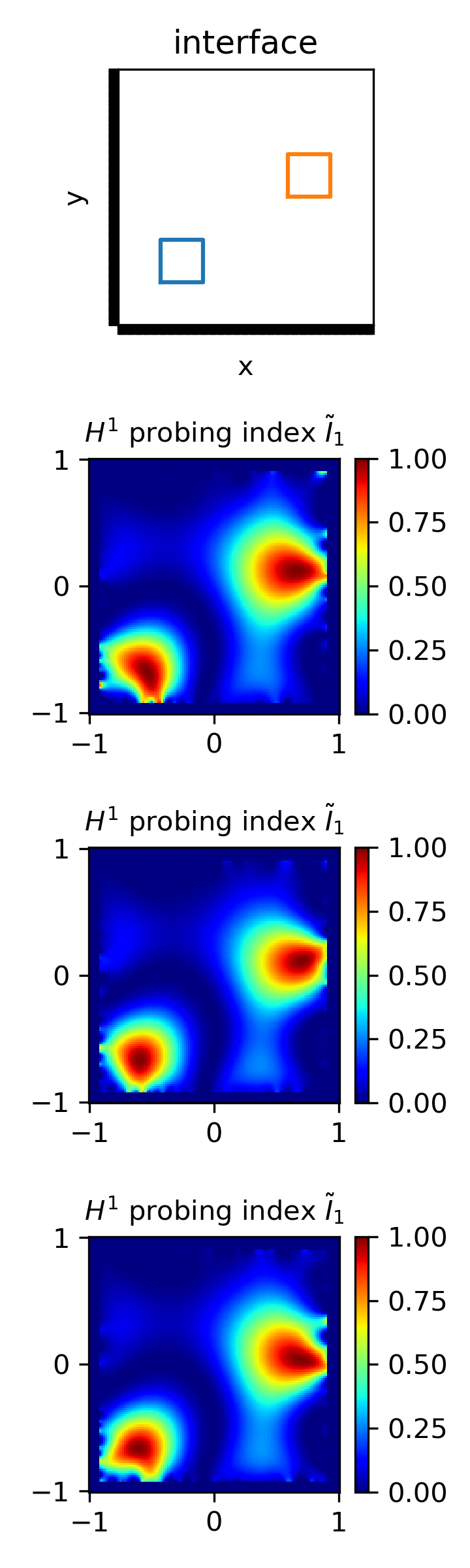}
    \end{minipage}
}
\subfigure
{
 	\begin{minipage}[b]{.11\linewidth}
        \centering
        \includegraphics[scale=0.255]{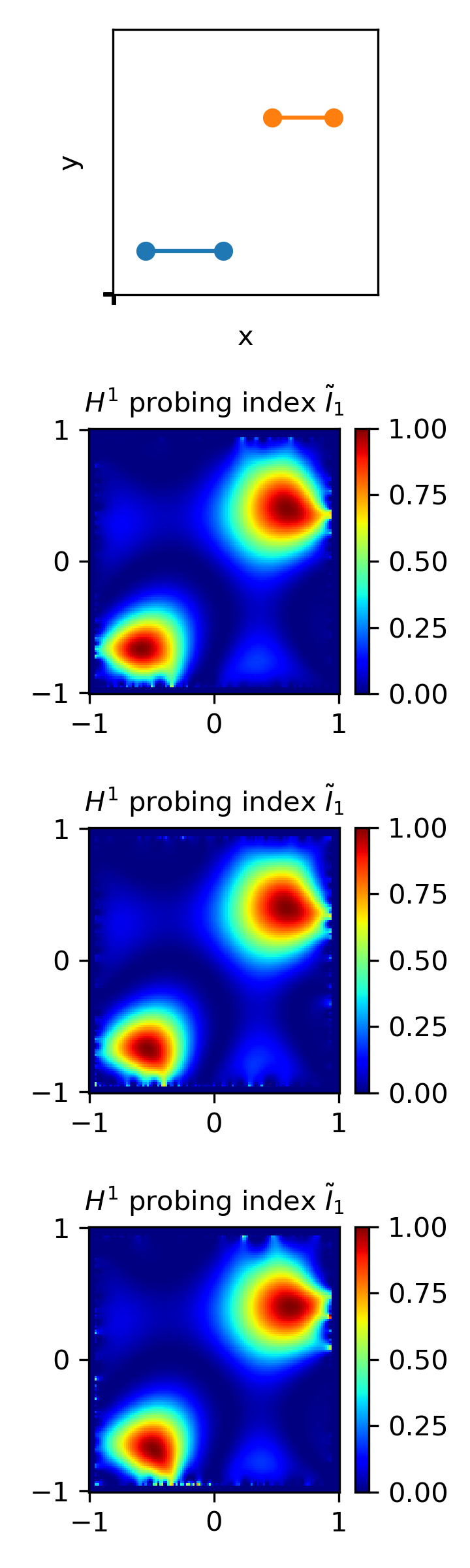}
    \end{minipage}
}
\subfigure
{
 	\begin{minipage}[b]{.11\linewidth}
        \centering
        \includegraphics[scale=0.255]{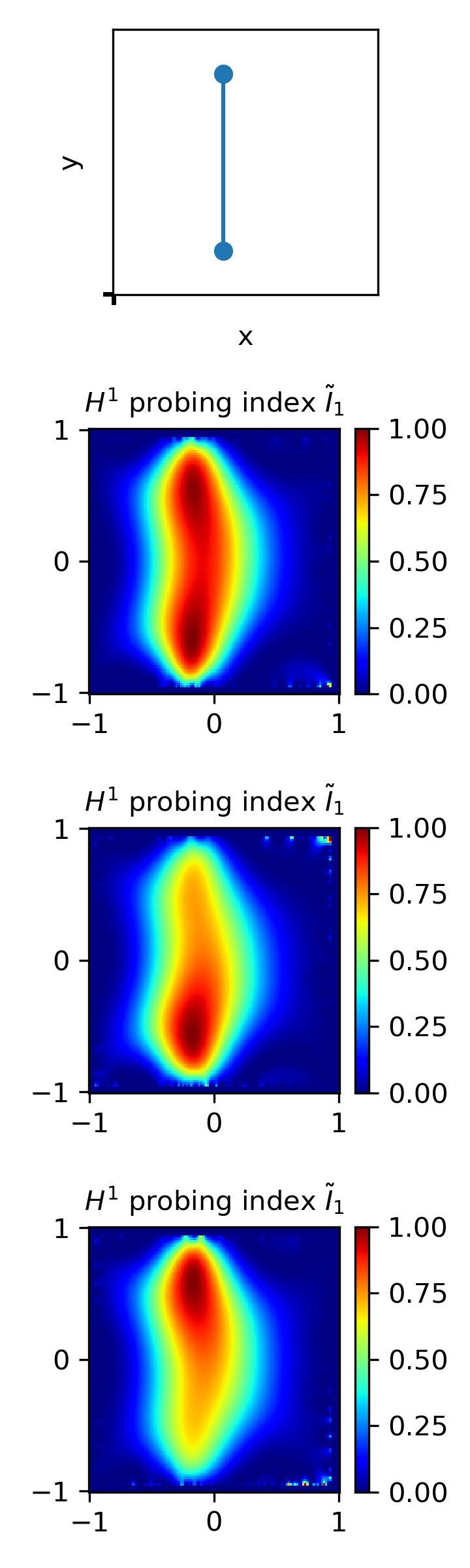}
    \end{minipage}
}
\subfigure
{
 	\begin{minipage}[b]{.11\linewidth}
        \centering
        \includegraphics[scale=0.255]{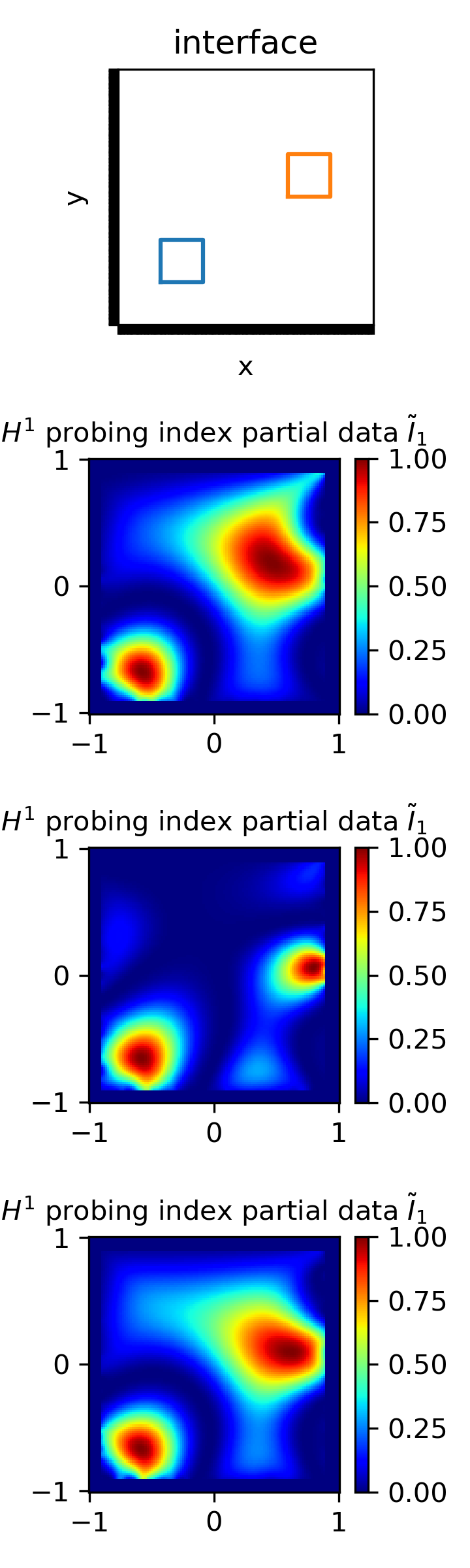}
    \end{minipage}
}
\subfigure
{
 	\begin{minipage}[b]{.11\linewidth}
        \centering
        \includegraphics[scale=0.255]{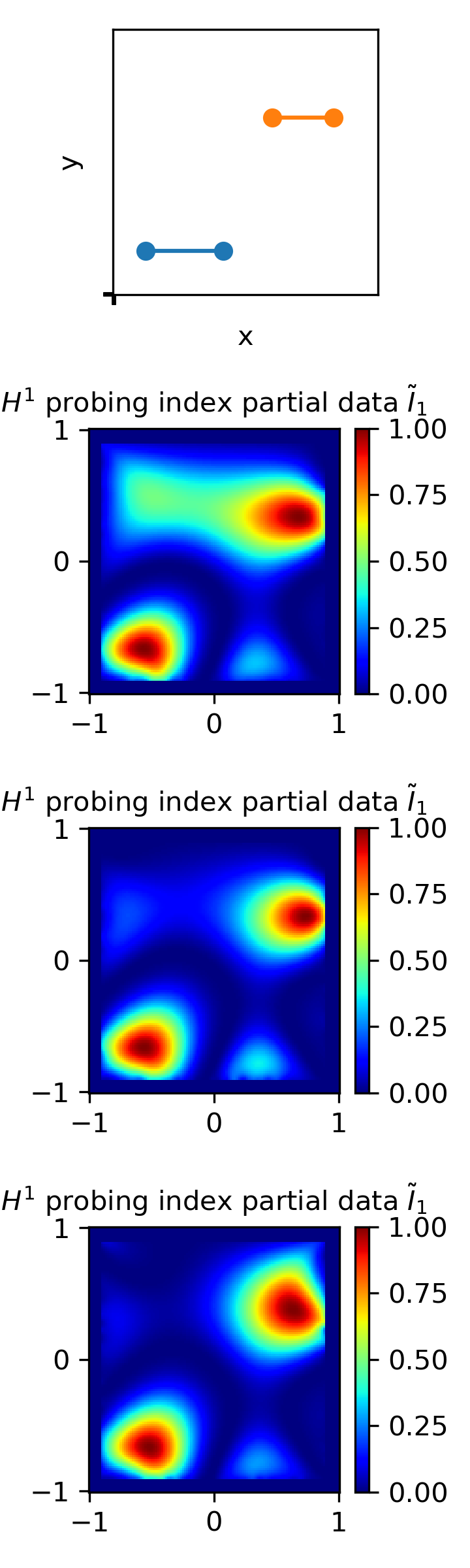}
    \end{minipage}
}
\subfigure
{
 	\begin{minipage}[b]{.11\linewidth}
        \centering
        \includegraphics[scale=0.255]{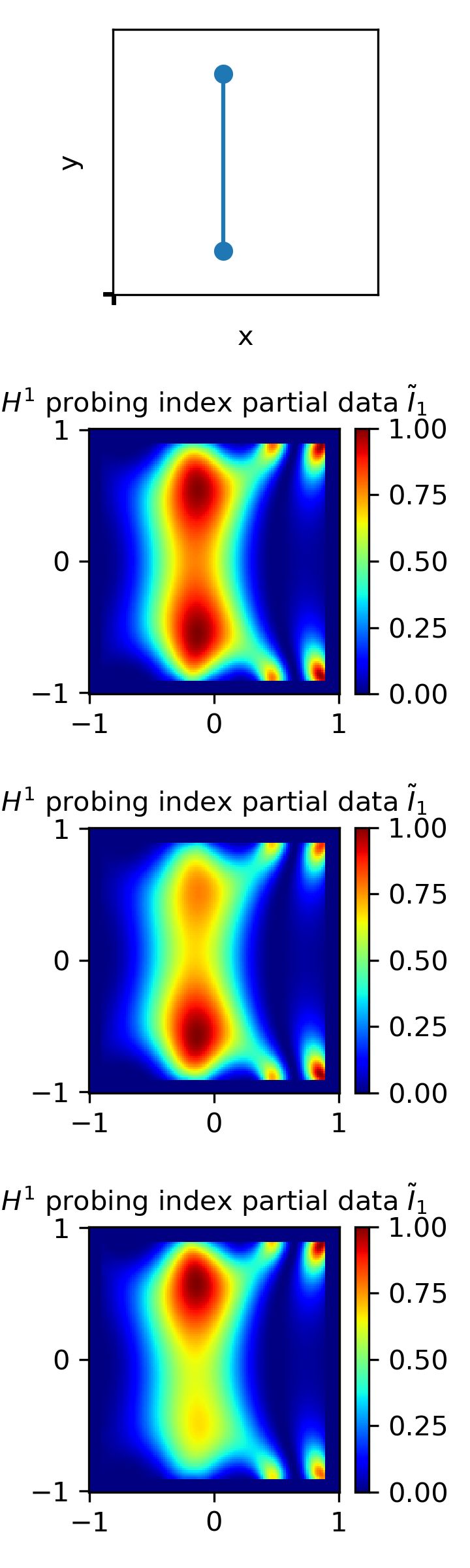}
    \end{minipage}
}
\caption{Problem II results: full data with $4\%$ noise (left 4 columns) and $75\%$ partial data with $2\%$ noise (right 3 columns). Rows correspond to different boundary fluxes $g$.}\label{problem2figure}
\end{figure}
For Problem II with full data available on the outer boundary, we present four numerical examples to demonstrate the capability of our proposed DSM for reconstructing the interface locations for four scenarios: (i) Three closed interfaces; (ii) Two closed interfaces; (iii) Two non-closed interfaces with relatively small size; (iv) One non-closed interface with relatively large size. For the partial data case ($75\%$ available), we investigated scenarios (i)(iii)(iv). We can see from Figure \ref{problem2figure} that the $H^1$ probing index function accurately captured the exact interface locations for scenarios (ii)(iii)(iv) with various boundary flux $g$. For the more complex scenario with three interfaces, the figure similarly demonstrated that for each boundary flux $g$, the $H^1$ probing index captured at least two interfaces. Combining these results yielded the desirable full identification of all jump locations.

To quantitatively evaluate Problem II, we adapted the CNR and MLE from Problem I. We utilized the same composite evaluation index $I_{\text{comp}}(x) = \sum_{i=1}^3 |\widetilde{I}^{(i)}(x)|^2$ (or $I_{\text{comp}}(x) = \sum_{i=1}^3 |\hat{I}^{(i)}(x)|^2$ for the refined case), which fuses the indices from the three distinct pairs of Cauchy data. We then extended the distance metric to 2D shapes by replacing the point-to-point distance $\|x - x_j^*\|$ with the shortest Euclidean distance to the true interface, $d(x, \Gamma^*) = \inf_{y \in \Gamma^*} \|x - y\|$. The CNR target region becomes $S_{\text{target}} = \{ x \in \Omega_h \mid d(x, \Gamma^*) \le \Delta x \}$, where the mesh size $\Delta x$ captures the numerical ridge thickness, leaving the remainder of the domain as the background $S_{\text{bg}} = \Omega_h \setminus S_{\text{target}}$.

For closed rectangular interfaces (Examples 1 and 2), the localization error for the $j$-th interface is $LE_j = d(\hat{x}_j, \Gamma_j^*)$, where the local peak $\hat{x}_j$ is identified within a $0.3$-neighborhood of the rectangle's geometric center $x_j^c$, i.e., $\hat{x}_j = \arg\max_{\substack{x \in \Omega_h \\ \|x - x_j^c\| \le 0.3}} I_{\text{comp}}(x)$. For open line segments (Examples 3 and 4), a single global peak cannot adequately capture trajectory distortions. Instead, we employed a cross-sectional ridge evaluation: the true line is uniformly divided into $10$ orthogonal bins. Within each bin, we searched for the local peak strictly within a perpendicular distance of $0.25$ from the true line. The Mean LE was then calculated by averaging the perpendicular distances of these local peaks to the true line, ensuring a rigorous assessment of the overall geometric alignment.

Table \ref{tab:quant_metrics_prob2} summarized the quantitative metrics for Problem II, corresponding to the seven configurations illustrated in Figure \ref{problem2figure}. The results confirmed that the $H^1$ probing index provides highly precise localizations for unknown interfaces. Under the full data scenario with $4\%$ noise (Examples 1-4), the spatial deviation remained exceptionally small, with several sub-interfaces achieving perfectly zero LE or being bounded tightly within $O(\Delta x)$ (where the mesh size $\Delta x \approx 0.0167$). Even under the severe condition of $25\%$ missing boundary data (Examples 5-7), the algorithm exhibited strong robustness; the maximum localization error slightly shifted but remained highly constrained, while maintaining a robust CNR (consistently above 2.0) across all scenarios to effectively distinguish the interfaces.

\begin{table}[htbp]
\centering
\caption{Quantitative evaluation of $H^1$ probing index for Problem II}
\label{tab:quant_metrics_prob2}
\renewcommand{\arraystretch}{1.3} 
\resizebox{\textwidth}{!}{ 
\begin{tabular}{clccl}
\toprule
\textbf{Example} & \textbf{Data Strategy / Noise} & \textbf{MLE} & \textbf{CNR} & \textbf{Localization Error (LE) Breakdown} \\
\midrule
\textbf{Example 1} & Full Data / 4\% & 0.0463 & 2.4378 & $LE_1=0.0667$, $LE_2=0.0722$, $LE_3=0.0000$ \\
\midrule
\textbf{Example 2} & Full Data / 4\% & 0.0083 & 2.9148 & $LE_1=0.0000$, $LE_2=0.0167$ \\
\midrule
\textbf{Example 3} & Full Data / 4\% & 0.0367 & 2.6465 & $LE_1=0.0083$, $LE_2=0.0650$ \\
\midrule
\textbf{Example 4} & Full Data / 4\% & 0.0250 & 2.1975 & $LE=0.0250$ \\
\midrule
\textbf{Example 5} & 75\% Data / 2\% & 0.0167 & 2.6423 & $LE_1=0.0000$, $LE_2=0.0333$ \\
\midrule
\textbf{Example 6} & 75\% Data / 2\% & 0.0117 & 2.6985 & $LE_1=0.0067$, $LE_2=0.0167$ \\
\midrule
\textbf{Example 7} & 75\% Data / 2\% & 0.0267 & 2.0682 & $LE=0.0267$ \\
\bottomrule
\end{tabular}
}
\end{table}

The numerical results in this paper demonstrated that the DSM delivers robust performance across a variety of problem settings, highlighting its strong versatility. Moreover, its computational cost was significantly lower than that of the iterative methods discussed in the literature. The experiments further indicated that the DSM can generate accurate initial guesses for iterative schemes, effectively addressing a major challenge faced by most iterative methods. These findings underscored the considerable potential of combining DSM with iterative approaches to enhance computational efficiency and accuracy.
\subsection{The Refinement approach}
We applied the same procedure, solving \eqref{con:model_problem} with the Neumann boundary condition \(\partial u/\partial n=g\) on \(\partial\Omega\), to obtain the noisy boundary measurement $f^{\epsilon}$ on $\partial\Omega$. We used a uniform mesh of size $181\times 181$ in the finite difference method to compute the true solution $u$. To avoid the inverse crime, we then utilized a different uniform mesh of size $121\times 121$ to execute the iterative algorithm presented in Section 4, sequentially solving for $u_{1}$, $\lambda^{1}$, and $\lambda^{2}$. The noise level was set to $15\%$ for these tests. The numerical results demonstrate that the refined probing index $\hat{I}_{0}$ defined in (\ref{index-refine-def}) outperforms the original probing index $I_{0}$. This refinement approach was applied to both Problem I and Problem II. For each example, three pairs of Cauchy data $(f, g)$ were provided, yielding three corresponding pairs of normalized probing index functions $I_0$ and $\hat{I}_0$. The graphs of $I_0$ and $\hat{I}_0$ are displayed in the first and second columns, respectively. The specific settings for these examples are detailed below.\\~\\
\textbf{Problem I}: Known interface. We set $\delta = 0.1$, $\alpha=10$.\\
\textbf{Example 1:} The same setting as in Example 2, Section 5.1.\\
\textbf{Example 2:} $\Gamma = \partial\Omega_{1}$, where $\Omega_{1} = [-0.5, 0.5]\times[-0.5, 0.5]$. 
$$\gamma(x,y) = \left\{
\begin{aligned}
100,& \quad y = 0.5, |x+0.1|\leq 0.05, \\
100,& \quad x = -0.5, |y+0.25|\leq 0.05,\\
100,& \quad x = 0.5, |y+0.25|\leq 0.05,\\
1,& \quad\text{else.}
\end{aligned}
\right.
$$
$$g_1(x,y) = 2, g_2(x,y) = 3\cos(x)-1, g_3(x,y) = |\cos(\pi(x+y))|.
$$
\textbf{Example 3:} $\Gamma = \partial\Omega_{1}$, where $\Omega_{1} = [-0.5, 0.5]\times[-0.5, 0.5]$. 
$$\gamma(x,y) = \left\{
\begin{aligned}
100,& \quad y = -0.5, |x+0.1|\leq 0.05, \\
100,& \quad x = -0.5, |y-0.2|\leq 0.05,\\
100,& \quad x = 0.5, |y-0.3|\leq 0.05,\\
1,& \quad\text{else.}
\end{aligned}
\right.
$$
$$g_1(x,y) = 2, g_2(x,y) = 3\cos(x)-1, g_3(x,y) = |\cos(\pi(x+y))|.
$$
\textbf{Problem II}: Unknown interface. \\
\textbf{Example 1:} The same interface $\Gamma$ as in Example 2, Section 5.2, $\gamma = 5$ on $\Gamma$. We set $\delta = 0.1$, $\alpha=10$.
$$g_1(x,y) = 2,\quad g_2(x,y) = 3\cos(x)-1,\quad g_3(x,y) = |\cos(\pi(x+y))|.
$$
\textbf{Example 2:} The same interface $\Gamma$ and $\gamma$ as in Example 4, Section 5.2. We set $\delta = 1$, $\alpha=1$.
$$g_1(x,y) = 2,\quad g_2(x,y) = |\sin(\pi(x+y))|,\quad g_3(x,y) = |\cos(\pi(x+y))|.
$$
\textbf{Example 3:} $\Gamma = \partial\Omega_{1}\cup\partial\Omega_{2}$, where $\Omega_{1} = [0.5, 0.75]\times[-0.75, -0.5]$, $\Omega_{2} = [0.5, 0.75]\times[0.5, 0.75]$, $\gamma=3$ on $\Gamma$. We set $\delta = 0.1$, $\alpha=10$.
$$g_1(x,y) = 2,\quad g_2(x,y) = 3\cos(x)-1,\quad g_3(x,y) = |\cos(\pi(x+y))|.
$$
\begin{figure}[h]
\centering
\subfigure
{
    \begin{minipage}[b]{.3\linewidth}
        \centering
        \includegraphics[scale=0.3]{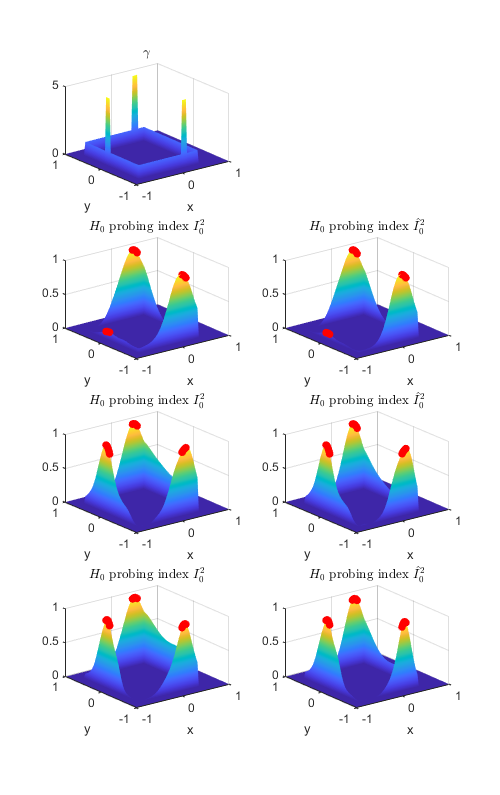}
    \end{minipage}
}
\subfigure
{
 	\begin{minipage}[b]{.3\linewidth}
        \centering
        \includegraphics[scale=0.3]{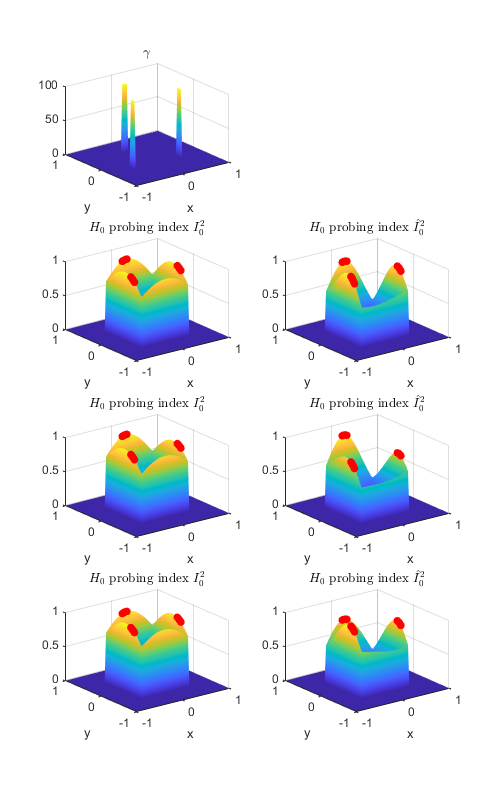}
    \end{minipage}
}
\subfigure
{
 	\begin{minipage}[b]{.3\linewidth}
        \centering
        \includegraphics[scale=0.3]{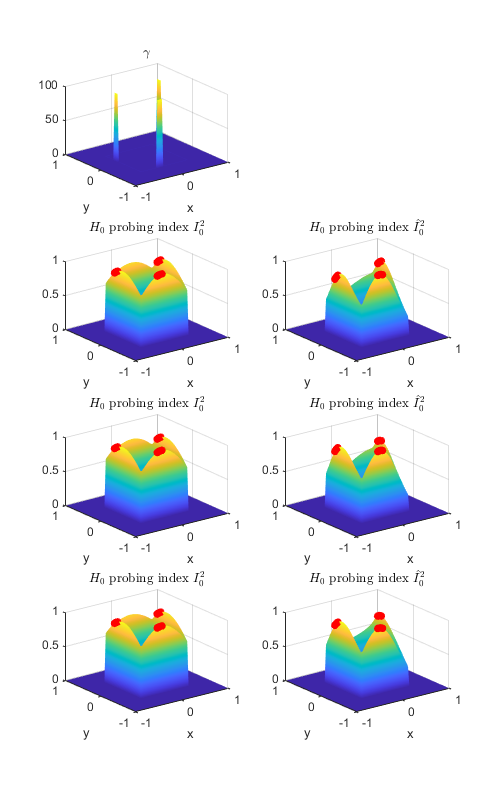}
    \end{minipage}
}
\caption{Problem I results with $15\%$ noise.}\label{refineproblem1}
\end{figure}
For Problem I with full data available on the outer boundary, we present 3 numerical examples to demonstrate the efficacy of our refinement approach. Figure \ref{refineproblem1} shows that in Example 1, both the original $H^0$ index $I_0$ and the refined index $\hat{I}_0$ successfully identified the exact jump locations (marked in red). However, in the next two examples, while $I_0$ produced a spurious jump that does not exist, the refined $\hat{I}_0$ remained accurate and correctly detected all true jump locations.

The quantitative improvements are detailed in Table \ref{tab:refinement_metrics}. The refinement approach not only consistently enhanced the CNR across all examples but also significantly reduced the MLE when dealing with complex multi-jump scenarios.

\begin{table}[h]
\centering
\caption{Quantitative evaluation of $I_0$ and $\hat{I}_0$ for Problem I.}
\label{tab:refinement_metrics}
\renewcommand{\arraystretch}{1.2}
\begin{tabular}{ccccc}
\toprule
Example & Index & MLE & CNR & True Centers $\Rightarrow$ Reconstructed Peaks \\
\midrule
\multirow{6}{*}{1} & \multirow{3}{*}{$I_0$} & \multirow{3}{*}{0.0000} & \multirow{3}{*}{1.9405} 
& $(0.2, -0.767) \Rightarrow (0.217, -0.767)$ \\
& & & & $(0.3, 0.733) \Rightarrow (0.267, 0.733)$ \\
& & & & $(-0.767, 0.1) \Rightarrow (-0.767, 0.133)$ \\
\cmidrule{2-5}
& \multirow{3}{*}{$\hat{I}_0$} & \multirow{3}{*}{0.0000} & \multirow{3}{*}{2.2723} 
& $(0.2, -0.767) \Rightarrow (0.217, -0.767)$ \\
& & & & $(0.3, 0.733) \Rightarrow (0.267, 0.733)$ \\
& & & & $(-0.767, 0.1) \Rightarrow (-0.767, 0.133)$ \\
\midrule
\multirow{6}{*}{2} & \multirow{3}{*}{$I_0$} & \multirow{3}{*}{0.0944} & \multirow{3}{*}{0.9256} 
& $(-0.1, 0.5) \Rightarrow (-0.05, 0.5)$ \\
& & & & $(-0.5, -0.25) \Rightarrow (-0.5, -0.05)$ \\
& & & & $(0.5, -0.25) \Rightarrow (0.5, -0.067)$ \\
\cmidrule{2-5}
& \multirow{3}{*}{$\hat{I}_0$} & \multirow{3}{*}{0.0389} & \multirow{3}{*}{1.3706} 
& $(-0.1, 0.5) \Rightarrow (-0.117, 0.5)$ \\
& & & & $(-0.5, -0.25) \Rightarrow (-0.5, -0.1)$ \\
& & & & $(0.5, -0.25) \Rightarrow (0.5, -0.183)$ \\
\midrule
\multirow{6}{*}{3} & \multirow{3}{*}{$I_0$} & \multirow{3}{*}{0.0944} & \multirow{3}{*}{0.9062} 
& $(-0.1, -0.5) \Rightarrow (-0.05, -0.5)$ \\
& & & & $(-0.5, 0.2) \Rightarrow (-0.5, 0.017)$ \\
& & & & $(0.5, 0.3) \Rightarrow (0.5, 0.1)$ \\
\cmidrule{2-5}
& \multirow{3}{*}{$\hat{I}_0$} & \multirow{3}{*}{0.0389} & \multirow{3}{*}{1.3196} 
& $(-0.1, -0.5) \Rightarrow (-0.117, -0.5)$ \\
& & & & $(-0.5, 0.2) \Rightarrow (-0.5, 0.05)$ \\
& & & & $(0.5, 0.3) \Rightarrow (0.5, 0.233)$ \\
\bottomrule
\end{tabular}
\end{table}

\begin{figure}[h]
\centering
\subfigure
{
 	\begin{minipage}[b]{.3\linewidth}
        \centering
        \includegraphics[scale=0.3]{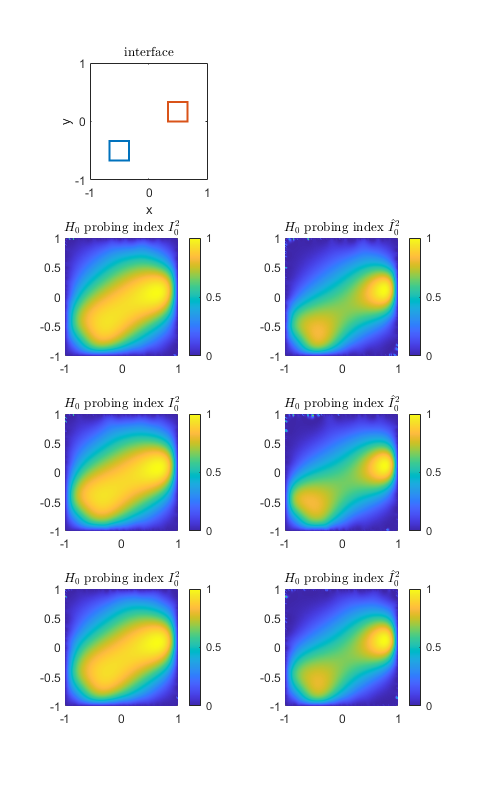}
    \end{minipage}
}
\subfigure
{
 	\begin{minipage}[b]{.3\linewidth}
        \centering
        \includegraphics[scale=0.3]{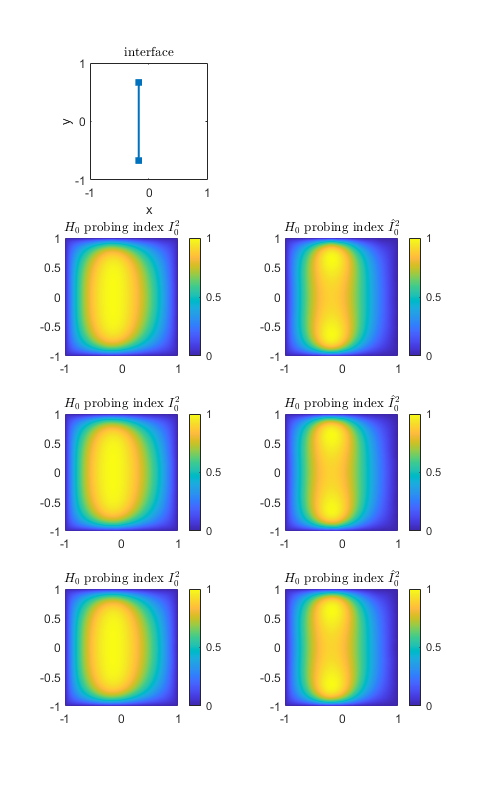}
    \end{minipage}
}
\subfigure
{
 	\begin{minipage}[b]{.3\linewidth}
        \centering
        \includegraphics[scale=0.3]{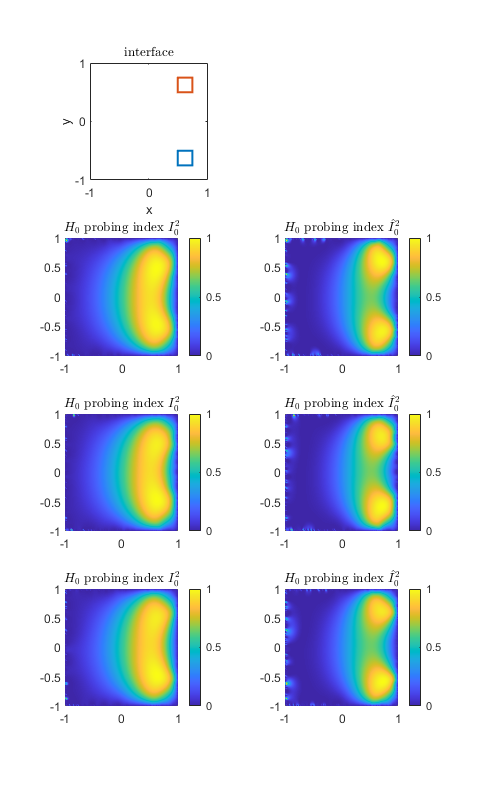}
    \end{minipage}
}
\caption{Problem II results with $15\%$ noise.}\label{refineproblem2}
\end{figure}
For Problem II with full boundary data, we present three numerical examples to validate the effectiveness of our refinement approach, covering scenarios involving multiple closed interfaces and a relatively large non-closed interface. As shown in Figure \ref{refineproblem2}, the original $H^0$ probing index $I_0$ failed to distinguish between different interfaces in the case of multiple small closed interfaces. In contrast, our refined $\hat{I}_0$ significantly improved performance by clearly resolving individual interfaces. Moreover, for the large non-closed interface case, 
$\hat{I}_0$ achieved substantially more accurate reconstruction results. 

These visual improvements were quantitatively detailed in Table \ref{tab:refine_problem2_metrics}. The refinement approach consistently enhanced the CNR across all examples, indicating a sharper index profile and stronger suppression of background artifacts. Notably, in Example 3, the refined index $\hat{I}_0$ achieved perfect localization with the MLE dropping to 0.0000, while in Example 2, the MLE was significantly reduced from 0.0167 to 0.0033. Even in cases where the overall MLE remained numerically similar (such as Example 1), the elevated CNR confirmed that the refined method yields a more robust and discernible reconstruction against high noise levels.

\begin{table}[h]
\centering
\caption{Quantitative evaluation of $I_0$ and $\hat{I}_0$ for Problem II.}
\label{tab:refine_problem2_metrics}
\renewcommand{\arraystretch}{1.2}
\begin{tabular}{ccccc}
\toprule
Example & Index & MLE & CNR & Localization Error (LE) Breakdown \\
\midrule
\multirow{2}{*}{1} & $I_0$ & 0.0420 & 1.2702 & $LE_1 = 0.0840$, $LE_2 = 0.0000$ \\
& $\hat{I}_0$ & 0.0420 & 1.5421 & $LE_1 = 0.0000$, $LE_2 = 0.0840$ \\
\midrule
\multirow{2}{*}{2} & $I_0$ & 0.0167 & 1.4180 & $LE = 0.0167$ \\
& $\hat{I}_0$ & 0.0033 & 1.4987 & $LE = 0.0033$ \\
\midrule
\multirow{2}{*}{3} & $I_0$ & 0.0333 & 1.6278 & $LE_1 = 0.0500$, $LE_2 = 0.0500$ \\
& $\hat{I}_0$ & 0.0000 & 2.1307 & $LE_1 = 0.0000$, $LE_2 = 0.0000$ \\
\bottomrule
\end{tabular}
\end{table}

The most significant advantage of our refinement approach lies in its ability to dramatically enhance the noise robustness of the DSM for both Problem I and Problem II, while maintaining the same level of accuracy.

\section{Conclusion}
In this work, we developed a novel Direct Sampling Method (DSM) to address two distinct yet highly ill-posed inverse interface problems for the Laplace equation. The key innovations of this paper lie in two aspects: (i) an optimized boundary condition selection strategy for the reference system, and (ii) a newly proposed refinement approach. Specifically, our numerical evaluations revealed that the $H^{0}$ probing index function is significantly more robust to noise than the $H^{1}$ index function, albeit at the cost of reduced localization accuracy. This crucial observation motivated the development of our iterative refinement scheme, which successfully enhances the accuracy of the $H^{0}$ index while preserving its inherent robustness. Theoretical analysis and numerical experiments demonstrate that this approach maintains the fundamental advantages of the DSM, namely, direct implementation, computational efficiency, and high precision. Remarkably, our method achieved reliable reconstructions using very limited pairs of Cauchy data, thereby bypassing the need for extensive boundary measurements.

This framework not only resolves the targeted problems but also provides a methodological reference for addressing broader classes of complex inverse interface problems. Beyond its standalone efficacy, this framework serves as a valuable pre-processing tool for iterative inversion algorithms. By providing high-quality initial guesses for unknown interfaces or flux jump distributions, our DSM-based approach can significantly accelerate convergence in gradient-based optimization methods. Looking ahead, the integration of our refinement-enhanced DSM with deep learning techniques holds promising potential for tackling more complex PDE-based inverse problems. Such hybrid methodologies could substantially reduce computational costs in real-world applications where data scarcity and noise resilience remain critical challenges.

\section{Appendix: Proofs of auxiliary results}
\subsection{Proof of Lemma \ref{lemma-wellpo}}
\label{app:lemma-wellpo-proof}
\begin{proof}
We work with the weak formulation \eqref{eq:weak-forward-interface}. Define
\[
    a(w,v)
    :=
    \int_\Omega \nabla w\cdot\nabla v\,dx
    +
    \int_\Gamma wv\,ds,
    \qquad w,v\in H^1(\Omega).
\]
The boundedness of \(a(\cdot,\cdot)\) follows from the Cauchy--Schwarz
inequality and the continuity of the trace operator
\(H^1(\Omega)\to L^2(\Gamma)\).

We first show that \(a(\cdot,\cdot)\) is coercive on \(H^1(\Omega)\). Suppose
otherwise. Then there exists a sequence \(\{w_n\}\subset H^1(\Omega)\) such
that $\|w_n\|_{H^1(\Omega)}=1$ and $a(w_n,w_n)\to0$.
Passing to a subsequence, we may assume that \(w_n\rightharpoonup w\) weakly
in \(H^1(\Omega)\). Since $\Omega$ is a bounded Lipschitz domain, by the Rellich-Kondrachov theorem, the Sobolev embedding $H^1(\Omega) \hookrightarrow L^2(\Omega)$ is compact, which implies $w_n \to w$ strongly in $L^2(\Omega)$.  Since
\[
    a(w_n,w_n)
    =
    \|\nabla w_n\|_{L^2(\Omega)}^2
    +
    \|w_n\|_{L^2(\Gamma)}^2
    \to0,
\]
we have \(\nabla w_n\to0\) strongly in \(L^2(\Omega)\). Hence
\(w_n\to w\) strongly in \(H^1(\Omega)\), and \(\nabla w=0\) in \(\Omega\).
Since \(\Omega\) is connected, \(w\) is a constant. Moreover,
\(\|w_n\|_{L^2(\Gamma)}\to0\), and the continuity of the trace operator gives
\(w|_\Gamma=0\). Since \(\Gamma\) has positive length, this constant must be
zero. Therefore \(w=0\), which contradicts
$\|w\|_{H^1(\Omega)}
    =
    \lim_{n\to\infty}\|w_n\|_{H^1(\Omega)}
    =
    1$.
Thus there exists \(C_a>0\) such that
\[
    a(w,w)\ge C_a\|w\|_{H^1(\Omega)}^2,
    \qquad \forall w\in H^1(\Omega).
\]
Consequently, \(a(\cdot,\cdot)\) defines an equivalent inner product on
\(H^1(\Omega)\). By the Riesz representation theorem, we define bounded linear operators
\(M,K:H^1(\Omega)\to H^1(\Omega)\) by
\[
    a(Mw,v)=\int_\Gamma wv\,ds,
    \qquad
    a(Kw,v)=\int_\Gamma \rho wv\,ds,
    \qquad \forall v\in H^1(\Omega).
\]
The boundary functional
\[
    F_g(v):=\int_{\partial\Omega}gv\,ds
\]
is bounded on \(H^1(\Omega)\), since \(g\in L^2(\partial\Omega)\) and the
trace operator \(H^1(\Omega)\to L^2(\partial\Omega)\) is continuous. Hence
there exists a unique \(b_g\in H^1(\Omega)\) such that
\[
    a(b_g,v)=F_g(v),
    \qquad \forall v\in H^1(\Omega).
\]
With these definitions, the weak formulation
\eqref{eq:weak-forward-interface} is equivalent to
\begin{equation}
    (I-M-\lambda K)u=b_g.
    \label{eq:fredholm-interface}
\end{equation}

We next show that \(M\) and \(K\) are compact operators on \(H^1(\Omega)\).
It suffices to prove this for \(K\), since the proof for \(M\) is identical
with \(\rho\equiv1\). Let \(\{w_n\}\) be bounded in \(H^1(\Omega)\). We first recall why the trace map from \(H^1(\Omega)\) to \(L^2(\Gamma)\)
is compact. Since \(\Gamma\) is a closed Lipschitz interface and partitions
\(\Omega\) into two Lipschitz subdomains, the standard trace theorem applied
on either side of \(\Gamma\) gives a continuous restriction map $H^1(\Omega)\longrightarrow H^{1/2}(\Gamma)$.
More precisely, for every \(w\in H^1(\Omega)\), its trace on \(\Gamma\) is
well defined and satisfies $\|w\|_{H^{1/2}(\Gamma)}
    \le C \|w\|_{H^1(\Omega)}$.
Moreover, since \(\Gamma\) is a compact Lipschitz curve, the Rellich-Kondrachov compact embedding yields $H^{1/2}(\Gamma)\hookrightarrow L^2(\Gamma)$
compactly; see, for instance, \cite[Corollary 7.2]{di2012hitchhikers}.
Consequently, the composed trace operator $H^1(\Omega)\longrightarrow L^2(\Gamma)$
is compact. Therefore, after passing to a subsequence, still denoted by \(\{w_n\}\), the
traces of \(w_n\) converge strongly in \(L^2(\Gamma)\). In particular,
\[
    \|w_n-w_m\|_{L^2(\Gamma)}\to0
    \qquad \text{as } n,m\to\infty
\]
along this subsequence. For this subsequence, using the definition of \(K\), the coercivity of
\(a(\cdot,\cdot)\), and the trace inequality, we obtain
\[
\begin{aligned}
C_a\|Kw_n-Kw_m\|_{H^1(\Omega)}^2
&\le
a(Kw_n-Kw_m,Kw_n-Kw_m)  \\
&=
\int_\Gamma \rho(w_n-w_m)(Kw_n-Kw_m)\,ds  \\
&\le
\|\rho\|_{L^\infty(\Gamma)}
\|w_n-w_m\|_{L^2(\Gamma)}
\|Kw_n-Kw_m\|_{L^2(\Gamma)}  \\
&\le
C\|w_n-w_m\|_{L^2(\Gamma)}
\|Kw_n-Kw_m\|_{H^1(\Omega)} .
\end{aligned}
\]
If \(Kw_n-Kw_m\neq0\), dividing by
\(\|Kw_n-Kw_m\|_{H^1(\Omega)}\) gives
\[
    \|Kw_n-Kw_m\|_{H^1(\Omega)}
    \le
    C\|w_n-w_m\|_{L^2(\Gamma)}.
\]
The same estimate is trivial when \(Kw_n-Kw_m=0\). Since the right-hand side
tends to zero, \(\{Kw_n\}\) is Cauchy in \(H^1(\Omega)\), and hence converges
strongly. Thus \(K\) is compact on \(H^1(\Omega)\). Similarly, \(M\) is
compact.

We now pass to the complexified space \(H^1(\Omega;\mathbb C)\), using the
same notation for the complexifications of \(M\) and \(K\). Define
\[
    \mathcal A(\lambda):=M+\lambda K,
    \qquad \lambda\in\mathbb C.
\]
Then \(\mathcal A(\lambda)\) is an analytic family of compact operators. To
apply the analytic Fredholm theorem, it remains to show that
\(I-\mathcal A(\lambda)\) is invertible for at least one value of
\(\lambda\). We choose \(\lambda=i\).

Suppose \(w\in H^1(\Omega;\mathbb C)\) satisfies $(I-M-iK)w=0$. Equivalently,
\[
    \int_\Omega \nabla w\cdot\nabla \bar v\,dx
    =
    i\int_\Gamma \rho w\bar v\,ds,
    \qquad \forall v\in H^1(\Omega;\mathbb C).
\]
Taking \(v=w\), we obtain
\[
    \int_\Omega |\nabla w|^2\,dx
    =
    i\int_\Gamma \rho |w|^2\,ds.
\]
Since \(\rho\) is real-valued, the right-hand side is purely imaginary.
Taking real parts gives
\[
    \int_\Omega |\nabla w|^2\,dx=0.
\]
Hence \(w=c\) is a constant. Substituting this back into the weak formulation
gives
\[
    ic\int_\Gamma \rho\bar v\,ds=0,
    \qquad \forall v\in H^1(\Omega;\mathbb C).
\]
Since \(\rho\not\equiv0\) in \(L^2(\Gamma)\), and the traces of
\(H^1(\Omega)\) on \(\Gamma\) are dense in \(L^2(\Gamma)\), there exists
\(v\in H^1(\Omega;\mathbb C)\) such that
\[
    \int_\Gamma \rho\bar v\,ds\ne0.
\]
Therefore \(c=0\), and hence \(w=0\). Thus \(I-\mathcal A(i)\) is injective.
Since \(\mathcal A(i)\) is compact, the Fredholm alternative implies that
\(I-\mathcal A(i)\) is boundedly invertible. By the analytic Fredholm theorem \cite[Theorem 8.26]{colton1998inverse} , \((I-\mathcal A(\lambda))^{-1}\) exists
as a bounded operator for all \(\lambda\in\mathbb C\) except possibly on a
discrete set \(\widetilde\Sigma\subset\mathbb C\) with no finite accumulation
points. Let
\[
    \Sigma:=\widetilde\Sigma\cap\mathbb R.
\]
Then \(\Sigma\subset\mathbb R\) is discrete with no finite accumulation
points. For every \(\lambda\in\mathbb R\setminus\Sigma\), the operator
\(I-M-\lambda K\) is boundedly invertible. Therefore
\eqref{eq:fredholm-interface} admits the unique solution
\[
    u=(I-M-\lambda K)^{-1}b_g
    \in H^1(\Omega).
\]
This proves the generic well-posedness of the interface problem.

\end{proof}

\subsection{Verification of the almost orthogonal property}
\label{app:almost-orthogonal-proof}
The Fourier representation of the probing kernel has been derived in \eqref{kernel_Fourier}. Here we give the expanded algebraic verification of the radial monotonicity estimates used to justify the almost orthogonal property for \(\alpha=1\). The proof uses the standard series for \(\log(1-t)\) and \(\operatorname{Li}_2(t)=\sum_{n=1}^{\infty}t^n/n^2\), and Sturm's theorem is used to confirm the signs of the resulting polynomials; see, for example, \cite[Sec.~25.12]{NIST:DLMF} and \cite[Sec.~2.2.2]{basu2006algorithms}.

When $\theta_x = \theta_z$, the sign-normalized kernel used in the localization analysis is given by
\begin{equation}
    K(x,z) = d\frac{\sum_{n=1}^{\infty}\frac{n^{2}}{n+\alpha}r_{x}^{n}r_{z}^{n}}{\left(\frac{r_{x}^{2}(1+r_{x}^{2})}{(1-r_{x}^{2})^{3}}\right)^{1/4}\left(\sum_{n\in\mathbb{Z}}\frac{n^{2}}{(|n|+\alpha)^{2}}r_{x}^{2|n|}\right)^{1/4}}.
\end{equation}
where $d>0$. This sign normalization does not affect peak locations, and $K(x,z)^4$ is proportional to the following function.

\begin{theorem}
Let \(0.005<c<1\) and \(c<x<1\), and set
\[
f(x)=
\frac{\left(\sum_{n=1}^{\infty}\frac{n^{2}}{n+1}c^{n}x^{n}\right)^{4}}
{\left(\frac{x^{2}(1+x^{2})}{(1-x^{2})^{3}}\right)
\left(\sum_{n=1}^{\infty}\frac{n^{2}}{(n+1)^{2}}x^{2n}\right)}.
\]
Define
\[
\begin{aligned}
A_1(c,x)&=\frac{7}{(1-x^2)(1-cx)}
+\frac{7}{2(1+x^2)(1+cx)}+\frac12>0,\\
A_2(c,x)&=\frac{8}{(1-cx)(1-x^2)}
+\frac{4(cx^3-1)}{(1+x^4)(1+c^2x^2)}
+5+\frac{2}{(1+cx)(1+x^2)}>0.
\end{aligned}
\]
Then
\[
(\log f(x))'
<
\begin{cases}
(c-x)A_1(c,x)<0, & 0.005<c<0.27,\ x<0.27^2/c,\\
(c-x)A_2(c,x)<0, & \text{otherwise}.
\end{cases}
\]
Consequently, \(f\) is strictly decreasing on \([c,1)\).
\end{theorem}

\begin{proof}
Let 
$$S_1(x) = \sum_{n=1}^{\infty}\frac{n^{2}}{n+1}c^{n}x^{n},\quad S_2(x) = \sum_{n=1}^{\infty}\frac{n^{2}}{(n+1)^{2}}x^{2n}.$$
Then we obtain 
$$f(x) = \frac{(1-x^2)^{3}S_1^{4}(x)}{x^{2}(1+x^{2})S_2(x)}.
$$
It suffices to show that for $c\in(0.005, 1)$,
\begin{equation}
    \left(\log f(x)\right)^{'} = 4\frac{S_1^{'}(x)}{S_1(x)}-\frac{S_2^{'}(x)}{S_2(x)}-\frac{6x}{1-x^2}-\frac{2}{x}-\frac{2x}{1+x^2}<0, \quad c<x<1.
    \label{goal}
\end{equation}
By direct computation,
$$S_1(x) = \sum_{n=1}^{\infty}\frac{n^{2}}{n+1}c^{n}x^{n} = \sum_{n=1}^{\infty}\left(n-1+\frac{1}{n+1}\right)c^{n}x^{n} = \frac{c^2 x^2}{(1-cx)^2}-\frac{\log(1-cx)}{cx}-1,
$$
$$S_1^{'}(x) = \frac{2c^2 x}{(1-cx)^{3}}+\frac{1}{x(1-cx)}+\frac{\log(1-cx)}{cx^2},
$$
$$S_2(x) = \sum_{n=1}^{\infty}x^{2n}-2\sum_{n=1}^{\infty}\frac{x^{2n}}{n+1}+\sum_{n=1}^{\infty}\frac{x^{2n}}{(n+1)^2}=\frac{1}{1-x^2}+\frac{2\log(1-x^2)}{x^2}+\frac{\text{Li}_2(x^2)}{x^2},
$$
$$S_2^{'}(x) = \frac{2x}{(1-x^2)^2}-\frac{4}{x(1-x^2)}-\frac{6\log(1-x^2)}{x^3}-\frac{2\text{Li}_2(x^2)}{x^3}.
$$
To achieve our goal, we first demonstrate (\ref{goal}) holds for $x=c$ and $c\in(0.005, 0.27)$. we prove the following inequalities (\ref{S_1_log_ine})(\ref{S_1_poly_ine})(\ref{S_2_ine}):
\begin{equation}
    \frac{\frac{1}{c(1-c^2)}+\frac{\log(1-c^2)}{c^3}}{\frac{c^4}{(1-c^2)^{2}}-\frac{\log(1-c^2)}{c^2}-1}<-\frac{7c}{8(1-c^2)}-\frac{3c}{8(1+c^2)}+\frac{1}{c}, \quad c\in(0.005, 0.27).\label{S_1_log_ine}
\end{equation}
\begin{equation}
    \frac{\frac{2c^3}{(1-c^2)^3}}{\frac{c^4}{(1-c^2)^{2}}-\frac{\log(1-c^2)}{c^2}-1}<\frac{21c}{8(1-c^2)}+\frac{5c}{4(1+c^2)}+\frac{1}{8}c, \quad c\in(0.005, 0.27).\label{S_1_poly_ine}
\end{equation}
\begin{equation}
    \frac{\frac{2c}{(1-c^2)^2}-\frac{4}{c(1-c^2)}-\frac{6\log(1-c^2)}{c^3}-\frac{2\text{Li}_2(c^2)}{c^3}}{\frac{1}{1-c^2}+\frac{2\log(1-c^2)}{c^2}+\frac{\text{Li}_2(c^2)}{c^2}}>\frac{2}{c}+\frac{c}{1-c^2}+\frac{3c}{2(1+c^2)}+\frac{c}{2}, \quad c\in(0.005, 1).\label{S_2_ine}
\end{equation}
To prove (\ref{S_1_log_ine}), note that (\ref{S_1_log_ine}) is equivalent to the following
\begin{equation}
    (-10c^8+15c^6+8c^4-21c^2+8)\log(1 - c^2)+ (  16c^8 - 17c^4+8c^2)<0
\end{equation}
Let $t = c^2$. It suffices to show that
\begin{equation}
    \log(1-t)+\frac{16t^4 - 17t^2+8t}{-10t^4+15t^3+8t^2-21t+8}<0,\quad t\in(2.5\times10^{-5}, 0.0729).\label{S_1_log_ine_1}
\end{equation}
By direct computation, the  inequality (\ref{S_1_log_ine_1}) holds for $t = 2.5\times10^{-5}$ and $t=0.0729$. Let $L_1(t)$ denote the left-hand side of (\ref{S_1_log_ine_1}). Differentiating $L_{1}(t)$ yields:
$$L_1^{'}(t)=\frac{t^2(100t^4+140t^3-79t^2-67t+4)}{(t - 1)(- 10t^3 + 5t^2 + 13t - 8)^2}
$$
Note that the denominator of $L_1^{'}(t)$ is strictly negative for $t\in(2.5\times10^{-5}, 0.0729)$. For the numerator, the part $t^2>0$, and the part $100t^4+140t^3-79t^2-67t+4$ is positive for $t=2.5\times10^{-5}$, and negative for $t=0.0729$. Additionally, the part $100t^4+140t^3-79t^2-67t+4$ is strictly decreasing on the interval $(2.5\times10^{-5}, 0.0729)$. Therefore, we conclude that $L_1(t)$ first strictly decreases from $2.5\times10^{-5}$ to some point in $(2.5\times10^{-5}, 0.0729)$, and then increases up to 0.0729. Thus, we obtain the desired $L_1(t)<0$, which proves (\ref{S_1_log_ine_1}), and consequently (\ref{S_1_log_ine}).

To prove (\ref{S_1_poly_ine}), we note that it is equivalent to the following inequality for $c\in(0.005, 0.27)$
$$(c^8-13c^6-9c^4+53c^2-32)\log(1-c^2)+(-32c^2+37c^4+7c^6-2c^8)>0.
$$
Let $t = c^2$. It suffices to show that
\begin{equation}
\log(1-t)<\frac{32t-37t^2-7t^3+2t^4}{t^4 -13t^3-9t^2+53t-32},\quad t\in(2.5\times10^{-5}, 0.0729).\label{S_1_poly_subine}    
\end{equation}
Note that $\log(1-t)<-t-t^2/2$. Thus, for $t\in(2.5\times10^{-5}, 0.0729)$, we have:
$$(-t-\frac{t^2}{2})(t^4 -13t^3-9t^2+53t-32)-(32t-37t^2-7t^3+2t^4) = -\frac{1}{2}t^3(t^3+2t^2-18t-5)>0.
$$
Therefore, the desired inequality (\ref{S_1_poly_subine}) holds, which completes the proof of(\ref{S_1_poly_ine}).

To prove (\ref{S_2_ine}), we note that it is equivalent to the following for $c\in(0.005, 1)$:
$$(20-8c^2-34c^4+20c^6 +2c^8)\log(1-c^2) + (8-2c^2 - 15c^4+8c^6+c^8)\text{Li}_2(c^2)<c^8+17c^6-2c^4-12c^2.
$$
Let $t = c^2$. Note that
$$8-2t - 15t^2+8t^3+t^4 = (t-1)(t^3 + 9t^2 - 6t - 8)>0,\quad t\in(2.5\times 10^{-5}, 1),
$$
and
$$20-8t-34t^2+20t^3 +2t^4=2(t-1)(t^3 + 11t^2 - 6t - 10)>0,\quad t\in(2.5\times 10^{-5}, 1).
$$
We have the following:
$$\text{Li}_{2}(t) = \sum_{k=1}^{\infty}\frac{t^k}{k^2}\leq t+\frac{t^2}{4}+\sum_{k=3}^{\infty}\frac{t^k}{3^2}=t+\frac{t^2}{4}+\frac{t^3}{9(1-t)},\quad \log(1-t)\leq -t-\frac{t^2}{2}-\frac{t^3}{3}-\frac{t^4}{4}.
$$
Therefore, to prove (\ref{S_2_ine}), it suffices to show for $t\in(2.5\times 10^{-5}, 1)$:
$$(20-8t-34t^2+20t^3 +2t^4)(-t-\frac{t^2}{2}-\frac{t^3}{3}-\frac{t^4}{4}) + (8-2t - 15t^2+8t^3+t^4)(t+\frac{t^2}{4}+\frac{t^3}{9(1-t)})<t^4+17t^3-2t^2-12t.
$$
This is equivalent to:
\begin{equation}
    \frac{- 18t^8 - 204t^7 + 35t^6 + 120t^5 - 51t^4 - 10t^3}{36}<0.\label{ggggg}
\end{equation}
Note that:
$$- 18t^8 - 204t^7 + 35t^6 + 120t^5 - 51t^4 - 10t^3 = -t^3(18t^5 + 204t^4 - 35t^3 - 120t^2 + 51t + 10).
$$
Let
\[
U(t)=18t^5+204t^4-35t^3-120t^2+51t+10.
\]
By Sturm's theorem, \(U\) has no zero in \((2.5\times10^{-5},1)\).
Since \(U(2.5\times10^{-5})>0\), it follows that \(U(t)>0\) on this
interval. Hence \eqref{ggggg} holds. This verifies that (\ref{S_2_ine}) holds for $c\in(0.005, 1)$. At this stage, we have proved for $c\in(0.005, 0.27)$,
$$(\log f(x))^{'}\bigg|_{x=c}<0.
$$
Next, We prove that for $c\in(0.27, 1)$, the following also holds:
$$(\log f(x))^{'}\bigg|_{x=c}<0.$$
We prove the following inequalities (\ref{S_1_ine_2})(\ref{S_2_ine_2}):
\begin{equation}
    \frac{\frac{1}{c(1-c^2)}+\frac{\log(1-c^2)}{c^3}+\frac{2c^3}{(1-c^2)^3}}{\frac{c^4}{(1-c^2)^{2}}-\frac{\log(1-c^2)}{c^2}-1}<\frac{2c}{1-c^2}+\frac{c}{2(1+c^2)}+\frac{1}{c}+\frac{5}{4}c-\frac{c}{1+c^4}, \quad c\in(0.27, 1).\label{S_1_ine_2}
\end{equation}
\begin{equation}
    \frac{\frac{2c}{(1-c^2)^2}-\frac{4}{c(1-c^2)}-\frac{6\log(1-c^2)}{c^3}-\frac{2\text{Li}_2(c^2)}{c^3}}{\frac{1}{1-c^2}+\frac{2\log(1-c^2)}{c^2}+\frac{\text{Li}_2(c^2)}{c^2}}>\frac{2}{c}+\frac{2c}{1-c^2}-\frac{4c}{1+c^4}+5c, \quad c\in( 0.27, 1).\label{S_2_ine_2}
\end{equation}
To prove (\ref{S_1_ine_2}), we note that it is equivalent to the following for $c\in( 0.27, 1)$,
\begin{equation}
    \frac{(-8-11c^2-6c^4-14c^6+2c^8+5c^{10})\log(1-c^2)+(- 8c^2 - 15c^4 - 14c^6 - 23c^8 - 10c^{10})}{4c^3(c^2 - 1)(c^2 + 1)(c^4 + 1)}<0.
\end{equation}
Note that for $c\in( 0.27, 1)$,
$$-8-11c^2-6c^4-14c^6+2c^8+5c^{10}<0.
$$
Therefore, it suffices to show that for $c\in( 0.27, 1)$,
\begin{equation}
        \log(1-c^2)+\frac{- 8c^2 - 15c^4 - 14c^6 - 23c^8 - 10c^{10}}{-8-11c^2-6c^4-14c^6+2c^8+5c^{10}}<0.
    \label{disgusting}
\end{equation}
Let $t=c^2$. and let $T_{1}(t)$ denote the left-hand side of (\ref{disgusting}). Computing the derivative of $T_{1}$ yields:
$$T_1^{'}(t) = \frac{t^2(t+1)(25t^7 + 90t^6 + 99t^5 + 120t^4 + 123t^3 + 66t^2 - 15t + 4)}{(t - 1)(- 5t^5 - 2t^4 + 14t^3 + 6t^2 + 11t + 8)^2
}<0.
$$
Therefore, the left-hand side of (\ref{disgusting}) is decreasing on $(0.27, 1)$. Since
(\ref{disgusting}) holds when $c=0.27$, the inequality holds for $c\in( 0.27, 1)$. Hence, (\ref{S_1_ine_2}) is true.

To prove (\ref{S_2_ine_2}), we note that it is equivalent to the following for $c\in(0.27,1)$:
\begin{equation}
\begin{aligned}
\frac{(-2+10c^2-12c^4+20c^6-26c^8+10c^{10})\log(1-c^2)}{c^3(c^2 - 1)^2(c^4 + 1)}\\
    +\frac{(3c^2-4c^4+8c^6-12c^8+5c^{10})\text{Li}_{2}(c^2)- 2c^2 + 7c^4 - 3c^6 + 11c^8 - 5c^{10}}{c^3(c^2 - 1)^2(c^4 + 1)}>0.
\end{aligned}
\end{equation}
By letting $t=c^2$, it suffices to show that for $t\in( 0.0729, 1)$,
$$(-2+10t-12t^2+20t^3-26t^4+10t^{5})\log(1-t)+(3t-4t^2+8t^3-12t^4+5t^{5})\text{Li}_{2}(t)- 2t + 7t^2 - 3t^3 + 11t^4 - 5t^{5}>0.
$$
Note that we have the following:
$$3t-4t^2+8t^3-12t^4+5t^{5} = t(t-1)(5t^3-7t^2+t-3)>0,\quad t\in( 0.0729, 1),
$$
and
$$\text{Li}_{2}(t) = \sum_{k=1}^{\infty}\frac{t^k}{k^2}\geq t.
$$
We only need to show that for $t\in( 0.0729, 1)$,
\begin{equation}
    (-2+10t-12t^2+20t^3-26t^4+10t^{5})\log(1-t)+(3t-4t^2+8t^3-12t^4+5t^{5})t- 2t + 7t^2 - 3t^3 + 11t^4 - 5t^{5}>0,
\end{equation}
which is equivalent to:
\begin{equation}
    2(t-1)(5t^4 - 8t^3 + 2t^2 - 4t + 1)\log(1-t)+t(5t^5 - 17t^4 + 19t^3 - 7t^2 + 10t - 2)>0,\quad t\in( 0.0729, 1).\label{stage2_last}
\end{equation}
Let $q(t)=5t^4-8t^3+2t^2-4t+1$. By Sturm's theorem,
the number of sign variations of the Sturm sequence of $q$ is $3$ at
$t=0.0729$, $2$ at $t=0.255$, and $2$ at $t=1$. Hence $q$ has exactly
one zero in $(0.0729,0.255)$ and no zero in $[0.255,1)$. Since
$q(0.0729)>0$ and $q(0.255)<0$, the condition $q(t)>0$ implies
$0.0729<t<0.255$. Therefore, when $q(t)>0$, we have
$$2(t-1)(5t^4 - 8t^3 + 2t^2 - 4t + 1)<0,\quad\log(1-t)<-t.
$$
Therefore, we have, for $5t^4 - 8t^3 + 2t^2 - 4t + 1>0$,
\begin{equation}
    \begin{aligned}
        2(t-1)(5t^4 - 8t^3 + 2t^2 - 4t + 1)\log(1-t)&+t(5t^5 - 17t^4 + 19t^3 - 7t^2 + 10t - 2)\\
        >2(t-1)(5t^4 - 8t^3 + 2t^2 - 4t + 1)(-t)+t(5t^5 - 17t^4 + 19t^3 -& 7t^2 + 10t - 2)=-t^3(5t^3 - 9t^2 + t - 5)>0.
    \end{aligned}
\end{equation}
Moreover, when $5t^4 - 8t^3 + 2t^2 - 4t + 1\leq0$, we have the following:
$$2(5t^4 - 8t^3 + 2t^2 - 4t + 1)<0,\quad\text{and}\quad(t-1)\log(1-t)=t-\sum_{k=2}^{\infty}\frac{t^k}{k(k-1)}<t-\frac{t^2}{2}.
$$
Therefore, we have, for $5t^4 - 8t^3 + 2t^2 - 4t + 1\leq0$,
\begin{equation}
    \begin{aligned}
        2(t-1)(5t^4 - 8t^3 + 2t^2 - 4t + 1)\log(1-t)&+t(5t^5 - 17t^4 + 19t^3 - 7t^2 + 10t - 2)\\
        >(5t^4 - 8t^3 + 2t^2 - 4t + 1)(2t-t^2)+t(5t^5 &- 17t^4 + 19t^3 - 7t^2 + 10t - 2)=t^2(t^2+1)(t+1)>0.
    \end{aligned}
\end{equation}
Therefore, (\ref{stage2_last}) is proved, which implies that (\ref{S_2_ine_2}) is true. Combined with (\ref{S_1_ine_2}), we have successfully proved that for $c\in(0.27, 1)$, we also have
$$(\log f(x))^{'}\bigg|_{x=c}<0.$$
At this stage, we have established that
$$(\log f(x))^{'}\bigg|_{x=c}<0
$$
holds for any $c\in(0.005,1)$.

Next, we show that for $c\in(0.005,1)$, 
$$(\log f(x))^{'}<0\quad\text{for all}\quad x>c.
$$
To this end, it suffices to show that:
\begin{equation}
    \frac{S_2^{'}(x)}{S_2(x)}>4\frac{S_1^{'}(x)}{S_1(x)}-\left(\frac{6x}{1-x^2}+\frac{2}{x}+\frac{2x}{1+x^2}\right).
\end{equation}
We show the following two inequalities (\ref{interval1_extend})(\ref{interval2_extend}).

For $0.005<c<0.27$ and $c<x<\min\{\frac{0.27^2}{c},1\}$, by (\ref{S_2_ine}), we only need to show that:
\begin{equation}
    \frac{2}{x}+\frac{x}{1-x^2}+\frac{3x}{2(1+x^2)}+\frac{x}{2}>4\frac{S_1^{'}(x)}{S_1(x)}-\left(\frac{6x}{1-x^2}+\frac{2}{x}+\frac{2x}{1+x^2}\right). \label{interval1_extend}
\end{equation}
To show (\ref{interval1_extend}), since $\sqrt{cx}\in(0.005, 0.27)$, we use (\ref{S_1_log_ine})(\ref{S_1_poly_ine}) to get
\begin{equation}
    4\frac{\frac{1}{x(1-cx)}+\frac{\log(1-cx)}{cx^2}+\frac{2c^2x}{(1-cx)^3}}{\frac{c^2x^2}{(1-cx)^{2}}-\frac{\log(1-cx)}{cx}-1}\sqrt{\frac{x}{c}}<\frac{7\sqrt{cx}}{1-cx}+\frac{7\sqrt{cx}}{2(1+cx)}+\frac{4}{\sqrt{cx}}+\frac{1}{2}\sqrt{cx}.
\end{equation}
Therefore, we have:
\begin{equation}
    4\frac{S_1^{'}(x)}{S_1(x)}<\sqrt{\frac{c}{x}}\left(\frac{7\sqrt{cx}}{1-cx}+\frac{7\sqrt{cx}}{2(1+cx)}+\frac{4}{\sqrt{cx}}+\frac{1}{2}\sqrt{cx}\right) = \frac{7c}{1-cx}+\frac{7c}{2(1+cx)}+\frac{4}{x}+\frac{1}{2}c.
\label{interval1_extend_mid}
\end{equation}
Comparing (\ref{interval1_extend_mid}) and (\ref{interval1_extend}),
it suffices to show that for $0.005<c<0.27$ and $c<x<\min\{\frac{0.27^2}{c},1\}$,
$$\frac{2}{x}+\frac{x}{1-x^2}+\frac{3x}{2(1+x^2)}+\frac{x}{2}+\frac{6x}{1-x^2}+\frac{2}{x}+\frac{2x}{1+x^2}>\frac{7c}{1-cx}+\frac{7c}{2(1+cx)}+\frac{4}{x}+\frac{1}{2}c.
$$
This is equivalent to
$$\frac{7(x-c)}{(1-x^2)(1-cx)}+\frac{7}{2}\frac{x-c}{(1+x^2)(1+cx)}+\frac{1}{2}(x-c)>0,
$$
which is clearly true. Therefore, (\ref{interval1_extend}) is proved for $0.005<c<0.27$ and $c<x<\min\{\frac{0.27^2}{c},1\}$.

For $0.27<c<x<1$, by (\ref{S_2_ine_2}), we only need to show that:
\begin{equation}
    \frac{2}{x}+\frac{2x}{1-x^2}-\frac{4x}{1+x^4}+5x>4\frac{S_1^{'}(x)}{S_1(x)}-\left(\frac{6x}{1-x^2}+\frac{2}{x}+\frac{2x}{1+x^2}\right). \label{interval2_extend}
\end{equation}
To show (\ref{interval2_extend}), since $\sqrt{cx}\in(0.27, 1)$, we use (\ref{S_1_ine_2}) to get:
\begin{equation}
        4\frac{\frac{1}{x(1-cx)}+\frac{\log(1-cx)}{cx^2}+\frac{2c^2x}{(1-cx)^3}}{\frac{c^2x^2}{(1-cx)^{2}}-\frac{\log(1-cx)}{cx}-1}\sqrt{\frac{x}{c}}<\frac{8\sqrt{cx}}{1-cx}+\frac{2\sqrt{cx}}{1+cx}+\frac{4}{\sqrt{cx}}+5\sqrt{cx}-\frac{4\sqrt{cx}}{1+c^2x^2}.
\end{equation}
Therefore, we have:
\begin{equation}
    4\frac{S_1^{'}(x)}{S_1(x)}<\sqrt{\frac{c}{x}}\left(\frac{8\sqrt{cx}}{1-cx}+\frac{2\sqrt{cx}}{1+cx}+\frac{4}{\sqrt{cx}}+5\sqrt{cx}-\frac{4\sqrt{cx}}{1+c^2x^2}\right) = \frac{8c}{1-cx}+\frac{2c}{1+cx}+\frac{4}{x}+5c-\frac{4c}{1+c^2x^2}.
\label{interval2_extend_mid}
\end{equation}
Comparing (\ref{interval2_extend_mid}) and (\ref{interval2_extend}),
it suffices to show that for $0.27<c<x<1$,
$$\frac{2}{x}+\frac{2x}{1-x^2}-\frac{4x}{1+x^4}+5x+\frac{6x}{1-x^2}+\frac{2}{x}+\frac{2x}{1+x^2}>\frac{8c}{1-cx}+\frac{2c}{1+cx}+\frac{4}{x}+5c-\frac{4c}{1+c^2x^2}
$$
This is equivalent to
$$\frac{8(x-c)}{(1-cx)(1-x^2)}+\frac{4(cx^3-1)(x-c)}{(1+x^4)(1+c^2x^2)}+5(x-c)+\frac{2(x-c)}{(1+cx)(1+x^2)}>0,
$$
which is true since $\frac{4(cx^3-1)}{(1+x^4)(1+c^2x^2)}>-4$. Therefore, (\ref{interval2_extend}) is proved for $0.27<c<x<1$.

For $0.005<c<0.27$ and $\frac{0.27^2}{c}\leq x<1$, we can obtain the desired results using the same techniques as in the previous case where $0.27<c<x<1$.

Through our previous calculation process, we can get either
$$(\log f(x))^{'} <(c-x)\left(\frac{7}{(1-x^2)(1-cx)}+\frac{7}{2}\frac{1}{(1+x^2)(1+cx)}+\frac{1}{2}\right)<0,
$$
or
$$(\log f(x))^{'} <(c-x)\left(\frac{8}{(1-cx)(1-x^2)}+\frac{4(cx^3-1)}{(1+x^4)(1+c^2x^2)}+5+\frac{2}{(1+cx)(1+x^2)}\right)<0.
$$
Therefore, for $0.005<c<1$, $f$ is strictly decreasing on $[c,1)$.
This completes the proof of Theorem 1.

\end{proof}

\begin{theorem}
Let \(\beta=10/9\), \(\mu=20/19\), and \(0.005<c<1\), and let \(f\)
be the function defined in the preceding theorem. Define
\[
\begin{aligned}
B_\beta(c,x)&=\frac{11}{(1-x^2)(1-cx/\beta)}
+\frac{5}{(1+x^2)(1+cx/\beta)}-\frac{17}{4},\\
B_\mu(c,x)&=\frac{11}{(1-cx/\mu)(1-x^2)}
+\frac{8}{(1+cx/\mu)(1+x^2)}-7.
\end{aligned}
\]
Let
\[
I_\beta(c)=
\begin{cases}
(0.0045,0.9c), & 0.005<c<0.5,\\
(0.0045,0.48^2/(c\mu)], & 0.5\leq c<1,
\end{cases}
\qquad
I_\mu(c)=\left(0.48^2/(c\mu),0.95c\right).
\]
Then
\[
(\log f(x))'>
\begin{cases}
(c/\beta-x)B_\beta(c,x)>0, & x\in I_\beta(c),\\
(c/\mu-x)B_\mu(c,x)>0, & 0.5\leq c<1,\ x\in I_\mu(c).
\end{cases}
\]
Consequently, \(f\) is strictly increasing on \((0.0045,0.9c)\) for
\(0.005<c<0.5\), and on \((0.0045,0.95c)\) for \(0.5\leq c<1\).
\end{theorem}
\begin{proof}
We use the same notation as in the proof of Theorem 1. It suffices to show that for $0.005<c<0.5$:
\begin{equation}
    (\log f(x))^{'} = 4\frac{S_1^{'}(x)}{S_1(x)}-\frac{S_2^{'}(x)}{S_2(x)}-\frac{6x}{1-x^2}-\frac{2}{x}-\frac{2x}{1+x^2}>0,\quad 0.0045<x\leq 0.9c,\label{goal1}
\end{equation}
and for $0.5\leq c<1$:
\begin{equation}
    (\log f(x))^{'} = 4\frac{S_1^{'}(x)}{S_1(x)}-\frac{S_2^{'}(x)}{S_2(x)}-\frac{6x}{1-x^2}-\frac{2}{x}-\frac{2x}{1+x^2}>0,\quad 0.0045<x\leq 0.95c.\label{goal2}
\end{equation}
To achieve this, we first demonstrate that (\ref{goal1}) holds for $x=0.9c$ and $0.005<c<0.5$, and (\ref{goal2}) holds for $x=0.95c$ and $0.5\leq c<1$.

To prove that $(\log f(x))^{\prime}|_{x=0.95c}>0$ for $0.5\le c<1$, we prove the following inequalities (\ref{goal1_S1_ine})(\ref{goal1_S2_ine}).

Let $\mu = \frac{20}{19}$, then $c=\mu x$. We show that:
\begin{equation}
    \frac{\frac{1}{x(1-\mu x^2)}+\frac{\log(1-\mu x^2)}{\mu x^3}+\frac{2\mu^2x^3}{(1-\mu x^2)^3}}{\frac{\mu^2x^4}{(1-\mu x^2)^{2}}-\frac{\log(1-\mu x^2)}{\mu x^2}-1}>\frac{11x}{4(1-x^2)}+\frac{2x}{1+x^2}+\frac{1}{x}-\frac{7}{4}x, \quad x\in[0.48/\mu, 1/\mu).\label{goal1_S1_ine}
\end{equation}
\begin{equation}
    \frac{\frac{2x}{(1-x^2)^2}-\frac{4}{x(1-x^2)}-\frac{6\log(1-x^2)}{x^3}-\frac{2\text{Li}_2(x^2)}{x^3}}{\frac{1}{1-x^2}+\frac{2\log(1-x^2)}{x^2}+\frac{\text{Li}_2(x^2)}{x^2}}<\frac{2}{x}+\frac{5x}{1-x^2}+\frac{6x}{1+x^2}-7x, \quad x\in(0.005/\mu, 1/\mu).\label{goal1_S2_ine}
\end{equation}
To prove (\ref{goal1_S1_ine}), let $y = \mu x^2$. This is equivalent to proving the following inequality:
\begin{equation}
\frac{\frac{1}{1-y}+\frac{\log(1-y)}{y}+\frac{2y^2}{(1-y)^3}}{\frac{y^2}{(1-y)^{2}}-\frac{\log(1-y)}{y}-1}>\frac{11y/\mu}{4(1-y/\mu)}+\frac{2y/\mu}{1+y/\mu}+1-\frac{7}{4}y/\mu, \quad y\in[0.48^2/\mu, 1/\mu).\label{S_1y_ine}
\end{equation}
We find the common denominator on both sides of the inequality, and it suffices to prove:
\begin{equation}
\begin{aligned}
    \frac{(- 48013y^6 + 180139y^5 - 343539y^4 + 365913y^3 - 117700y^2 - 100800y + 64000)\log(1-y)}{80y(19y - 20)(19y + 20)(y - 1)^3}\\
    +\frac{64000y- 68800y^2 - 149700y^3 + 309833y^4 - 245119y^5 + 96026y^6}{80y(19y - 20)(19y + 20)(y - 1)^3}>0
\end{aligned}
\end{equation}
Note that:
\begin{align*}
- 48013y^6 + 180139y^5 - 343539y^4 + 365913y^3 - 117700y^2 - 100800y + 64000\\
= -(y-1)^3(48013y^3 - 36100y^2 + 91200y + 64000)>0.
\end{align*}
Denoting $R(y) = - 48013y^6 + 180139y^5 - 343539y^4 + 365913y^3 - 117700y^2 - 100800y + 64000$, and $P(y) = 64000y- 68800y^2 - 149700y^3 + 309833y^4 - 245119y^5 + 96026y^6$,
we need to show that for $y\in[0.48^2/\mu, 1/\mu)$
$$\log(1-y)>-\frac{P(y)}{R(y)}.
$$
We first check that it is true for $y = 0.48^2/\mu=0.21888$. Then, taking derivatives for both sides, it suffices to show that
$$-\frac{1}{1-y}+\frac{P^{'}R-PR^{'}}{R^2}>0.
$$
A direct simplification gives
\[
-\frac{1}{1-y}+\frac{P'R-PR'}{R^2}
=
\frac{y^2 p(y)}{(1-y)^4(48013y^3-36100y^2+91200y+64000)^2},
\]
where
\[
\begin{aligned}
p(y)=&2305248169y^7-4853154040y^6+2209071271y^5
+13054323160y^4\\
&-7282525200y^3-6421088000y^2+4679680000y-563200000.
\end{aligned}
\]
By Sturm's theorem, the number of sign variations of the Sturm sequence
of $p$ is $3$ at $y=684/3125$ and also $3$ at $y=19/20$. Hence $p$
has no zero on $[684/3125,19/20]$. Since $p(684/3125)>0$, we have
$p(y)>0$ throughout this interval. Therefore, we have
$$-\frac{1}{1-y}+\frac{P^{'}R-PR^{'}}{R^2}>0.$$
Thus, (\ref{goal1_S1_ine}) is proved.

To prove (\ref{goal1_S2_ine}), it is enough to prove the following
inequality for $0<x<1$:
\begin{equation}
    \frac{(10-2x^2-20x^4+26x^6-14x^8)\log(1-x^2)+(4-9x^4+12x^6-7x^8)\text{Li}_2(x^2)+(6x^2 + 2x^4 - 9x^6 + 7x^8)}{x^3(x^2 - 1)^2(x^2 + 1)}>0.\label{whole_st2}
\end{equation}
Let $t=x^2$. Note that
$$10-2t-20t^2+26t^3-14t^4 = -2(t-1)(7t^3 - 6t^2 + 4t + 5)>0,
$$
$$4-9t^2+12t^3-7t^4 = -(t-1)(7t^3 - 5t^2 + 4t + 4)>0,
$$
$$(t-1)\log(1-t)=t-\sum_{k=2}^{\infty}\frac{t^k}{k(k-1)}<t-\frac{t^2}{2}-\frac{t^3}{6}-\frac{t^4}{12}-\frac{t^5}{20}-\frac{t^6}{30}-\frac{t^7}{42},
$$
$$\text{Li}_2(t)>t+\frac{t^2}{4}+\frac{t^3}{9}+\frac{t^4}{16}+\frac{t^5}{25}+\frac{t^6}{36}.
$$
Therefore, to prove (\ref{whole_st2}), it suffices to show:
\begin{align*}
    -2(7t^3 - 6t^2 + 4t + 5)(t-\frac{t^2}{2}-\frac{t^3}{6}-\frac{t^4}{12}-\frac{t^5}{20}-\frac{t^6}{30}-\frac{t^7}{42})\\
    +(4-9t^2+12t^3-7t^4)(t+\frac{t^2}{4}+\frac{t^3}{9}+\frac{t^4}{16}+\frac{t^5}{25}+\frac{t^6}{36})+6x^2 + 2x^4 - 9x^6 + 7x^8>0,
\end{align*}
which is equivalent to for $t\in(0.005^2/\mu^2, 1/\mu^2)$,
$$t^3(3500t^7 + 5904t^6 + 7131t^5 + 17228t^4 + 30205t^3 + 33432t^2 - 21000t + 2800
)>0.
$$
Let
\[
p(t)=3500t^7+5904t^6+7131t^5+17228t^4+30205t^3
+33432t^2-21000t+2800.
\]
By Sturm's theorem, the number of sign variations of the Sturm sequence
of $p$ is $3$ at $t=0$ and $3$ at $t=1$. Hence $p$ has no zero in
$(0,1)$. Since $p(0)=2800>0$, it follows that $p(t)>0$ for $0<t<1$. To this stage, we have proved (\ref{whole_st2}) is true, and (\ref{goal1_S2_ine}) follows. Therefore, for $0.5\leq c<1$,
$$(\log f(x))^{'}\bigg|_{x=0.95c}>0.
$$
Then we show that for $0.005<c<0.5$, 
$$(\log f(x))^{'}\bigg|_{x=0.9c}>0.
$$
Let $\beta = \frac{10}{9}$, then $c=\beta x$. Note that
\[
(0.005/\beta, 0.5/\beta)\subset(0.005/\beta, 0.48/\mu),
\]
since $0.5/\beta=0.45<0.48/\mu=0.456$. We show that:
\begin{equation}
    \frac{\frac{1}{x(1-\beta x^2)}+\frac{\log(1-\beta x^2)}{\beta x^3}+\frac{2\beta^2x^3}{(1-\beta x^2)^3}}{\frac{\beta^2x^4}{(1-\beta x^2)^{2}}-\frac{\log(1-\beta x^2)}{\beta x^2}-1}>\frac{11x}{4(1-x^2)}+\frac{5x}{4(1+x^2)}+\frac{1}{x}-\frac{17}{16}x, \quad x\in(0.005/\beta, 0.48/\mu).\label{goal2_S1_ine}
\end{equation}
\begin{equation}
    \frac{\frac{2x}{(1-x^2)^2}-\frac{4}{x(1-x^2)}-\frac{6\log(1-x^2)}{x^3}-\frac{2\text{Li}_2(x^2)}{x^3}}{\frac{1}{1-x^2}+\frac{2\log(1-x^2)}{x^2}+\frac{\text{Li}_2(x^2)}{x^2}}<\frac{2}{x}+\frac{5x}{1-x^2}+\frac{3x}{1+x^2}-\frac{17}{4}x, \quad x\in(0.005/\beta, 0.48/\mu).\label{goal2_S2_ine}
\end{equation}
To prove (\ref{goal2_S1_ine}), let $w = \beta x^2 = cx$. Then (\ref{goal2_S1_ine}) is equivalent to
\begin{equation}
    \frac{\frac{1}{1-w}+\frac{\log(1-w)}{w}+\frac{2w^2}{(1-w)^3}}{\frac{w^2}{(1-w)^{2}}-\frac{\log(1-w)}{w}-1}>\frac{11w/\beta}{4(1-w/\beta)}+\frac{5w/\beta}{4(1+w/\beta)}+1-\frac{17}{16}w/\beta, \quad w\in[0.005^2/\beta, 0.48^2\beta/\mu^2).
\end{equation}
We find the common denominator on both sides of the inequality, and it suffices to prove
\begin{equation}
\begin{aligned}
    \frac{(- 12393w^6 + 43659w^5 - 98919w^4 + 126733w^3 - 37380w^2 - 53700w + 32000)\log(1-w)}{160w(9w - 10)(9w + 10)(w - 1)^3}\\
    +\frac{32000w- 37700w^2 - 53380w^3 + 103473w^4 - 63099w^5 + 24786w^6}{160w(9w - 10)(9w + 10)(w - 1)^3}>0.
\end{aligned}
\end{equation}
Note that
\begin{align*}
    - 12393w^6 + 43659w^5 - 98919w^4 + 126733w^3 - 37380w^2 - 53700w + 32000\\ = -(w-1)^3(12393w^3 - 6480w^2 + 42300w + 32000)>0.
\end{align*}
$$(w-1)\log(1-w)<w-\frac{w^2}{2}-\frac{w^3}{6}.
$$
Therefore, it suffices to show that
\begin{align*}
    -(w-1)^2(12393w^3 - 6480w^2 + 42300w + 32000)(w-\frac{w^2}{2}-\frac{w^3}{6})\\
    +32000w- 37700w^2 - 53380w^3 + 103473w^4 - 63099w^5 + 24786w^6>0.
\end{align*}
This is equivalent to for $w\in[0.005^2/\beta, 0.48^2\beta/\mu^2) = (2.25\times 10^{-5}, 0.23104)$,
$$\frac{1}{6}w^3(12393w^5 + 5913w^4 + 48213w^3 - 47119w^2 + 15980w + 1100)>0,
$$
which holds clearly. Therefore, (\ref{goal2_S1_ine}) is true. Next, we prove (\ref{goal2_S2_ine}). We note that (\ref{goal2_S2_ine}) is equivalent to for $x\in(0.005/\beta, 0.48/\mu)$,
\begin{equation}
\begin{aligned}
\frac{(40-10x^2-54x^4+58x^6-34x^8)\log(1-x^2)+(16-x^2-23x^4+25x^6-17x^8)\text{Li}_2(x^2)}{4x^3(x^2 - 1)^2(x^2 + 1)}\\
\frac{+24x^2 + 7x^4 - 24x^6 + 17x^8}{4x^3(x^2 - 1)^2(x^2 + 1)}>0.    
\end{aligned}
 \label{goal2_S2_mid}
\end{equation}
Let $t=x^2$. Note that
$$40-10t-54t^2+58t^3-34t^4 = -2(t-1)(17t^3 - 12t^2+ 15t + 20)>0,
$$
$$16-t-23t^2+25t^3-17t^4 = -(t-1)(17t^3 - 8t^2 + 15t + 16)>0,
$$
$$(t-1)\log(1-t)<t-\frac{t^2}{2}-\frac{t^3}{6}-\frac{t^4}{12},\quad \text{Li}_2(t)>t+\frac{t^2}{4}+\frac{t^3}{9}+\frac{t^4}{16}.
$$
Therefore, it suffices to show that
\begin{align*}
    -2(17t^3 - 12t^2+ 15t + 20)(t-\frac{t^2}{2}-\frac{t^3}{6}-\frac{t^4}{12})+(16-t-23t^2+25t^3-17t^4)(t+\frac{t^2}{4}+\frac{t^3}{9}+\frac{t^4}{16})\\
    +(24t + 7t^2 - 24t^3 + 17t^4)>0.
\end{align*}
This is equivalent to show for $t\in(0.005^2/\beta^2, 0.48^2/\mu^2)=(2.025\times 10^{-5}, 0.207936)$,
$$-\frac{1}{144}t^3(153t^5 - 361t^4 - 109t^3 - 307t^2 + 76t - 28)>0,
$$
which holds clearly. Therefore, (\ref{goal2_S2_mid}) is true and we have proved (\ref{goal2_S2_ine}). To this stage, we have shown that for $c\in(0.005, 0.5)$,
$$(\log f(x))^{'}\bigg|_{x=0.9c} > 0.
$$
We now extend the endpoint inequalities to the full intervals. For
$0.5\leq c<1$, we split
\[
0.0045<x<0.95c
\]
into
\[
0.0045<x\leq \frac{0.48^2}{c\mu}
\quad\text{and}\quad
\frac{0.48^2}{c\mu}<x<\frac{c}{\mu}=0.95c.
\]
For $0.005<c<0.5$, we prove the result on
\[
0.0045<x<\frac{c}{\beta}=0.9c.
\]
Specifically, we show that
\begin{equation}
    \frac{S_2^{'}(x)}{S_2(x)}<4\frac{S_1^{'}(x)}{S_1(x)}-\left(\frac{6x}{1-x^2}+\frac{2}{x}+\frac{2x}{1+x^2}\right).
\end{equation}
We show the following two inequalities (\ref{goal2_interval1_extend})(\ref{goal2_interval2_extend}).

Let $\mu = 20/19$. For $0.5\leq c<1$ and $\frac{0.48^2}{c\mu}<x<c/\mu$, by (\ref{goal1_S2_ine}), we only need to show that
\begin{equation}
    \frac{2}{x}+\frac{5x}{1-x^2}+\frac{6x}{1+x^2}-7x<4\frac{S_1^{'}(x)}{S_1(x)}-\left(\frac{6x}{1-x^2}+\frac{2}{x}+\frac{2x}{1+x^2}\right). \label{goal2_interval1_extend}
\end{equation}
To show (\ref{goal2_interval1_extend}), since $cx\in(0.48^2/\mu, 1/\mu)$, we use (\ref{S_1y_ine}) to get
\begin{equation}
    \frac{\frac{1}{1-cx}+\frac{\log(1-cx)}{cx}+\frac{2c^2x^2}{(1-cx)^3}}{\frac{c^2x^2}{(1-cx)^{2}}-\frac{\log(1-cx)}{cx}-1}>\frac{11cx/\mu}{4(1-cx/\mu)}+\frac{2cx/\mu}{1+cx/\mu}+1-\frac{7}{4}cx/\mu.
\end{equation}
Therefore, we have
\begin{equation}
    4\frac{S_1^{'}(x)}{S_1(x)}>\frac{4}{x}\left(\frac{11cx/\mu}{4(1-cx/\mu)}+\frac{2cx/\mu}{1+cx/\mu}+1-\frac{7}{4}cx/\mu\right) = \frac{11c/\mu}{1-cx/\mu}+\frac{8c/\mu}{1+cx/\mu}+\frac{4}{x}-7c/\mu.
\label{goal2_interval1_extend_mid}
\end{equation}
Comparing (\ref{goal2_interval1_extend_mid}) and (\ref{goal2_interval1_extend}),
We only need to show that for $0.5\leq c<1$ and $\frac{0.48^2}{c\mu}<x<c/\mu$,
$$\frac{2}{x}+\frac{5x}{1-x^2}+\frac{6x}{1+x^2}-7x+\frac{6x}{1-x^2}+\frac{2}{x}+\frac{2x}{1+x^2}<\frac{11c/\mu}{1-cx/\mu}+\frac{8c/\mu}{1+cx/\mu}+\frac{4}{x}-7c/\mu.
$$
This is equivalent to
$$\frac{11(x-c/\mu)}{(1-x^2)(1-cx/\mu)}+8\frac{x-c/\mu}{(1+x^2)(1+cx/\mu)}-7(x-c/\mu)<0,
$$
which is clearly true. Therefore, (\ref{goal2_interval1_extend}) is proved for $0.5\leq c<1$ and $\frac{0.48^2}{c\mu}<x<c/\mu$.\\

For $0.005<c<0.5$ and $0.0045<x<0.9c$, by (\ref{goal2_S2_ine}), we only need to show that
\begin{equation}
    \frac{2}{x}+\frac{5x}{1-x^2}+\frac{3x}{1+x^2}-\frac{17}{4}x<4\frac{S_1^{'}(x)}{S_1(x)}-\left(\frac{6x}{1-x^2}+\frac{2}{x}+\frac{2x}{1+x^2}\right). \label{goal2_interval2_extend}
\end{equation}
Let $\beta = 10/9$. To show (\ref{goal2_interval2_extend}), since
\[
cx\in(0.005^2/\beta,0.5^2/\beta)
\subset (0.005^2/\beta,0.48^2\beta/\mu^2),
\]
we use (\ref{goal2_S1_ine}) to obtain
\begin{equation}
        \frac{\frac{1}{1-cx}+\frac{\log(1-cx)}{cx}+\frac{2c^2x^2}{(1-cx)^3}}{\frac{c^2x^2}{(1-cx)^{2}}-\frac{\log(1-cx)}{cx}-1}>\frac{11cx/\beta}{4(1-cx/\beta)}+\frac{5cx/\beta}{4(1+cx/\beta)}+1-\frac{17}{16}cx/\beta.
\end{equation}
Therefore, we have:
\begin{equation}
    4\frac{S_1^{'}(x)}{S_1(x)}>\frac{4}{x}\left(\frac{11cx/\beta}{4(1-cx/\beta)}+\frac{5cx/\beta}{4(1+cx/\beta)}+1-\frac{17}{16}cx/\beta\right) = \frac{11c/\beta}{1-cx/\beta}+\frac{5c/\beta}{1+cx/\beta}+\frac{4}{x}-\frac{17c/\beta}{4}.
\label{goal2_interval2_extend_mid}
\end{equation}
Comparing (\ref{goal2_interval2_extend_mid}) and (\ref{goal2_interval2_extend}),
We only need to show that for $0.005<c<0.5$ and $0.005/\beta<x<c/\beta$,
$$\frac{2}{x}+\frac{5x}{1-x^2}+\frac{3x}{1+x^2}-\frac{17}{4}x+\frac{6x}{1-x^2}+\frac{2}{x}+\frac{2x}{1+x^2}<\frac{11c/\beta}{1-cx/\beta}+\frac{5c/\beta}{1+cx/\beta}+\frac{4}{x}-\frac{17c/\beta}{4}.
$$
This is equivalent to
$$\frac{11(x-c/\beta)}{(1-cx/\beta)(1-x^2)}-\frac{17}{4}(x-c/\beta)+\frac{5(x-c/\beta)}{(1+cx/\beta)(1+x^2)}<0,
$$
which is clearly true. Therefore, (\ref{goal2_interval2_extend}) is proved for $0.005<c<0.5$ and $0.005/\beta<x<c/\beta$.

It remains to handle the lower subinterval for $0.5\leq c<1$. If
\[
0.0045<x\leq \frac{0.48^2}{c\mu},
\]
then
\[
x\leq \frac{0.48^2}{c\mu}\leq \frac{0.48^2}{0.5\mu}=0.43776
<0.45\leq \frac{c}{\beta}.
\]
Moreover,
\[
cx\leq \frac{0.48^2}{\mu}<\frac{0.48^2\beta}{\mu^2}.
\]
Therefore, the same argument used in the proof of
(\ref{goal2_interval2_extend}) applies and gives
\[
(\log f(x))^{'}
>
(c/\beta-x)\left(\frac{11}{(1-x^2)(1-cx/\beta)}
+\frac{5}{(1+x^2)(1+cx/\beta)}-\frac{17}{4}\right)>0.
\]

Combining the estimates above, for $0.005<c<0.5$ and
$0.0045<x<0.9c$, and also for $0.5\leq c<1$ and
$0.0045<x\leq \frac{0.48^2}{c\mu}$, we have
\[
(\log f(x))^{'}
>
(c/\beta-x)\left(\frac{11}{(1-x^2)(1-cx/\beta)}
+\frac{5}{(1+x^2)(1+cx/\beta)}-\frac{17}{4}\right)>0.
\]
For $0.5\leq c<1$ and
$\frac{0.48^2}{c\mu}<x<0.95c=c/\mu$, we have
\[
(\log f(x))^{'}
>
(c/\mu-x)\left(\frac{11}{(1-cx/\mu)(1-x^2)}
+\frac{8}{(1+cx/\mu)(1+x^2)}-7\right)>0.
\]
This completes the proof of Theorem 2.
\end{proof}


\renewcommand{\refname}{Bibliography}
\bibliographystyle{siamplain}
\bibliography{References}
\end{document}